%% file: frac_hm_arxv.tex
\documentclass[12pt]{amsart}
\usepackage[colorlinks=true,citecolor=blue,linkcolor=blue]{hyperref}
\usepackage{latexsym,amsmath,amssymb}
\usepackage{accents}

\usepackage{cleveref}

\usepackage[dvips,bottom=1.4in,right=1in,top=1in, left=1in]{geometry}

\usepackage[textsize=small]{todonotes}
\def\eps{\varepsilon}
\def\Om{\Omega}

\def\tcI{\widetilde{\cI}}
\def\tO{\widetilde{\O}}
\def\tcT{\widetilde{\cT}}

\def\C{{\mathcal C}}

\def\N{{\mathbb N}}

\def\T{{\mathcal T}}

\theoremstyle{definition}

\def\dist{{\rm dist\,}}

\def\supp{{\rm supp\,}}

\newcommand{\R}{\mathbb{R}}

\newcommand{\brac}[1]{\left (#1 \right )}
\newcommand{\abs}[1]{\left |#1 \right |}

\newcommand{\lap}{\Delta }
\newcommand{\aleq}{\precsim}
\newcommand{\ageq}{\succsim}
\newcommand{\aeq}{\approx}

\newcommand{\rlaps}[1]{(-\lap_{\Omega})^{\frac{#1}{2}}}
\newcommand{\laps}[1]{(-\lap)^{\frac{#1}{2}}}

\newcommand{\lapms}[1]{I^{#1}}

\newcommand{\vn}{{\vec{N}}}

\include{sb_macros}

\def\tH{\widetilde{H}}
\def\inv{{\rm inv}}

\def\tv{\widetilde{v}}
\def\i{{\rm i}}

\author{Harbir Antil}
\address[Harbir Antil]{Department of Mathematical Sciences and the Center for Mathematics and Artificial Intelligence (CMAI) 
George Mason University, Fairfax, VA 22030, USA. \newline
ORCiD: 0000-0002-6641-1449}
\email{hantil@gmu.edu}

\author{S\"oren Bartels}
\address[S\"oren Bartels]{Department for Applied Mathematics
Albert Ludwigs University of Freiburg 79104, Germany.  \newline 
ORCiD: 0000-0002-8084-5112}
\email{bartels@mathematik.uni-freiburg.de}

\author{Armin Schikorra}
\address[Armin Schikorra]{Department of Mathematics, University of Pittsburgh, Pittsburgh, PA 15261, USA.\newline
ORCiD: 0000-0001-9242-1782}
\email{armin@pitt.edu}

\title{Approximation of Fractional Harmonic Maps}

\thanks{HA is partially supported by NSF grants DMS-1818772 and DMS-1913004, the Air Force Office of Scientific Research 
  under Award NO: FA9550-19-1-0036, and the Department of the Navy, Naval Postgraduate School under Award NO: N00244-20-1-0005.}

\thanks{SB acknowledges support by the DFG via the Research Unit FOR 3013 {\em Vector- and tensor-valued surface PDEs}.}

\thanks{AS is supported by NSF Career DMS-2044898 and Simons foundation grant no 579261}
\begin{document}

\begin{abstract}
This paper addresses the approximation of fractional harmonic maps. Besides a 
unit-length constraint, one has to tackle the difficulty of nonlocality. We establish 
weak compactness results for critical points of the fractional Dirichlet energy on 
unit-length vector fields. We devise and analyze numerical methods for the approximation 
of various partial differential equations related to fractional harmonic maps. 
The compactness results imply the convergence of numerical approximations. 
Numerical examples on spin chain dynamics and point defects are presented
to demonstrate the effectiveness of the proposed methods. 
\end{abstract}



\maketitle

\section{Introduction}\label{sec:intro}
A fundamental problem in the calculus of variations concerns 
critical points of energy functionals subject to pointwise 
constraints. Related applications arise in ferromagnetism to 
model magnetization fields, liquid crystal theories defining
orientations of rod-like molecules, continuum mechanics for
describing inextensible rods and unshearable plates, and in
quantum mechanics for spin systems. We refer the reader to 
the articles~\cite{Alou97,BarPro06,BaFePr07,BBFP07,Alou08,KarWeb14,Bart15,HPPRSS19,KPPRS19,BoNoWa20,nochetto2017finite} 
for corresponding mathematical models with  
numerical methods and to~\cite{LenSch18,GerLen18,BLSS20} for recent analytical results.

In this article we consider the case of energies related to the 
fractional Laplace operator. Fractional operators are nonlocal 
and enable long range interactions. They enforce less smoothness
in comparison to their classical counterparts. These features 
make them attractive for applications leading to certain singularities
such as defects in the mathematical description of liquid crystals,
which are often modeled by harmonic maps. 
While some ideas from the treatment of standard, local 
differential operators can be employed to define stable numerical
schemes, new ideas are required to establish the convergence of
discrete stationary configurations.

Our starting point is a fractional Dirichlet energy 
\begin{equation}\label{eq:fHmap}
I[u] = \frac12 \int_\O |(-\Delta)^{\frac{s}{2}} u|^2 \dv{x}
\end{equation}
for an appropriate definition of the fractional Laplace operator
$(-\Delta)^{\frac{s}{2}}$ with $0 < s < 1$. Here $\O\subset \R^d$ is an open bounded 
domain with Lipschitz boundary $\partial\O$. We then consider 
stationary points for $I$ subject to a unit-length constraint, i.e., 
in the set
\[
\cA = \{v- \vec{N} \in \tH^s(\O;\R^N): |v(x)|^2 = 1 
\ \text{for a.e.} \, x\in \O\},
\]
where 
\[
	\widetilde{H}^s(\Omega;\mathbb{R}^N) 
	 = \{ f \in H^s(\mathbb{R}^d;\R^N) \, : \, f = 0 \quad \mbox{in } \mathbb{R}^d \setminus \Omega  \} ,
\]	
and $\vec{N}\in C^\infty(\R^d;\R^N)$ is a fixed vector field that defines a unit-length exterior Dirichlet condition on 
$\R^d \setminus \O$. Obviously, a homogeneous boundary condition is incompatible
with the unit-length constraint.

Stationary points for $I$ in $\cA$ are called {\em fractional
harmonic maps} and are formally characterized
by the Euler--Lagrange equations
\begin{equation}\label{eq:frac_hm}
(-\Delta)^s u = \l u \quad \mbox{in } \O, \quad |u|^2 = 1 \quad \mbox{in } \O, \quad u|_{\R^d\setminus\O} = \vec{N}, 
\end{equation}
where $\l\in L^1(\O)$ is a Lagrange multiplier related to the
pointwise unit-length constraint. The function $\l$ depends
nonlinearly on the vector field $u$, e.g., in the classical
case $s=1$ we have that $\l = |\nabla u|^2$. This critical
nonlinear dependence requires appropriate arguments to show
that accumulation points of bounded sequences of solutions
are again solutions of the nonlinear equation. Such stability results
are crucial for showing that numerical approximations converge
to fractional harmonic maps. 

A useful equivalent characterization of fractional harmonic
maps is the weak formulation 
\begin{equation}\label{eq:frac_hm_weak}
\big((-\Delta)^{\frac{s}{2}} u, (-\Delta)^{\frac{s}{2}} v \big) = 0
\end{equation}
for all $v\in \tH^s(\O;\R^N)$ satisfying the pointwise 
orthogonality relation $u \cdot v = 0$ almost everywhere in~$\O$.
We refer the reader to section~\ref{s:not} below for a specification
of the bilinear form in~\eqref{eq:frac_hm_weak}.
This characterization states that critical points are 
stable with respect to tangential perturbations. If $N=3$
then the latter equation is equivalent to the identity 
\begin{equation}\label{eq:frac_hm_cross}
\big((-\Delta)^{\frac{s}{2}}u, (-\Delta)^{\frac{s}{2}} (u\times \phi) \big) = 0
\end{equation}
for all $\phi\in C^\infty_c(\O;\R^3)$. 
Observe that for $|u| = 1$ and $v \cdot u=0$ we have $v=u \times (v \times u)$. 
We further note that the latter identity can be generalized to other target dimensions
$N\neq 3$ by considering $v = Xu$ with a skew-symmetric matrix valued mapping
$X:\O \to so(N)$ whose pointwise application to $u$ is identified with
a product $\phi \wedge u$, where $\phi(x)$ is for almost every $x\in \O$
a skew-symmetric bilinear form that is identified with a vector $\phi(x)\in \R^{N'}$;
for ease of readability we also write in this case $u\times \phi$.

It turns out that 
a limit passage in the nonlinear equation~\eqref{eq:frac_hm_cross} is possible.
In particular, in section~\ref{s:compact}, we 
shall establish that if $\{u_j\}_{j\in \mathbb{N}} \subset {\mathcal{A}}$, such that
 $u_j \rightharpoonup u$ in an appropriate fractional order Sobolev space, 
 as $j \rightarrow \infty$, then 
	\begin{equation}
		\label{eq:lim}
		\left((-\Delta)^{\frac{s}2} u_j, (-\Delta)^{\frac{s}2} (u_j \times \phi) \right) \rightarrow \left((-\Delta)^{\frac{s}2} u, (-\Delta)^{\frac{s}2} (u \times \phi)\right)
	\end{equation}
for all $\phi \in C^\infty_c(\O;\R^{N'})$. 
The key challenge here is the fact that due to nonlocality of $(-\Delta)^s$, the standard 
arguments from the classical case of $s=1$ cannot be applied. Our proof uses a localization 
argument combined with properties of the Hardy-Littlewood maximal function. Simpler arguments
lead to this result when the fractional Laplace operator is defined via a Fourier
transformation. We use \eqref{eq:lim} 
to carry out critical limit passages in the justification of three numerical problems related 
to fractional harmonic maps. We refer the reader to the pioneering work \cite{DLR2011}
and to the contributions~\cite{M11,R18,S15,MillotSire15,MS18,MP20,MillotPegonSchikorra20}
for various properties of minimizing, stationary, and critical fractional harmonic maps. 

\subsection*{Discrete fractional harmonic maps}

The first numerical problem concerns the convergence of discrete
fractional harmonic maps as the mesh-sizes of underlying
triangulations tend to zero. We consider a sequence
$\{\cT_h\}_{h>0}$ of uniformly shape regular triangulations
of the polygonal or polyhedral Lipschitz domain $\O \subset \R^d$
with maximal mesh-sizes $h\to 0$.
Discrete fractional harmonic
maps belong to the discrete admissible set 
\[
\cA_h = \{v_h - \tcI_h \vec{N} \in \cS^1_0(\cT_h)^N: |v_h(z)|^2 = 1 
\text{ for all } z\in \cN_h\},
\]
where $\cS^1_0(\cT_h)$ is the space
of piecewise linear, globally continuous functions for a
triangulation $\cT_h$ of $\widetilde\O$ vanishing in the exterior 
$\widetilde\O \setminus \O$; the set 
$\cN_h$ contains the vertices of elements inside $\O$ at which the 
unit-length constraint is imposed, $\cI_h$ and $\tcI_h$ are the 
nodal interpolation operators on $\cT_h$ and an extension 
$\widetilde{\cT}_h$ which provides a triangulation of a domain $\tO$ 
such that $\overline{\O}\subset \tO$
and the support of $\vec{N}$ is contained in $\tO$.

We then define discrete fractional harmonic maps as vector fields
$u_h\in \cA_h$ with the property 
\begin{equation}\label{eq:frac_hm_discr}
\big((-\Delta)^{\frac{s}{2}} u_h, (-\Delta)^{\frac{s}{2}} v_h \big) = 0
\end{equation}
for all $v_h \in \cF_h[u_h]$, where $\cF_h[u_h]$ is defined 
as
\[
\cF_h[u_h] = \left\{v_h\in \cS^1_0(\cT_h)^N: v_h(z) \cdot 
u_h(z) = 0 \text{ for all } z\in \cN_h \right\}. 
\]
If $N=3$ then the vector fields $v_h\in \cF_h[u_h]$ are
represented by 
\[
v_h = \cI_h[u_h \times \phi]
\]
for $\phi\in C^\infty_c(\O;\R^3)$. 
In section~\ref{s:fhm}, we will show that if $\{u_h\}_{h>0}$ is a bounded sequence
of discrete fractional harmonic maps then every weak limit 
$u \in \tH^s(\O;\R^N)$ as $h\to 0$ is a fractional harmonic map. 
Compact perturbations $R_h$ that model solution errors or consistency terms
can be included on the right-hand side of~\eqref{eq:frac_hm_discr} 
and incorporated in the analysis provided that $\|R_h\|_{\tH^s(\O;\R^N)'} \to 0$
as $h\to 0$. For ease of presentation, 
we assume an exact discretization of the bilinear form associated with
the fractional Laplace operator. We refer
the reader to~\cite{GAcosta_FMBersetche_JPBorthagaray_2017a,GAcosta_JPBorthagaray_NHeuer_2019a,HAntil_RKhatri_MWarma_2019a}
for corresponding results in the context of the linear fractional
Poisson problem. In our experiments we follow~\cite{GAcosta_FMBersetche_JPBorthagaray_2017a}
and~\cite{AntBar17} for finite element and spectral method implementations, 
respectively. We also refer to \cite{bonito2019numerical,HAntil_PDondl_LStriet_2020a} 
for other efficient approaches to implement integral fractional Laplacian. Some other applications of fractional
operators include, imaging \cite{AntBar17}, geophysics \cite{CJWeiss_BGVBWaanders_HAntil_2020a},
and optimal control \cite{HAntil_MWarma_2020a}. 

\subsection*{Fractional harmonic map heat flow}

The second application addresses a parabolic evolution defined by
the $L^2$-gradient flow for $I$ given in \eqref{eq:fHmap}; 
it was studied analytically in \cite{S88,PG11,SSW17}. Its discretization 
or discretizations of gradient flows for other metrics define  
fully practical methods to determine discrete fractional harmonic
maps. The $L^2$-flow of fractional harmonic maps is formally given
by the partial differential equation  
\[
\p_t u = - (-\Delta)^s u + \l u, \quad |u|^2 = 1,
\]
where again $\l$ is the Lagrange multiplier subject to the unit-length
constraint. Rigorously, we define solutions of the fractional
harmonic map heat flow as maps $u:(0,T)\times \O \to \R^N$ with 
\[
u-\vec{N} \in H^1(0,T;L^2(\O;\R^N)) \cap L^\infty(0,T;\tH^s(\O;\R^N))
\]
that satisfy $u(0)= u_0$ for a given vector field $u_0\in \cA$,
the constraint $|u(t,x)|^2=1$ almost everywhere in $(0,T)\times \O$,
and with the inner product $(\cdot,\cdot)$ in $L^2(\O;\R^N)$
\begin{equation}\label{eq:frac_hm_heat_flow}
(\p_t u,v) + \big((-\Delta)^{\frac{s}{2}} u, (-\Delta)^{\frac{s}{2}} v\big) = 0
\end{equation}
for all vector fields $v\in \tH^s(\O;\R^N)$ and almost every $t\in (0,T)$
with the orthogonality relation
\[
u(t,x) \cdot v(x) = 0
\]
for almost every $(t,x)\in (0,T)\times \O$, we furthermore require solutions
to satisfy an energy-decay property. Our numerical scheme
adopts ideas from~\cite{Alou97,Bart15-book,Bart16} and
imposes the orthogonality condition at the nodes of a triangulation
in an explicit way while the evolution equation is discretized
implicitly with the backward difference quotient operator 
\[
d_t u^k = \tau^{-1} (u^k - u^{k-1})
\]
for a step size $\tau>0$. We hence compute a sequence 
\[
\{u_h^k\}_{k=0,\dots,K} \in \vec{N} + \cS^1_0(\cT_h)^N
\]
such that $u_h^0 = u_{0,h}$ and $d_t u_h^k \in \cF_h[u_h^{k-1}]$  is
for $k=1,2,\dots,K$ such that 
\begin{equation}\label{eq:frac_hm_flow_discr}
(d_t u_h^k , v_h) + \big((-\Delta)^{\frac{s}{2}} u_h^k, (-\Delta)^{\frac{s}{2}} v_h \big) = 0
\end{equation}
for all $v_h \in \cF_h[u_h^{k-1}]$, i.e., for all $v_h\in \cS^1_0(\cT_h)^N$ with 
\[
u_h^{k-1}(z) \cdot v_h(z) =0
\]
for all $z\in \cN_h$. Note that the problems in the time steps are 
linear systems with unique solutions and testing with $v_h = d_t u_h^k$
shows the energy monotonicity
\begin{equation}\label{eq:ener_mon}
\|d_t u_h^k\|^2 + \frac{d_t}{2}\|(-\Delta)^{\frac{s}{2}} u_h^k\|^2 
+ \frac{\tau}{2} \|d_t u_h^k\|^2 = 0.
\end{equation}
The linearized, explicit treatment 
of the constraint leads to a violation that is controlled by the step-size
$\tau>0$, i.e., since $d_t u_h^k(z) \cdot u_h^{k-1}(z)= 0$ and
$u_h^k = u_h^{k-1} + \tau d_t u_h^k$ we have
\[
|u_h^k(z)|^2 = |u_h^{k-1}(z)|^2 + \tau^2 |d_t u_h^k(z)|^2 
= \dots = |u_h^0(z)|^2 + \tau^2 \sum_{\ell=1}^k |d_t u_h^k(z)|^2.
\]
By carrying out a discrete integration, i.e., multiplying by local volumes~$\b_z$
and summing over $z\in \cN_h$, and noting $|u_h^0(z)|^2=1$, we find that
\[
\big\||u_h^k|^2 -1\big\|_{L^1_h(\O)} 
\le \tau^2 \sum_{\ell=1}^k \|d_t u_h^k\|_{L^2_h(\O)}^2.
\]
The right-hand side is of order $\cO(\tau)$ owing to~\eqref{eq:ener_mon},
and we have 
\[
\|v\|_{L^p_h(\O)}  = \int_\O \cI_h |v|^p \dv{x} = \sum_{z\in \cN_h} \b_z |v(z)|^p, \quad \b_z = \int_\O \vphi_z\dv{x},
\]
with the nodal basis functions $\{\vphi_z\}_{z\in \cN_h}$.

\subsection*{Hyperbolic system for spin dynamics}

The third numerical problem is a hyperbolic evolution equation
determined by the force balance
\begin{equation}\label{eq:spin_system}
\p_t u = \d I[u] \times u = (-\Delta)^s u \times u, \quad |u|^2 =1, 
\end{equation}
which has been used to model nonlocal effects in spin chains, cf.~\cite{ZhoSto15,GerLen18,LenSch18}.
This evolution is constraint and energy preserving which follows
directly from testing the equation with $u$ and $(-\Delta)^s u$, 
respectively. To obtain these features for a discretization, we
follow~\cite{KarWeb14,Bart15} and 
use Crank-Nicolson type midpoint approximations, i.e., we consider
the time stepping scheme
\[
d_t u^k = - u^{k-1/2} \times (-\Delta)^s u^{k-1/2},
\]
with the average 
\[
u^{k-1/2} = \frac12 (u^{k-1} + u^k).
\]
A binomial formula then implies discrete energy and constraint
preservation, e.g., testing with $u^{k-1/2}$ implies that
\[
d_t |u^k|^2 = d_t u^k \cdot u^{k-1/2} = 0,
\]
so that $|u^k|^2 = |u^{k-1}|^2 = \dots = |u^0|^2$ almost everywhere
in $\O$. A spatial discretization uses quadrature to allow for
a localization of the preservation properties, i.e., 
\begin{equation}\label{eq:spin_system_discr}
(d_t u_h^k,v_h)_h = \big( (-\Delta)^s_h u_h^{k-1/2} , \cI_h[ u_h^{k-1/2} \times v_h] \big)_h
\end{equation}
where the discrete inner product is consistent with the norm $\|\cdot\|_{L^2_h(\O)}$ 
and given by
\[
(y_h,v_h)_h = \int_\O \cI_h[y_h \cdot v_h] \dv{x} 
= \sum_{z\in \cN_h} \b_z y_h(z) \cdot v_h(z),
\]
and $y_h = (-\Delta)_h^s w_h$ is for given $w_h\in \cS^1_0(\cT_h)^3$ 
the uniquely defined function $y_h\in \cS^1_0(\cT_h)^3$  with 
\[
(y_h,v_h)_h = \big((-\Delta)^{\frac{s}{2}} w_h, (-\Delta)^{\frac{s}{2}} v_h\big)
\]
for all $v_h\in \cS^1_0(\cT_h)^3$. Choosing the test function
$v_h = u_h^{k-1/2}(z) \vphi_z$ with the hat function $\vphi_z \in \cS^1_0(\cT_h)$ 
associated with a node $z\in \cN_h$ in~\eqref{eq:spin_system_discr}
then leads to the discrete constraint preservation property
\[
\b_z d_t |u_h^k(z)|^2 = 0.
\]
Analogously, by choosing $v_h = (-\Delta)_h^s u_h^{k-1/2}$ we find with the
definition of $(-\Delta)_h^s$ that
\[\begin{split}
d_t \frac12 \big\|(-\Delta)^{\frac{s}{2}} u_h^k\big\|_h^2 
&= \big((-\Delta)^{\frac{s}{2}} d_t u_h^k, (-\Delta)^{\frac{s}{2}} u_h^{k-1/2}\big)  \\
&= \big(d_t u_h^k, (-\Delta)_h^s u_h^{k-1/2}\big)_h = 0,
\end{split}\]
i.e., the preservation of the discrete fractional Dirichlet energy. 
The scheme~\eqref{eq:spin_system_discr} requires the iterative solution of a nonlinear
system of equations in every time step. A simple fixed-point iteration
is constraint preserving and
convergent provided that the step-size condition $\tau= O(h^{2s})$ is satisfied. 

\subsection*{Outline}
The remainder of the paper is organized as follows: In section~\ref{s:not}, we introduce 
definitions and notation related to fractional Sobolev spaces. Section~\ref{s:compact} is 
devoted to the weak compactness result for fractional harmonic maps defined via
the integral representation of the fractional Laplace operator;  
a corresponding result for a spectral version of the fractional Laplacian is provided in 
Appendix~\ref{app:spec}. Required finite element spaces are defined  
in section~\ref{s:fem} along with a general application of the the continuous compactness 
result to the discrete setting. Section~\ref{s:nconv} specifies the three numerical algorithms 
described above and provides corresponding stability and convergence results. 
In section~\ref{s:numexp}, we provide numerical experiments which
illustrate the good approximation properties of our numerical schemes.  

\section{Fractional Sobolev spaces}\label{s:not}

Without any specific mention, we use 
$(\cdot,\cdot)$ to denote the $L^2$-scalar product and $\|\cdot\|$ the $L^2$-norm.
The scalar product is typically defined over $\O$ or $\R^d$ with an appropriate interpretation
in case of the fractional Laplacian. For a Banach space $X$, 
we denote its topological dual by $X'$ and the pairing between 
$X'$ and $X$ by $\langle \cdot,\cdot \rangle_{X',X}$. Moreover, we use  
$\rightarrow $ and $\rightharpoonup$ to indicate strong and weak convergence, respectively. 
We occasionally denote a relation $A \le C B$, with $C$ being a non-essential constant, by $A \aleq B$. 
The set $B_\rho$ denotes the open ball of radius $\rho$ centered at $0$.
  
To define the fractional Laplace operator we consider the weighted Lebesgue space
\[
\mathbb{L}^1_s(\R^d) = 
\Big\{ f : \R^d \rightarrow \R \mbox{ measurable }, \int_{\R^d} \frac{|f(x)|}{(1+|x|)^{d+2s}} \dv{x} < \infty \Big\} 
\]
and first define for $f \in \mathbb{L}^1_s(\R^d)$, $\varepsilon > 0$, and $x\in \R^d$ the quantity
\[
(-\Delta)^s_\varepsilon f (x) = C_{d,s} \int_{\{ y \in \R^d , |y-x| > \varepsilon \}} \frac{f(x)-f(y)}{|x-y|^{d+2s}} \dv{y} 
\]
where the constant $C_{d,s} = \big(s2^{2s}\Gamma\left(\frac{2s+d}{2}\right)\big)/\big(\pi^{\frac{d}{2}}\Gamma(1-s)\big)$	
is obtained with Euler's Gamma function. 
We then define the  {\em integral version} of the fractional Laplace operator for $s \in (0,1)$
via a limit passage for $\veps \to 0$, i.e., 
\begin{equation}\label{eq:integ}
(-\Delta)^s f(x) = C_{d,s} \mbox{P.V. } \int_{\mathbb{R}^d}  \frac{f(x)-f(y)}{|x-y|^{d+2s}} \dv{y} 
	       = \, \lim_{\varepsilon \to 0} (-\Delta)^s_{\varepsilon} f(x) ,
\end{equation}
where P.V. indicates the Cauchy principal value. Note that this definition for the full space $\R^d$ 
coincides with the spectral definition of the fractional Laplacian obtained using Fourier transform 
\cite[Proposition~3.4]{MR2944369}, see also \cite{MR2354493}. Such an equivalence 
also holds in case of periodic boundary conditions \cite[Eq.~(2.53)]{MR3967804}. 

\begin{remark}
If we replace the integration domain $\R^d$ in \eqref{eq:integ} by an open set $\Om$ we obtain
the so-called \emph{regional fractional Laplacian}. All arguments given in this paper can be adapted
to that setting 
by minor modifications provided the boundary conditions are meaningful, e.g., in the case $s > 1/2$.
\end{remark}

Based on the definition of the operator $(-\Delta)^s$ we introduce
fractional order Sobolev spaces $H^{s}(\R^d)$ for $s \in (0,1)$ by setting 
\begin{align*}
[f]_{H^{s}(\R^d)}   &= \| (-\Delta)^{\frac{s}{2}} f\|_{L^2(\R^d)} = \left( \int_{\R^d} \int_{\R^d} \frac{|f(x)-f(y)|^2}{|x-y|^{d+2s}} \dv{y} \dv{x} \right)^{\frac12},\\
\|f\|_{H^{s}(\R^d)} &= \|f\|_{L^2(\R^d)} + [f]_{H^{s,2}(\R^d)} .
\end{align*}
Then the Sobolev space $H^{s}(\R^d)$ is defined as 
\[
H^{s}(\R^d) = \left\{  f \in L^2(\R^d) \, : \, \|f\|_{H^{s}(\R^d)} < +\infty \right\}
\]
which is a Hilbert space. The set of vectorial functions $f : \R^d \rightarrow \R^N$ whose components
belong to $H^{s}(\R^d)$ is denoted by $H^{s}(\R^d;\R^N)$. We will omit dependence on 
$N$ while writing corresponding norms when it is clear from the context. 

For bounded open sets $\O\subset \R^d$ and parameters $s \in (0,1)$ we define
Sobolev spaces $\widetilde{H}^{s}(\Om)$ by considering trivial extensions to $\R^d$, i.e.,
we set  
\[ 
\widetilde{H}^{s}(\Om) = \{f \in L^2(\R^d) \, : \, (-\Delta)^{\frac{s}{2}} f \in L^2(\R^d) , 
\quad f \equiv 0 \mbox{ in } \R^d \setminus \Om  \} .
\]
We recall the following density result for $s \in (0,1]$ and $\Om \subset \R^d$ for domains with 
Lipschitz boundary \cite{MR3310082}
\[
\widetilde{H}^{s}(\Om) = \overline{\mathcal{D}(\Om)}^{\|\cdot\|_{\widetilde{H}^{s}(\Om)}}
\]
and by a Poincar\'e type inequality, which is a consequence of H\"older's inequality and Sobolev imbedding theorem, \cite[Theorem 3.1.4.]{AH96}, a norm on $\widetilde{H}^s(\Omega;\R^N)$ is given by
\[
\|f\|_{\widetilde{H}^s(\Omega)} = \| (-\Delta)^{\frac{s}{2}} f\|_{L^2(\R^d)} .
\]
We refer the reader to \cite[Theorem~7.1]{MR2944369} for boundedness and compactness results
of embeddings of $\tH^s(\O)$ into $L^2(\O)$. 
Following~\cite{SDipierro_XRosOton_EValdinoci_2017a}, an integration-by-parts formula
can be established for the fractional Laplace operator, i.e., for all 
$f,g \in \widetilde{H}^s(\Omega)$ we have 
\begin{equation}\label{eq:34593846}\begin{split}
\left( (-\Delta)^{\frac{s}{2}} f, (-\Delta)^{\frac{s}{2}}g \right) 
 &= \frac{C_{d,s}}{2} \int_{\R^d} \int_{\R^d} \frac{(f(x)-f(y))(g(x)-g(y))}{|x-y|^{d+2s}} \dv{y} \dv{x} \\
 &= \langle (-\Delta)^s f , g \rangle_{\widetilde{H}^s(\O)',\widetilde{H}^s(\O)}.
\end{split}\end{equation}

In the proofs of \cref{s:compact} we will also encounter the Hardy-Littlewood maximal function $\mathcal{M}$. It is defined as 
\[
 \mathcal{M} f(x) = \sup_{r>0} |B_r(x)|^{-1} \int_{B_r(x)} |f(y)| dy.
\]
The maximal theorem, see \cite{Stein93}, states that for $p \in (1,\infty]$, the (sub-linear) operator $\mathcal{M}$ is bounded from $L^p(\R^d)$ to $L^p(\R^d)$. That is, there exists a constant $C > 0$ such that for all $f\in L^p(\R^d)$ we have
\[
 \|\mathcal{M} f\|_{L^p(\R^d)} \leq C\, \|f\|_{L^p(\R^d)}.
\]
The maximal function of a derivative controls the H\"older or Lipschitz-constant of a function. More precisely the following inequality holds true, see \cite{BH93,H96}
\[
 \frac{|f(x)-f(y)|}{|x-y|} \aleq \mathcal{M} |\nabla f|(x)+\mathcal{M} |\nabla f|(y).
\]
A similar estimate holds for the fractional Laplacian, which was shown in \cite[Proposition 6.6.]{S18} (but may have been known before): for $\alpha \in (0,1)$ we have 
\[
 \frac{|f(x)-f(y)|}{|x-y|^\alpha} \aleq \mathcal{M} |\laps{\alpha} f|(x)+\mathcal{M} |\laps{\alpha}f|(y).
\]
Here, by an abuse of notation we will write $\mathcal{M}$ for the square of the maximal function $\mathcal{M} \circ \mathcal{M}$.

\section{Weak compactness for integral fractional Laplacian}\label{s:compact}

The goal of this section is to identify a weak compactness property
for fractional harmonic maps. This result is critical to establish convergence 
of our numerical approximations. In particular, we establish that 
if $\{u_j\}_{j\in \mathbb{N}} \subset {\mathcal{A}}$ such that $u_j \rightharpoonup u$ in 
$\widetilde{H}^s(\Omega;\mathbb{R}^N)$ as $j \rightarrow \infty$ then we have
\begin{equation}\label{eq:conv_nonlin}
\left((-\Delta)^{\frac{s}2} u_j, (-\Delta)^{\frac{s}2} (u_j \times \phi) \right) 
\rightarrow \left((-\Delta)^{\frac{s}2} u, (-\Delta)^{\frac{s}2} (u \times \phi)\right)
\end{equation}
for any $\phi \in C_c^\infty(\Omega;\R^{N'})$, where we recall that  fractional 
harmonic maps fulfill \eqref{eq:frac_hm_cross}. In the classical setting for $s=1$
the result is a direct consequence of the product rule and properties of the
cross produt. 

To generalize the critical limit passage we begin by rewriting the nonlinear
term as follows 
\begin{align*}
\left((-\Delta)^{\frac{s}2} u_j, (-\Delta)^{\frac{s}2} (u_j \times \phi) \right) 
    &= \left((-\Delta)^{\frac{s}2} u_j,  (-\Delta)^{\frac{s}2} u_j \times \phi \right) \\
     &\ +\left((-\Delta)^{\frac{s}2} u_j,   u_j \times (-\Delta)^{\frac{s}2} \phi \right) \\
     &\ +\left((-\Delta)^{\frac{s}2} u_j,  (-\Delta)^{\frac{s}2} (u_j \times \phi) - (-\Delta)^{\frac{s}2}u_j \times \phi-u_j \times (-\Delta)^{\frac{s}2} \phi \right)  \\
     &= \left((-\Delta)^{\frac{s}2} u_j,   u_j \times (-\Delta)^{\frac{s}2} \phi \right) \\
     &\ +\left((-\Delta)^{\frac{s}2} u_j,  (-\Delta)^{\frac{s}2} (u_j \times \phi) - (-\Delta)^{\frac{s}2}u_j \times \phi-u_j \times (-\Delta)^{\frac{s}2} \phi  \right)  \\
     &= a_j + b_j
\end{align*}
where we used that $\left((-\Delta)^{\frac{s}2} u_j,  (-\Delta)^{\frac{s}2} u_j \times \phi \right) = 0$. 
Since $\widetilde{H}^s(\Omega;\mathbb{R}^N) $ is compactly embedded in $L^2(\Omega;\mathbb{R}^N)$
for $s>0$, by a weak-strong limiting argument we conclude that
\[
a_j \rightarrow \left((-\Delta)^{\frac{s}2} u,   u \times (-\Delta)^{\frac{s}2} \phi \right) . 
\]
It thus remains to show that 
\[
b_j \rightarrow \left((-\Delta)^{\frac{s}2} u,  (-\Delta)^{\frac{s}2} (u \times \phi) -u \times (-\Delta)^{\frac{s}2} \phi \right)
\]
to deduce the convergence result~\eqref{eq:conv_nonlin}.

For ease of notation we abbreviate the second argument in the inner products defining
the quantities $b_j$, i.e., we define a bilinear operator $H_s$ via 
\[
H_s(f,g) = (-\Delta)^{\frac{s}2} (fg) -f (-\Delta)^{\frac{s}2} g - ((-\Delta)^{\frac{s}2} f) g.
\]
$H_s$ measures the error term in the fractional Leibniz rule.
Since the fractional Laplace operator is applied componentwise we can extend the 
definition of $H_s$ to products of vector fields. For this, we represent
the linear cross product operation $z \mapsto u \times z$ for $z\in \R^{N'}$ by the
a matrix multiplication $z\mapsto A_u z$ for a suitably defined matrix $A_u$. 
In particular, for a matrix-valued mapping $A(x) \in \R^{N \times N'}$ and a vector-valued
function $v(x) \in \R^N$ we write  
\[
H_s (A,v)(x) = \Big(\sum_{k=1}^{N'} H_s(A^{i,k},v^k)(x)\Big)_{i=1}^N \in \R^N.
\]
With this preparation, the sequence $\{b_j\}_{j\in \N}$ is represented as 
\[
b_j = \big(\laps{s} u_j, H_s(u_j \times , \phi)\big).
\]
The following proposition provides a strong continuity property of 
the operator $H_s$ that implies the main result.

\begin{proposition}\label{pr:strongconv}
Let $s \in (0,1)$.
Let $\Omega \subset \R^d$ be a bounded open set, and assume that the sequence
$\{f_j\}_{j\in \N} \subset L^2(\R^d)$ converges to $f \in L^2(\R^d)$ locally in $\R^d$, that is
\[
\forall \text{$K \subset \R^d$ compact:} \lim_{j \to \infty} \|f_j-f\|_{L^2(K)} = 0
\]
and
\[
 \sup_{j\in \N} \|f_j\|_{L^2(\R^d)} + \|\laps{s} f_j\|_{L^2(\R^d)} < \infty.
\]
Then for every fixed $\varphi \in C_c^\infty(\Omega)$
\[
\|H_s(f_j,\varphi) - H_s(f,\varphi)\|_{L^2(\R^d)} \xrightarrow{j \to \infty} 0.
\]
\end{proposition}

\begin{remark}[$s \ge 1$]
\label{rem:sge1}
The result of \Cref{pr:strongconv} directly works for $s \in (0,2)$ with $s < d$. 
Further it can be extended to the case $s \geq 2$ by using that
the classical Laplace operator satisfies a product rule. 
\end{remark}

\begin{proof}
Abbreviating $\tilde{g}_j = f_j-f$ we need to show that for any fixed $\varphi \in C_c^\infty(\Omega)$,
\[
\|H_s(\tilde{g}_j,\varphi)\|_{L^2(\R^d)} \to  0.
\]
Let $\eta_R \in C_c^\infty(B_{2R})$, $\eta \equiv 1$ in $B_R$, where we assume that $R$ is large so
that $B_{R/2} \supset \supp \varphi$. Then
\[
 H_s(\tilde{g}_j,\varphi) = H_s(\eta_R \tilde{g}_j,\varphi) + H_s((1-\eta_R) \tilde{g}_j,\varphi).
\]
Observe that $g_j := \eta_R \tilde{g}_j$ satisfies the same assumptions as $g_j$ and additionally it has compact support.

We are going to show that 
\begin{equation}\label{eq:prop31:etaR}
 \lim_{j \to \infty} \|H_s(\eta_R \tilde{g}_j,\varphi)\|_{L^2(\R^d)} = 0 \quad \forall R > 0
\end{equation}
and
\begin{equation}\label{eq:prop31:1metaR}
 \lim_{R \to \infty} \limsup_{j \to \infty} \|H_s((1-\eta_R) \tilde{g}_j,\varphi)\|_{L^2(\R^d)} = 0
\end{equation}
Together, \eqref{eq:prop31:etaR} and \eqref{eq:prop31:1metaR} imply the claim.

{\em Proof of \eqref{eq:prop31:1metaR}:}
Observe that by disjoint support of $\varphi$ and $(1-\eta_R)$
\begin{equation}\label{eq:prop31:346346}
 H_s((1-\eta_R) \tilde{g}_j,\varphi) = \varphi \laps{s} ((1-\eta_R) \tilde{g}_j) + \laps{s} \varphi ((1-\eta_R) \tilde{g}_j).
\end{equation}
Using the integral representation of $\laps{s}$ we have for any $x \in \supp \varphi$ (and thus $1-\eta_R(x) = 0$), 
\[
\laps{s} \brac{(1-\eta_R) \tilde{g}_j}(x) = c\int_{\R^d} |x-y|^{-s-d}(1-\eta_R(y)) \tilde{g}_j(y) \dv{y}.
\]
Since $|x-y|^{-s-d}$ is smooth and bounded whenever $x \in \supp \varphi$ and $y \in \supp (1-\eta_R)$ we find that 
\[
 \|\laps{s} \brac{(1-\eta_R) \tilde{g}_j}\|_{L^\infty(\supp \varphi)} \aleq C(R) \|\tilde{g}_j\|_{L^2(\R^d)}
\]
and
\[
 \|\nabla \laps{s} \brac{(1-\eta_R) \tilde{g}_j}\|_{L^\infty(\supp \varphi)} \aleq C(R) \|\tilde{g}_j\|_{L^2(\R^d)}.
\]
That is, by the assumptions on $L^2$-boundedness of $\tilde{g}_j$, the Lipschitz norm of $\laps{s} \brac{(1-\eta_R) \tilde{g}_j}$ is uniformly bounded in $\supp \varphi$. On the other hand, by weak convergence we have for almost every $x$
\[
\laps{s} \brac{(1-\eta_R) \tilde{g}_j}(x) \to 0
\]
as $j\to \infty$. 
Since a.e. limits and uniform limits must coincide, we can argue by Arzela-Ascoli, and conclude that 
\[
 \limsup_{j \to \infty} \|\laps{s} \brac{(1-\eta_R) \tilde{g}_j}\|_{L^\infty(\supp \varphi)} = 0.
\]
For the other term in \eqref{eq:prop31:346346} we observe that for $x \in \supp (1-\eta_R)$
\[
 |\laps{s}\varphi(x)| \aleq \int_{\R^d} |x-y|^{-s-d} |\varphi(y)| \dv{y} \leq \dist(x,\supp \varphi)^{-s-d} \|\varphi\|_{L^1(\R^d)}.
\]
By assumption on $R$ we have $\dist(\R^d \backslash B_R,\supp \varphi) \geq \frac{R}{2}$. Thus
\[
\|\laps{s} \varphi ((1-\eta_R) \tilde{g}_j)\|_{L^2(\R^d)} \leq R^{-s-d} \|\varphi\|_{L^1(\R^d)}\, \|\tilde{g}_j\|_{L^2(\R^d)}.
\]
By the $L^2(\R^d)$-boundedness of $(\tilde{g}_j)_{j}$ and since $\varphi \in C_c^\infty(\R^d)$, we conclude that 
\[
 \limsup_{j \to \infty} \|\laps{s} \varphi ((1-\eta_R) \tilde{g}_j)\|_{L^2(\R^d)} \aleq C(\varphi) R^{-s-d}
\]
so that 
\[
 \lim_{R \to \infty} \limsup_{j \to \infty} \|\laps{s} \varphi ((1-\eta_R) \tilde{g}_j)\|_{L^2(\R^d)} =0.
\]

{\em Proof of \eqref{eq:prop31:etaR}:}
It remains to show that for any fixed $R>0$, setting $g_j := \eta_R \tilde{g}_j$
\[
\|H_s(g_j,\varphi)\|_{L^2(\R^d)} \to  0.
\]
Since $s \in (0,1)$, a direct calculation as in~\cite{S11} or \cite{DAP19} yields that 
\[
|H_s (g_j,\varphi)(x)| \aleq \abs{\int_{\R^d} \frac{(g_j(x)-g_j(y))(\varphi(x)-\varphi(y))}{|x-y|^{n+s}}\, \dv{y}},
\]
with a constant that depends on $d$ and $s$. Our strategy is to show that 
for any $t \in (0,1)$, $t > s-1$
\begin{equation}\label{eq:estclaim}
\|H_s(g_j,\varphi)\|_{L^2(\R^d)} 
\aleq \big(\|\varphi\|_{L^\infty(\R^d)} + \|\nabla \varphi\|_{L^\infty(\R^d)}\big)\, 
\big(\|g_j\|_{L^2(\R^d)} + \|\laps{t} g_j\|_{L^2(\R^d)}\big) .
\end{equation}
A compact embedding property $\widetilde{H}^s(\Omega) \to \widetilde{H}^t(\Omega)$
proved in \Cref{pr:RellichweakconvOmega} below then implies the statement. 

It thus remains to prove \eqref{eq:estclaim}. Since $\Omega$ is bounded 
there is $R > 0$ such that $\overline{\Omega} \subset B_{R/2}$, where 
$B_\rho$ denotes the ball of  radius $\rho$ centered at $0$.
We partition $\R^d = (\R^d \setminus B_R) \cup B_R$ and thereby obtain the
estimate
\[
\|H_s (g_j,\varphi)\|_{L^2(\R^d)}^2 \aleq I + II + III + IV,
\]
where 
\[
\begin{split}
 I =& \int_{\R^d \backslash B_R} \abs{\int_{\R^d \backslash B_R} \frac{(g_j(x)-g_j(y))(\varphi(x)-\varphi(y))}{|x-y|^{d+s}} \dv{y}}^2 \dv{x}, \\
 II =& \int_{B_R} \abs{\int_{\R^d \backslash B_R} \frac{(g_j(x)-g_j(y))(\varphi(x)-\varphi(y))}{|x-y|^{d+s}} \dv{y}}^2 \dv{x} \\
   =&\int_{B_{R/2}} \abs{\int_{\R^d \backslash B_R} \frac{g_j(x)\varphi(x)}{|x-y|^{d+s}} \dv{y}}^2  \dv{x}, \\
 III =& \int_{\R^d\backslash B_R} \abs{\int_{B_{R/2}} \frac{g_j(y)\varphi(y)}{|x-y|^{d+s}} \dv{y}}^2  \dv{x}, \\
 IV =& \int_{B_R} \abs{\int_{B_R} \frac{(g_j(x)-g_j(y))(\varphi(x)-\varphi(y))}{|x-y|^{d+s}} \dv{y}}^2 \dv{x} .
\end{split}
 \]
We will show that the terms $I,II,III,IV$ are bounded in such a way that
we can deduce~\eqref{eq:estclaim}.  \\
\emph{Estimate for $I$.} Noting that  $\supp \varphi \cup \supp g_j \subset B_{\frac{1}{2}R}$ 
we find that 
\[ 
I = 0.
\]
\emph{Estimate for $II$.}
Observe that if $x \in B_{R/2}$ and $y \in \R^d \backslash B_R$ then $|x-y| \aeq 1+|y|$, with a constant depending on $R$. Thus
\[
\begin{split}
 II \aleq& 
 \|\varphi\|_{L^\infty(\R^d)}^2 \int_{B_{R/2}} |g_{j}(x)|^2 \abs{\int_{\R^d \backslash B_R} \frac{1}{(1+|y|)^{d+s}} \dv{y} }^2 \dv{x}\\
 & \aleq  \|\varphi\|_{L^\infty(\R^d)}^2\, \|g_j\|_{L^2(\R^d)}^2.
 \end{split}
\]
\emph{Estimate for $III$.} Similarly as for $II$, for $x \in \R^d \backslash B_R$ and $y \in B_{R/2}$ we have $|x-y| \aeq 1+|x|$, and thus
\[
\begin{split}
 III \aleq& \|\varphi\|_{L^\infty(\R^d)}^2\, \int_{\R^d\backslash B_R} \frac{1}{(1+|x|)^{d+s}} \abs{\int_{B_{R/2}} |g_j(y)|\, \dv{y}}^2\, \dv{x}\\
 \aleq& \|\varphi\|_{L^\infty(\R^d)}^2\, \|g_j\|_{L^1(B_{R/2})}^2 \\
 \aleq&\|\varphi\|_{L^\infty(\R^d)}^2\, \|g_j\|_{L^2(\R^d)}^2. \\
 \end{split}
\]
\emph{Estimate for $IV$.}
Recall that we have \cite[Proposition 6.6.]{S18} for any $t \in (0,1)$,
\[
 |g_j(x)-g_j(y)| \aleq |x-y|^{t} \brac{|\mathcal{M}\laps{t} g_j(x)|+|\mathcal{M}\laps{t} g_j(y)|}.
\]
Here $\mathcal{M}$ is a finite power of the Hardy-Littlewood maximal function. 
Then
\[
\begin{split}
 IV \aleq& \|\nabla \varphi\|_{L^\infty(\R^d)}^2 
 \int_{B_R} \brac{\int_{B_R} \brac{|\mathcal{M}\laps{t} g_j(x)|+|\mathcal{M}\laps{t} g_j(y)|}\, \frac{1}{|x-y|^{d+s-1-t}} \dv{y}}^2 \dv{x}\\
 \aleq&\|\nabla \varphi\|_{L^\infty(\R^d)}^2 \int_{B_R} |\mathcal{M}\laps{t} g_j(x)|^2 \brac{\int_{B_R} \, |x-y|^{1+t-s-d}\, \dv{y}}^2 \dv{x}\\
 &+\|\nabla \varphi\|_{L^\infty(\R^d)}^2  \int_{B_R} \brac{\int_{B_R} |\mathcal{M}\laps{t} g_j(y)|\, |x-y|^{1+t-s-d}\, \dv{y}}^2 \dv{x} . 
 \end{split}
 \]
Since $1+t-s > 0$ we have for $|x| \leq R$ that 
\[
 \int_{B_R} \, |x-y|^{1+t-s-d}\dv{y} \leq C(R).
\]
On the other hand recall the Riesz potential $\lapms{\sigma} = (-\Delta)^{-\sigma/2}$ which for $\sigma \in (0,d)$ is defined as
\[
 \lapms{\sigma} f(x) = c_{d,\sigma} \int_{\R^d} |x-y|^{\sigma-d} f(y) \dv{y}.  
\]
Then, 
\[
\begin{split}
\int_{B_R} \brac{\int_{B_R} |\mathcal{M}\laps{t} g_j(y)|\, |x-y|^{1+t-s-d}\, \dv{y}}^2 \dv{x} 
 \aleq&  \big\|\lapms{1+t-s} \brac{\chi_{B_R} \mathcal{M}\laps{t} g_j} \big\|_{L^2(B_R)}^2\\
 \aleq& \big \|\lapms{1+t-s} \brac{\chi_{B_R} \mathcal{M}\laps{t} g_j} \big\|_{L^p(\R^d)}^2 
 \end{split}
\]
for any $p \in [2,\infty)$. Observe that $\chi_{B_R} \mathcal{M}\laps{t} g_j \in L^q(\R^d)$ for any $q \in [1,2]$. Indeed, by H\"older's inequality and maximal theorem, cf.~\cite{Stein93},
\[
 \|\chi_{B_R} \mathcal{M}\laps{t} g_j\|_{L^q(\R^d)} \aleq C(R) \|\mathcal{M}\laps{t} g_j\|_{L^2(\R^d)} \aleq \|\laps{t} g_j\|_{L^2(\R^d)}.
\]
Since $1+t-s \in (0,d)$ there are $p \in [2,\infty)$ and $q \in [1,2]$ with
\[
 1+t-s - \frac{d}{q} = -\frac{d}{p}.
\]
For such $p$ and $q$, by Sobolev embedding the operator $\lapms{1+t-s}: L^q(\R^d) \to L^p(\R^d)$ 
is bounded. Consequently, for that choice of $p$ and $q$ we have
\[
\begin{split}
 \left \|\lapms{1+t-s} \brac{\chi_{B_R} \mathcal{M}\laps{t} g_j} \right \|_{L^p(\R^d)} \aleq & \|\mathcal{M}\laps{t} g_j\|_{L^q(B_R)}\\
 \aleq &\|\laps{t} g_j\|_{L^2(B_R)} \aleq \|\laps{t} g_j\|_{L^2(\R^d)}.
 \end{split}
\]
On combining previous estimates we find that 
\[
 IV \aleq \|\nabla \varphi\|_{L^\infty(\R^d)}^2\, \|\laps{t} g_j\|_{L^2(\R^d)}^2.
\]
The estimates for $I,II,III,IV$ imply \eqref{eq:estclaim}. 
\end{proof}

The following embedding result is used in the proof of \Cref{pr:strongconv}.

\begin{proposition}\label{pr:RellichweakconvOmega}
Let $s \in (0,1)$ and let $\Omega \subset \R^d$ be a bounded open set.
Assume $\{g_j\}_{j\in \N} \subset L^2(\R^d;\R^N)$ strongly converges to $0$ with
$\supp g_j  \subset \overline{\Omega}$ and 
\begin{equation}\label{eq:prrellich:boundedness}
 \sup_{j\in \N} \|\laps{s} g_j\|_{L^2(\R^d)} < \infty.
\end{equation}
Then for any $t \in (0,s)$
\begin{equation}\label{eq:estclaim2}
 \lim_{j \to \infty} \|\laps{t} g_j\|_{L^2(\R^d)} =0.
\end{equation}
\end{proposition}
\begin{proof}
It suffices to show that there is a subsequence that satisfies \eqref{eq:estclaim2}, since 
then any cluster point of $(\|\laps{t} g_j\|_{L^2(\R^d)})_{j \in \N}$ is zero, and
thus the whole sequence converges. 
First we observe that $g_j$ weakly converges to $0$ in $W^{s,2}(\R^d)$. Indeed, 
by assumption
\begin{equation}\label{eq:supjgjlp2}
 \sup_{j} \brac{\|g_j\|_{L^2(\R^d)} + \|\laps{s} g_j\|_{L^2(\R^d)} } < \infty.
\end{equation}
Thus up to taking a subsequence $g_j$ converges to some $g \in W^{s,2}(\R^d)$ with $g \equiv 0$ in $\R^d \backslash \Omega$. 
Since $g_j$ converges to zero strongly in $L^2$, we know that $g$ vanishes identically. 

Noting $t < s$ we can use Sobolev embedding 
and have for some $p>2$ (if $s < d/2$ we can take $p = 2d/(d-2s)$, otherwise any $p > 2$ is permitted) 
from \eqref{eq:supjgjlp2} 
\begin{equation}\label{eq:supjgjlp}
 \sup_{j} \brac{\|\laps{t} g_j\|_{L^2(\R^d)} + \|\laps{t} g_j\|_{L^p(\R^d)}} 
\aleq \sup_{j} \brac{\|g_j\|_{L^2(\R^d)} + \|\laps{s} g_j\|_{L^2(\R^d)} }. 
\end{equation}
As in \cite{ARS20} we split $\R^d$ into two sets. For $\eps > 0$ we define
\[
 \Omega_{o,\eps} = \{x \in \R^d: \dist(x,\Omega) > \eps\}, \quad 
 \Omega_{i,\eps} = \R^d \backslash \Omega_{o,\eps}.
\]
\emph{Estimate on $\Omega_{o,\eps}$.}
Since $\supp g \subset \overline{\Omega}$ and $\dist(\Omega,\Omega_{o,\eps}) \ageq \eps$ we have for $x \in \Omega_{o,\eps}$
\[
 \laps{t} g_j(x) = \int_{\Omega} g_j(y) |x-y|^{-d-t}\dv{y} \aleq (1+|x|)^{-d-t} \int_{\Omega} |g_j(y)|\dv{y}.
\]
Using H\"older's inequality we then find 
\begin{equation}\label{eq:oomegaset}
{\begin{split}
&\|\laps{t} g_j \|_{L^2(\Omega_{o,\eps})}
\aleq& C(\eps) \|g_j\|_{L^2(\R^d)} \to  0.
\end{split}}
\end{equation}
\emph{Estimate on $\Omega_{i,\eps}$.}
Observe that if we set $f_j := \laps{t} g_j$ we have from \eqref{eq:supjgjlp2} and \eqref{eq:supjgjlp}
\[
 \sup_{j\in \N} \brac{\|\laps{s-t} f_j \|_{L^2(\R^d)} + \|f_j\|_{L^2(\R^d)}} < \infty.
\]
From \cite[Proposition 3.2]{ARS20} we obtain that $f_j$ converges strongly to some $f$ in $L^2(\Omega_{i,\eps})$ 
(since $\Omega_{i,\eps}$ is bounded). Since on the other hand $f_j$ weakly converges to zero we have
$f \equiv 0$ and thus
\begin{equation}\label{eq:iomegaset}
 \lim_{j \to \infty} \|\laps{t} g_j\|_{L^2(\Omega_{i,\eps})} = \lim_{j \to \infty} \|f_j\|_{L^2(\Omega_{i,\eps})} = 0.
 \end{equation}
On combining \eqref{eq:oomegaset}, \eqref{eq:iomegaset} we infer \eqref{eq:estclaim2}, which
completes the proof. 
\end{proof}
 
Remark~\ref{rem:sge1}, also directly applies to Proposition~\ref{pr:RellichweakconvOmega}.
The propositions imply the main result of this section. 

\begin{theorem}[Weak compactness]\label{thm:compact1}
Let $\{u_j\}_{j\in \mathbb{N}} \subset L^2(\R^d;\R^N)$ be a sequence such that 
\begin{equation}\label{eq:cond}
\sup_{j\in \N} \|u_j\|_{L^2(\R^d)} + \|(-\Delta)^{\frac{s}{2}} u_j\|_{L^2(\Omega)} < \infty
\end{equation}
$|u_j(x)|^2 \to 1$ as $j\to \infty$ for almost every $x\in \O$.
For every accumulation point $u\in L^2(\R^d;\R^N)$ we have $|u(x)|^2=1$ for 
almost every $x\in \O$. Moroever if $u_j \wto u$ in $L^2(\R^d,\R^N)$ as $j\to \infty$
then 
\[
\big((-\Delta)^{\frac{s}{2}} u_j ,(-\Delta)^{\frac{s}{2}} (u_j \times \phi)\big)
\to \big((-\Delta)^{\frac{s}{2}} u ,(-\Delta)^{\frac{s}{2}} (u \times \phi)\big)
\]
for every $\phi \in C^\infty_c(\O;\R^{N'})$. If $u_j \in \vec{N} + \widetilde{H}^s(\O;\R^N)$
for all $j\in \N$ then also $ u \in \vec{N} + \widetilde{H}^s(\O;\R^N)$.
\end{theorem}

\begin{proof}
Set 
\[
 \Gamma_j := \big((-\Delta)^{\frac{s}{2}} u_j ,(-\Delta)^{\frac{s}{2}} (u_j \times \phi)\big)
\]
and 
\[
 \Gamma := \big((-\Delta)^{\frac{s}{2}} u ,(-\Delta)^{\frac{s}{2}} (u \times \phi)\big).
\]
Since $u_j$ is by assumption uniformly bounded in $H^{s}(\R^d)$, by Rellich's theorem, up to taking a subsequence $u_{j_k}$ converges strongly to $u$ in $L^2(K)$ for any compact set $K$. 

Splitting as described in the beginning of this section $\Gamma_j$ into $a_j$ and $b_j$, we obtain from \Cref{pr:strongconv} that a subsequence of $b_{j_k}$ converges to $b$.

That is, we have
\[
 \Gamma_{j_k} \to \Gamma.
\]
We can make this argument for any subsequence of $(\Gamma_j)_{j \in \N}$ and obtain a subsubsequence which converges to $\Gamma$. This implies that any cluster point of $(\Gamma_j)_{j \in \N}$ must actually be $\Gamma$, which implies that $\Gamma$ is indeed the limit of the whole sequence $\Gamma_j$.
\end{proof}

\begin{remark}
The assumed uniform bound $\|(-\Delta)^{\frac{s}{2}} u_j\|_{L^2(\Omega)}$  
in Theorem~\ref{thm:compact1} can equivalently be replaced
by a bound for $\|(-\Delta)^{\frac{s}{2}} u_j\|_{L^2(\R^d)}$. 
\end{remark}

\section{Finite element setting}\label{s:fem}

We consider sequences of uniformly shape regular and conforming triangulations $\{\cT_h\}_{h>0}$
of the bounded polyhedral Lipschitz domain $\O\subset \R^d$ consisting of triangles or tetrahedra;
the parameter $h>0$ represents a maximal mesh-size. 
The space of continuous, piecewise affine finite element functions is defined via 
\[
\cS^1(\cT_h) =
\{v_h \in C(\overline{\O}): v_h|_T \in P_1(T) \text{ for all }T\in \cT_h\}.
\]
We let $\cN_h$ be the set of vertices of elements, which are the nodes of the finite
element space. The set $\{\vphi_z:z\in \widetilde\cN_h\}$ is the nodal basis
consisting of hat functions $\vphi_z\in \cS^1(\cT_h)$ associated with vertices 
$z\in \cN_h$. The corresponding nodal interpolation operator
$\cI_h: C(\overline{\O})\to \cS^1(\cT_h)$ is given by
\[
\cI_h v = \sum_{z\in \cN_h} v(z) \vphi_z.
\]
We note the classical nodal interpolation estimates 
\[
h_T^{-1} \|(v-\cI_hv)\|_{L^2(T)} + \|\nabla (v-\cI_hv)\|_{L^2(T)}
\aleq  h_T \|D^2 v\|_{L^2(T)}
\]
for $v \in H^2(T)$ with the diameter $h_T>0$ of an element
$T\in \cT_h$ and a constant $c>0$ that is independent of $h>0$.
We remark that the discrete $L^p$ norms defined via
\[
\|v\|_{L^p_h(\O)}^p = \int_\O \cI_h |v|^p \dv{x} = \sum_{z\in \cN_h} \b_z |v(z)|^p,
\quad \b_z = \int_\O \vphi_z\dv{x}, 
\]
for $v\in C(\overline{\O})$ are equivalent to $L^p$ norms on the
space $\cS^1(\cT_h)$. 
If the triangulations are quasiuniform then for 
given $u_h\in \cS^1(\cT_h)^N$ and $\phi \in C^\infty_c(\O;\R^{N'})$
we have for $0< s\le 1$ that 
\begin{equation}\label{eq:aux_interpol}
\begin{split}
\|(-\Delta)^{\frac{s}{2}}(u_h &\times \phi - \cI_h [u_h\times \phi])\| \\
&\aleq\|\nabla (u_h \times \phi - \cI_h [u_h\times \phi]\| \\
&\aleq h \big(\|u_h\| \|D^2 \phi\|_{L^\infty(\O)} + \|\nabla u_h\| \|\nabla \phi\|_{L^\infty(\O)}\big) \\
&\aleq h \|\phi\|_{W^{2,\infty}(\O)} \big(\|u_h\| + h^{s-1} \|(-\Delta)^{\frac{s}{2}} u_h\|\big) \\
&\aleq c h^s \|\phi\|_{W^{2,\infty}(\O)} \big(\|u_h\| + \|(-\Delta)^{\frac{s}{2}} u_h\|\big).
\end{split}
\end{equation}
Here, we used the inequality 
$\|(-\Delta)^{\frac{s}{2}} v\| \aleq \|\nabla v\|$ for $v\in H^1(\O)$
and the inverse estimate  
\begin{equation}\label{eq:inv_est_a}
\|\nabla v_h\|\aleq h^{s-1} \|(-\Delta)^{\frac{s}{2}} v_h\|
\end{equation}
for $v_h \in \cS^1(\cT_h)^N$, cf., e.g.,~\cite[Prop.~3.1]{JPBorthagaray_PCJr_2019a}.  
Below we also use the inverse estimate, which for quasi-uniform meshes immediately
follows from \cite[Eq.~(3.2)]{borthagaray2020local}, however, a similar expression
can also be derived for just the shape regular meshes following the proof
of \cite[Lemma~5.2]{borthagaray2020local},
\begin{equation}\label{eq:inv_est_b}
\|(-\Delta)^{\frac{s}{2}} v_h\| \aleq  h^{-s} \|v_h\|.
\end{equation}
To impose exterior Dirichlet conditions and to approximate the fractional Laplace operator
we consider a larger domain $\tO\subset \R^d$ with $\overline{\O}\subset \tO$
and a triangulation $\tcT_h$ of $\tO$ that extends $\cT_h$. We then let
$\widetilde{\cI}_h \vec{N} \in \cS^1(\tcT_h)^N$ be the nodal interpolant of 
$\vn\in C^\infty_c(\tO;\R^N)$ on $\tcT_h$. With this we obtain the following discrete variant
of Theorem~\ref{thm:compact1}.

\begin{corollary}[Discrete weak compactness]\label{cor:compact_discrete}
Let $\{u_h\}_{h>0} \subset L^2(\widetilde\O;\R^N)$ be a sequence of finite element
functions $u_h\in \cS^1(\tcT_h)^N$ subordinated to a sequence of quasiuniform
triangulations $\{\cT_h\}_{h>0}$ such that $u_h = \tcI_h \vn$ in $\tO\setminus \O$,
\[
\|u_h\|_{L^2(\O)} + \|(-\Delta)^{\frac{s}{2}} u_h\| \le c
\]
for all $h>0$ and $|u_h(x)|^2 \to 1$ as $h\to 0$ for almost every $x\in \O$.
For every accumulation point $u\in L^2(\widetilde\O;\R^N)$ we have $|u(x)|^2=1$ 
for almost every $x\in \O$ and if $u_{h'} \wto u$ in $L^2(\widetilde\O;\R^N)$
for a subsequence $h' \to 0$ then we have 
\[
\big((-\Delta)^{\frac{s}{2}} u_{h'},(-\Delta)^{\frac{s}{2}} \cI_{h'} [u_{h'} \times \phi]\big)
\to \big((-\Delta)^{\frac{s}{2}} u ,(-\Delta)^{\frac{s}{2}} [u \times \phi]\big)
\]
for every $\phi \in C^\infty_c(\O;\R^{N'})$ as $h' \to 0$. 
\end{corollary}

\begin{proof}
The result follows from applying Theorem~\ref{thm:compact1} to the corrected sequence
$\{\tu_h\}_{h>0} \subset L^2(\tO;\R^N)$ defined via $\tu_h = u_h - \cI_h \vn + \vn$, 
which satisfies $\tu_h = \vn$ in $\tO\setminus \O$,  
noting that $\cI_h \vn - \vn \to 0$ in $H^1(\O;\R^N)$, and
incorporating the estimate~\eqref{eq:aux_interpol}.
\end{proof}


\section{Numerical schemes and convergence}\label{s:nconv}

In this section we devise numerical schemes for prototypical
problems related to fractional harmonic maps into spheres and
show that they approximate corresponding continuous objects. 
Throughout the following we use the definitions
\[\begin{split}
\cA_h &= \{u_h - \cI_h [\vec{N}] \in  \cS^1_0(\cT_h)^N: |u_h(z)|^2 = 1 
\text{ f.a. }z\in \cN_h\}, \\
\cF_h[u_h] &= \{v_h\in \cS^1_0(\cT_h)^N: v_h(z)\cdot u_h(z) = 0 
\text{ f.a. } z\in \cN_h\}.
\end{split}\]

\subsection{Fractional harmonic maps}\label{s:fhm}
We consider the problem of finding critical points for
the fractional Dirichlet energy subject to a sphere constraint
and Dirichlet exterior conditions in $\R^d\setminus \O$ 
determined by a suitable vector field
$\vec{N}\in C^\infty(\R^d;\R^N)$. We recall that the problem is 
equivalent to determining $u\in \cA$ such that
\begin{equation}\label{eq:hm_crit_cross}
((-\Delta)^{\frac{s}{2}} u,(-\Delta)^{\frac{s}{2}} v) = 0
\end{equation}
for all $v\in C^\infty_c(\O;\R^N))$ with $u\cdot v = 0$ in $\O$. 
A {\em discrete fractional harmonic map} $u_h\in \cA_h$ satisfies the equation
\begin{equation}\label{eq:discr_hm}
((-\Delta)^{\frac{s}{2}} u_h, (-\Delta)^{\frac{s}{2}} v_h) = 0
\end{equation}
for all $v_h\in \cF_h[u_h]$. Bounded sequences of discrete fractional 
harmonic maps weakly accumulate at fractional harmonic maps. 

\begin{proposition}
Let $N\ge 2$ and $\{u_h\}_{h>0} \subset L^2(\O;\R^N)$ be a sequence of discrete 
fractional harmonic maps on a sequence of quasiuniform triangulations 
with $ \|(-\Delta)^{\frac{s}{2}} u_h\| \le c$ 
for all $h>0$. Then every accumulation point $u\in \vec{N} + \tH^s(\O;\R^N)$ satisfies
$u\in \cA$ and~\eqref{eq:hm_crit_cross}. 
\end{proposition}

\begin{proof} 
The statement is an immediate consequence of  
Corollary~\ref{cor:compact_discrete} together with the nodal 
interpolation estimate 
\[\begin{split}
\||u_h|^2 -1 \|_{L^1(\O)} &= \||u_h|^2 - \cI_h[|u_h|^2]\|_{L^1(\O)}  \\
& \aleq h^2 \|D_h^2  |u_h|^2 \|_{L^1(\O)} \le c h^2 \|\nabla u_h\|_{L^2(\O)}^2,
\end{split}\]
where $D_h^2$ denotes the elementwise application of the second derivative.
With the compact embedding $\tH^s(\O;\R^N)\hookrightarrow L^1(\O;\R^N)$ and the inverse
estimate~\eqref{eq:inv_est_a} we deduce that 
$|u|^2=1=\lim_{h'\to 0} |u_{h'}|^2$ almost everywhere in $\O$.
\end{proof}

\subsection{Fractional harmonic map heat flow}\label{s:fhmheat}
We next discuss the convergence of numerical approximations of the
$L^2$-gradient flow of the constrained fractional Dirichlet energy,
i.e., suitable solutions $u:(0,T)\times \O \to S^{N-1}$ of the evolution equation
\[
(\p_t u,v) + ((-\Delta)^{\frac{s}{2}} u,(-\Delta)^{\frac{s}{2}} v) = 0
\]
for all $v\in \tH^s(\O;\R^N)$ with
$u(t,x)\cdot v(x) = 0$. The problem is complemented by the Dirichlet exterior 
condition $u(t,\cdot) - \vec{N} \in \tH^s(\O;\R^N)$ for all $t\in (0,T)$
and the initial condition $u(0,\cdot) = u_0$ in $\O$. The following numerical
scheme uses a semi-implicit time discretization with an explicit
treatment of the linearized length constraint. Note that we follow~\cite{Bart16}
and avoid a correction step which leads to a progressive
constraint violation. 

\begin{algorithm}[Discrete $L^2$-flow]\label{alg:l2_flow_discr}
Let $\tau>0$ and $u_h^0\in \cI_h \vec{N} + \cS^1_0(\cT_h)^N$
with $|u_h^0(z)|^2 =1$ for all $z\in \cN_h$. For 
$k=0,1,\dots,K$ compute $d_t u_h^k\in \cF_h[u_h^{k-1}]$ such that
\[
(d_t u_h^k,v_h) 
+ \left((-\Delta)^{\frac{s}{2}} [u_h^{k-1}+\tau d_t u_h^k],(-\Delta)^{\frac{s}{2}} v_h\right) = 0
\]
for all $v_h \in \cF_h[u_h^{k-1}]$, and define 
$u_h^k= u_h^{k-1} + \tau d_t u_h^k$.
\end{algorithm}

The algorithm is unconditionally stable and convergent; the violation
of the constraint is bounded independently of the number of iterations.

\begin{proposition}\label{prop:stab_conv}
There exist uniquely defined iterates $\{u_h^k\}_{k=0,\dots,K} \in \cS^1(\cT_h)^N$
with $u_h^k|_{\widetilde\O \setminus \O} = \cI_h\vec{N}]_{\widetilde\O \setminus \O}$ 
and for all $K'\le K$
\[
\frac12 \|(-\Delta)^{\frac{s}{2}} u_h^{K'} \|^2 
+ \tau \sum_{k=1}^{K'} \|d_t u_h^k\|^2 
\le \frac12 \|(-\Delta)^{\frac{s}{2}} u_h^0 \|^2.
\]
Moreover, letting $e_{h,0}$ denote the discrete initial energy on
the right-hand side of the inequality we have that
\[
\| \cI_h |u_h^k|^2 -1 \|_{L^1(\O)} \aleq \tau e_{h,0}.
\]
If the triangulations are quasiuniform then every weak accumulation point 
\[
u\in \vec{N}+ H^1(0,T;L^2(\O;\R^N))\cap L^\infty(0,T;\tH^s(\O;\R^N))
\]
for $(h,\tau) \to 0$ of the sequence of linear interpolants $\{\hu_h\}_{h>0}$ 
of the iterates $\{u_h^k\}_{k=0,\dots,K}$ solves the fractional harmonic map 
heat flow problem.
\end{proposition}

\begin{proof}
For every $k=1,2,\dots,K$ we have by the Lax--Milgram lemma
that there exists a unique solution $d_t u_h^k\in \cF_h[u_h^{k-1}]$
for $k=1,2,\dots,K$. 
By choosing $v_h = d_t u_h^k$ in the discrete equation and using
the binomial formula $2 b\cdot (b-a) = (b-a)^2 + (b^2-a^2)$ we
find that 
\[
\|d_t u_h^k \|^2 + \frac{d_t}{2} \|(-\Delta)^\frac{s}{2} u_h^k\|^2 
+ \frac{\tau}{2} \|(-\Delta)^\frac{s}{2} d_t u_h^k\|^2 = 0.
\]
A summation over $k=1,2,\dots,K'$ yields the asserted identity.
The orthogonality relation $d_t u_h^k \cdot u_h^{k-1}= 0$ at the nodes
in $\cN_h$ shows that for all $z\in \cN_h$ we have 
\[
|u_h^k(z)|^2  = |u_h^{k-1}(z)|^2 + \tau^2 |d_t u_h^k(z)|^2
= \dots = |u_h^0(z)|^2 + \tau^2 \sum_{\ell=1}^k |d_t u_h^\ell(z)|^2.
\]
By using $|u_h^0(z)|^2=1$ and summing over the nodes $z\in \cN_h$
and using the equivalence of discrete and continuous $L^p$ norms 
we deduce the estimate for the constraint violation. 
We let $\hu_h$ and $u_h^\pm$ denote the piecewise linear and
constant interpolants of $(u_h^k)_{k=0,\dots,K}$. In particular, 
$u_h^- = u_h^{k-1} + \tau d_t u_h^k$ and $u_h^+ = u_h^{k-1}$. 
For almost every
$t \in (0,T)$ we have that 
\[
(\p_t \hu_h,v_h) + \left((-\Delta)^{\frac{s}{2}} u_h^-,(-\Delta)^{\frac{s}{2}} v_h\right) = 0
\]
for all $v_h \in \cS^1_0(\cT_h)^N$ satisfying $v_h \cdot u_h^+(t,\cdot)=0$.
With the help of Corollary~\ref{cor:compact_discrete}
we may pass to a limit in this equation as $(h,\tau) \to 0$. 
\end{proof}

\subsection{Spin dynamics}\label{s:spin}
We finally address the approximation of solutions of the
unconstrained but length-preserving evolution equation
\[
\p_t u = (-\Delta)^s u \times u
\]
for a given initial state $u_0$, which describes the physical
principle that the rate of change of angular momentum equals torque.
For simplicity, we consider here periodic or homogeneous Neumann 
boundary conditions on $\R^d\setminus\O$. The evolution equation is length and energy 
preserving which is also satisfied by the following numerical scheme.
For this, the use of midpoint values
\[
u_h^{k-1/2}(z) = \frac12 (u_h^k(z)+u_h^{k-1}(z))
\]
is essential, we follow~\cite{KarWeb14,Bart15}.

\begin{algorithm}[Discrete spin dynamics]\label{alg:precess_discr}
Let $\tau>0$ and $u_h^0 \in \cS^1(\cT_h)^3$ with $|u_h^0(z)|^2=1$ for all
$z\in \cN_h$.
For $k=1,2,\dots,K$ compute $u_h^k \in \cS^1(\cT_h)^3$ such that
\[
(d_t u_h^k,v_h)_h = \big((-\Delta)^{\frac{s}{2}} u_h^{k-1/2}, (-\Delta)^{\frac{s}{2}} \cI_h [u_h^{k-1/2} \times v_h] \big)
\]
for all $v_h \in \cS^1(\cT_h)^3$. 
\end{algorithm}

A useful 
representation of the scheme is obtained with the discrete fractional
Laplacian $(-\Delta)^s_h:\cS^1(\cT_h)^3\to \cS^1(\cT_h)^3$ obtained as
the representative of the corresponding bilinear form with the
discrete inner product, i.e., $(-\Delta)^s_h u_h \in \cS^1_0(\cT_h)^3$
is defined as the unique function $y_h\in \cS^1_0(\cT_h)^3$ with 
\[
(y_h,v_h)_h = \big((-\Delta)^{\frac{s}{2}} u_h, (-\Delta)^{\frac{s}{2}} v_h\big) 
\]
for all $v_h \in \cS^1(\cT_h)^3$. With this discrete operator we have
\[
(d_t u_h^k,v_h)_h = \big(((-\Delta)^s_h u_h^{k-1/2}, \cI_h[u_h^{k-1/2}\times v_h]\big)_h
\]
for all $v_h\in \cS^1(\cT_h)^3$. Owing to the use of the discrete
inner product this is equivalent to the equality of nodal values, i.e.,
\[
d_t u_h^k(z) = u_h^{k-1/2}(z) \times (-\Delta)^s_h u_h^{k-1/2}(z)
\]
for all $z\in \cN_h$. With these preparations we deduce the
constraint and energy preservation properties.  

\begin{proposition}\label{prop:discr_precession}
There exists a sequence $\{u_h^k\}_{k=0,\dots,K}$ that satisfies the nonlinear
discrete system of Algorithm~\ref{alg:precess_discr} for $k=1,2,\dots,K$.
Every solution $\{u_h^k\}_{k=0,\dots,K}$ satisfies $|u_h^k(z)|^2 = 1$ for all
$z\in \cN_h$ and $k=0,1,\dots,K$ and 
\[
\frac12 \|(-\Delta)^{\frac{s}{2}} u_h^k\|^2 =  \frac12 \|(-\Delta)^{\frac{s}{2}} u_h^0\|^2
\]
for $k=1,2,\dots,K$. 
\end{proposition}

\begin{proof}
Given $u_h^{k-1}$ the average $\ou_h^k = (u_h^k+u_h^{k-1})/2$ is required to 
satisfy $\Phi_h(\ou_h^k)[v_h] = 0$ for all $v_h\in \cS^1(\cT_h)^3$, where
\[
\Phi_h(\ou_h^k)[v_h] = \frac{2}{\tau} (\ou_h^k-u_h^{k-1},v_h)_h -
 \big((-\Delta)^s_h \ou_h^k, \cI_h[\ou_h^k\times v_h]\big)_h.
\]
By choosing $v_h = \ou_h^k$ we find that 
\begin{align*}
\Phi_h(\ou_h^k)[\ou_h^k] = \frac{2}{\tau} (\ou_h^k- u_h^{k-1},\ou_h^k) 
&\ge \frac{2}{\tau} \big( \|\ou_h^k\|^2 - \|\ou_h^k\| \|u_h^{k-1}\|\big) \\
& \ge \frac{1}{\tau} \left( \|\ou_h^k\|^2 - \|\ou_h^{k-1}\|^2 \right) ,
\end{align*}
i.e., $\Phi_h(\ou_h^k)[\ou_h^k] \ge 0$ for $\|\ou_h^k\| \ge \|\ou_h^{k-1}\|$.
Hence, Brouwer's fixed-point theorem implies the existence of a solution
$\Phi_h(\ou_h^k)[v_h] = 0$ for all $v_h\in \cS^1(\cT_h)^3$. If 
$\{u_h^k\}_{k=0,\dots,K}$ is an arbitrary sequence satisfying the equations of  
Algorithm~\ref{alg:precess_discr} then choosing $v_h = u_h^{k-1/2}(z) \vphi_z$
implies that 
\[
\b_z d_t u_h^k(z) \cdot u_h^{k-1/2}(z) = 0,
\]
i.e., $\b_z d_t |u_h^k(z)|^2 =0$ and hence $|u_h^k(z)|^2 =1$ for all
$z\in \cN_h$ and $k=1,2,\dots,K$. By choosing $v_h = (-\Delta)^s_h u_h^{k-1/2}$
we find that 
\[
0 = (d_t u_h^k,(-\Delta)^s_h u_h^{k-1/2})_h = ((-\Delta)^{\frac{s}{2}} d_t u_h^k, (-\Delta)^{\frac{s}{2}} u_h^{k-1/2}) = 0,
\]
i.e., $d_t \|(-\Delta)^{\frac{s}{2}} u_h^k\|^2 = 0$. 
\end{proof}

If the step size is sufficiently small then the nonlinear systems of equations 
that arise in the steps of Algorithm~\ref{alg:precess_discr} have unique solutions which
can be computed with a simple fixed-point
iteration.

\begin{proposition}\label{prop:fixed_spin}
Given $u_h^{k-1} \in \cS^1(\cT_h)^3$ with $\|u_h^{k-1}\|_{L^\infty(\O)}= 1$ 
a solution $u_h^k \in \cS^1(\cT_h)^3$ 
is determined via $u_h^k = 2 r_h - u_h^{k-1}$, where $r_h\in \cS^1(\cT_h)^3$
is a fixed point of the iteration 
\[
r_h^\ell = u_h^{k-1} 
+ \frac{\tau}{2} \cI_h [ r_h^\ell \times (-\Delta)^s_h r_h^{\ell-1}]
\]
for $\ell = 1,2,\dots$ with arbitrary $r_h^0 \in  \cS^1(\cT_h)^3$. The 
iteration is globally convergent provided that $\tau < c_\inv^{-2} h^{2s} |\O|^{-1/2}$. 
\end{proposition}

\begin{proof}
The Lax--Milgram lemma implies the existence of uniquely defined 
iterates $(r_h^\ell)_{\ell=1,2,\dots}$ which are equivalently characterized via
\[
(r_h^\ell,v_h)_h = (u_h^{k-1},v_h)_h + \frac{\tau}{2} 
( r_h^\ell \times (-\Delta)^s_h r_h^{\ell-1},v_h)_h
\]
for all $v_h\in \cS^1(\cT_h)^3$. By choosing $v_h = r_h^\ell $ we find that
$\|r_h^\ell\|_h \le \|u_h^{k-1}\|_h \le |\O|^{1/2}$
for $\ell =1,2,\dots$. The difference $\d_h^\ell = r_h^\ell - r_h^{\ell-1}$
of two iterates satisfies 
\[
\d_h^\ell = \frac{\tau}{2} \cI_h [ \d_h^\ell \times (-\Delta)^s_h r_h^{\ell-1}]
+ \frac{\tau}{2} \cI_h [ r_h^{\ell-1} \times (-\Delta)^s_h \d_h^{\ell-1}]
\]
using the inverse estimate~\eqref{eq:inv_est_b} we find that
$\|(-\Delta)^s_h r_h^{\ell-1} \|_h \le  c_\inv^2 h^{-2s} |\O|^{1/2}$ and
\[
2 \|\d_h^\ell\|_h \le \tau c_\inv^2 h^{-2s} |\O|^{1/2} \big( \|\d_h^\ell\|_h +  \|\d_h^{\ell-1}\|_h \big).
\]
Hence, if $q= \tau c_\inv^2 h^{-2s} |\O|^{1/2} <1$ we find that $\|\d_h^\ell\|_h \le q^{\ell-1} \|\delta^1_h\|_h$
and hence that $r_h^\ell$ converges as $\ell \to \infty$. 
\end{proof}

\begin{remark}
With the linear interpolants $(\hu_{h,\tau})$ and the
piecewise averages $(\ou_{h,\tau})$ of the iterates $(u_h^k)_{k=0,\dots,K}$ 
the numerical scheme can be written as 
\[
(\p_t \hu_{h,\tau},v_h) + ((-\Delta)^{\frac{s}{2}} \ou_h, (-\Delta)^{\frac{s}{2}}\cI_h[\ou_h \times v_h])  = 0
\]
for all $v_h \in \cS^1(\cT_h)^3$. Weak accumulation points of the sequence
$(\hu_{h,\tau})$ as $(h,\tau)\to 0$ for sequences of quasiuniform triangulations
then satisfy the equation
\[
- \int_0^T (u,\p_t \phi) \dv{t} 
+ \int_0^T ((-\Delta)^{\frac{s}{2}} u, (-\Delta)^{\frac{s}{2}}[u\times \phi]) \dv{t} 
= (u_0,\phi(0))
\]
for all $\phi \in C^\infty([0,T];C^\infty_c(\O;\R^3))$ with $\phi(T,\cdot)=0$. This
follows from an application of Corollary~\ref{cor:compact_discrete}.
\end{remark}


\section{Numerical experiments}\label{s:numexp}

In this section we illustrate the performance of the numerical methods via numerical
experiments for one-dimensional spin chain dynamics and the fractional harmonic map heat
flow. Fractional harmonic maps arise here as stationary limiting points of the 
fractional harmonic map heat flow. 

\subsection{Spin dynamics}
We consider the spin system from~\cite{ZhoSto15}
\begin{equation}\label{eq:spin_per}
\p_t u = - u \times (-\Delta)^s u, \quad u(0) = u_0,
\end{equation}
in a one-dimensional periodic setting, i.e., we use periodic
boundary conditions on $\O=(0,2\pi)$ and write $\O=\T$.
This allows us to approximate the fractional Laplace operator via
a Fourier sum, i.e., given a continuous function $w \in C(\T)$ we
define its discrete Fourier transform via the coefficients
\[
\tv_k = \frac{2\pi}{M} \sum_{j=0}^{M}  e^{-\i k x_j} v(x_j) 
\]
for $k=-M/2,-M/2+1,\dots,M/2-1$, $M\in \N$ even, and  
with $x_j = j 2 \pi/M$, $j=0,1,\dots,M$, we refer the reader to~\cite{AntBar17}
for details. The coefficients are obtained from standard
implementations of the FFT method. The span of the trigonometric basis functions
$\vphi^k(x) = e^{\i k x}$, $x\in \T$, $k=-M/2,\dots,M/2+1$, defines
the discrete space $\cS_M$. For $v\in \cS_M$ we have the
representation
\[
v = \frac{1}{2\pi} \sum_{k=-M/2}^{M/2-1} \tv_k \vphi^k.
\]
The discrete fractional Laplace operator $(-\Delta)_M^s$ is for
$v\in C(\T)$ defined as
\[
(-\Delta)_M^s v = \frac{1}{2\pi} \sum_{k=-M/2}^{M/2-1} |k|^{2s} \tv_k \vphi^k.
\]
We remark that for functions $v,w\in \cS_M$ quadrature is exact in
the approximation of the $L^2$~inner product of complex valued functions,
i.e., we have
\[
(v,w)_{L^2(\T;\C)} = \frac{2\pi}{M} \sum_{j=0}^M v(x_j) \overline{w(x_j)}
= (v,\overline{w})_h. 
\]
With these settings we replace Algorithm~\ref{alg:precess_discr}
by the following iteration in which $\cT_h$ is a uniform partition
of $\T$ into $M$ intervals $T_j = [x_{j-1},x_j]$, $j=1,2,\dots,M$ of
length $h=2\pi/M$

\begin{algorithm}[Discrete spin dynamics]\label{alg:precess_discr_2}
Let $\tau>0$ and $u_h^0 \in \cS^1(\cT_h)^3$ with $|u_h^0(z)|^2=1$ for all
$z\in \cN_h$.
For $k=1,2,\dots,K$ compute $u_h^k \in \cS^1(\cT_h)^3$ such that 
\[
d_t u_h^k = - \cI_h \big[ u_h^{k-1/2} \times (-\Delta)^s_M u_h^{k-1/2} \big].
\]
\end{algorithm}

Our first example leads to a solitary traveling wave solution given 
via the simplest Blaschke function $\cB(z) = z$, cf.~\cite{LenSch18}.

\begin{example}\label{ex:travel_spin}
Let $s=1/2$, $T=4\pi$, $v=1/2$, and for $x \in \T$ define 
\[
u^0(x) = \big[v,(1-v^2)^{1/2} \cos(x), (1-v^2)^{1/2} \sin(x)\big]^\transp.
\]
Then $u(t,x)= u^0(x-vt)$ solves the spin dynamics system~\eqref{eq:spin_per}.
\end{example}

Figure~\ref{fig:ex_blaschke_snaps} shows snapshots of the evolution computed
with Algorithm~\ref{alg:precess_discr_2}. We observe that the initial
state re-occurs when the time horizon $T=4\pi$ is reached by the
time stepping scheme. The nonlinear systems of equations in the time
steps of the algorithm were approximately solved with the fixed-point 
iteration specified in the proof of Proposition~\ref{prop:fixed_spin}. Our 
overall observation is that a few iterations are sufficient to 
decrease the $L^2$~difference of two iterates below the 
tolerance $\tau^2$. Nearly no variations of the discrete energies and
lengths of the vectors were observed.

\begin{figure}[p]
\includegraphics[width=4.3cm]{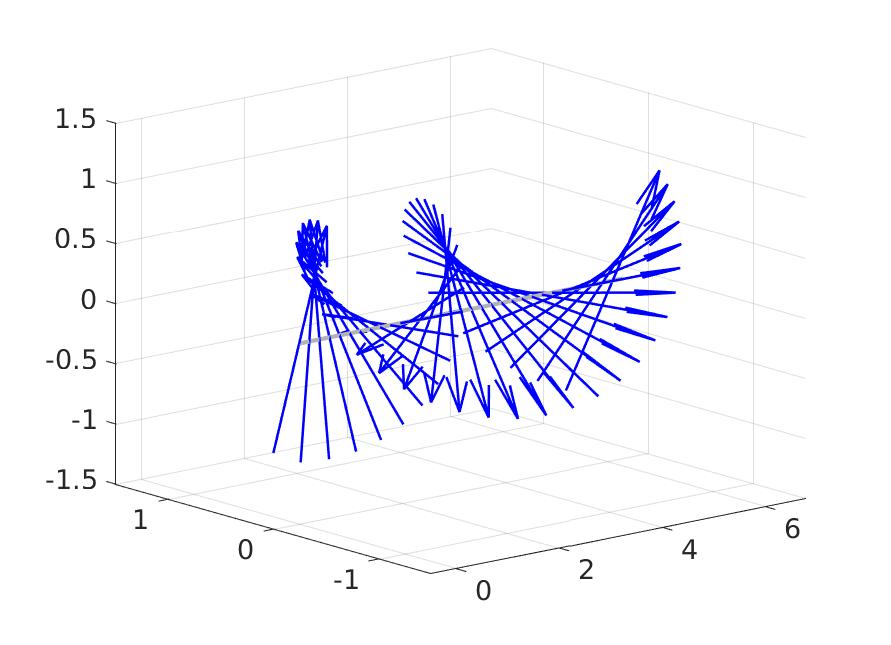}
\includegraphics[width=4.3cm]{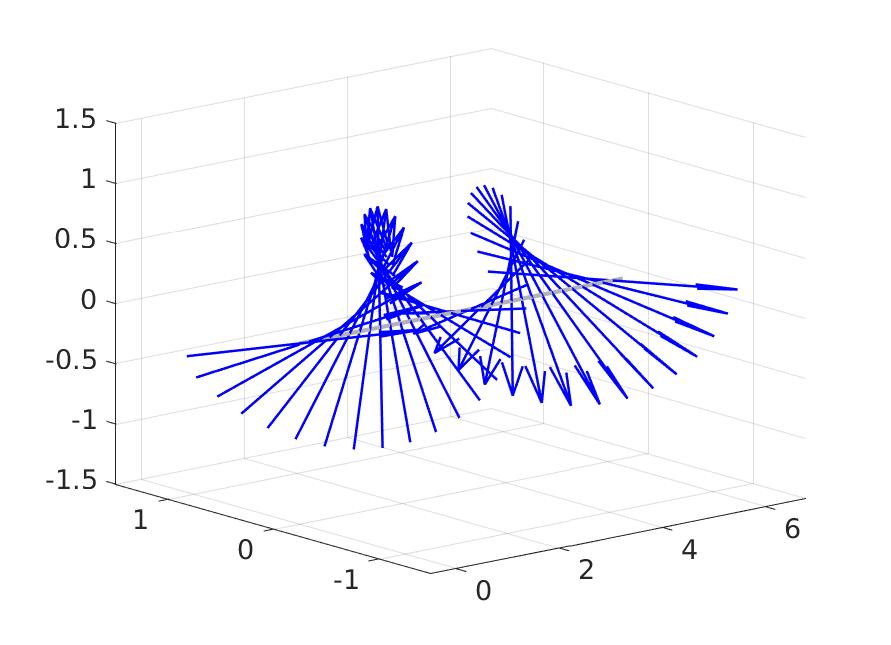}
\includegraphics[width=4.3cm]{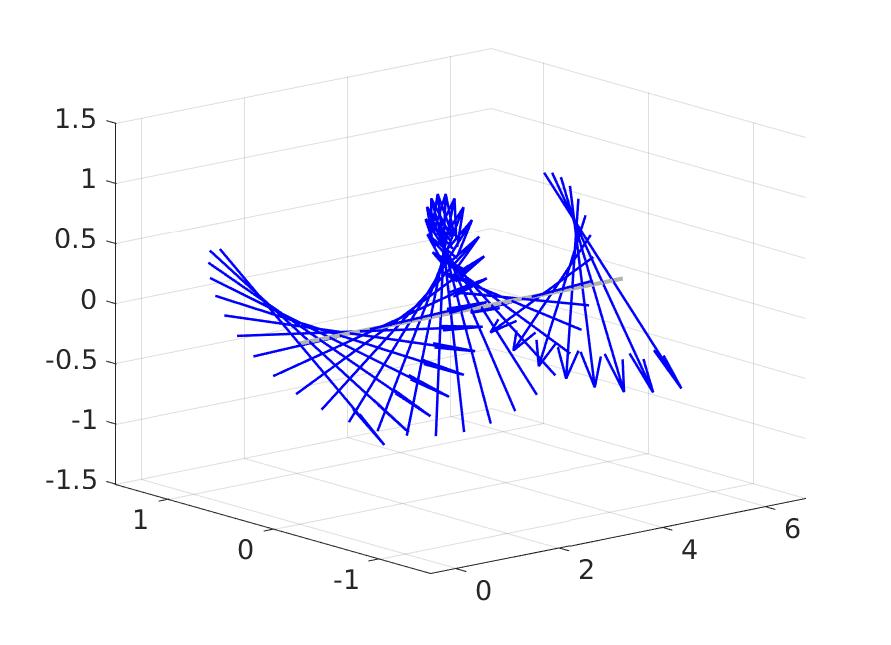} \\
\includegraphics[width=4.3cm]{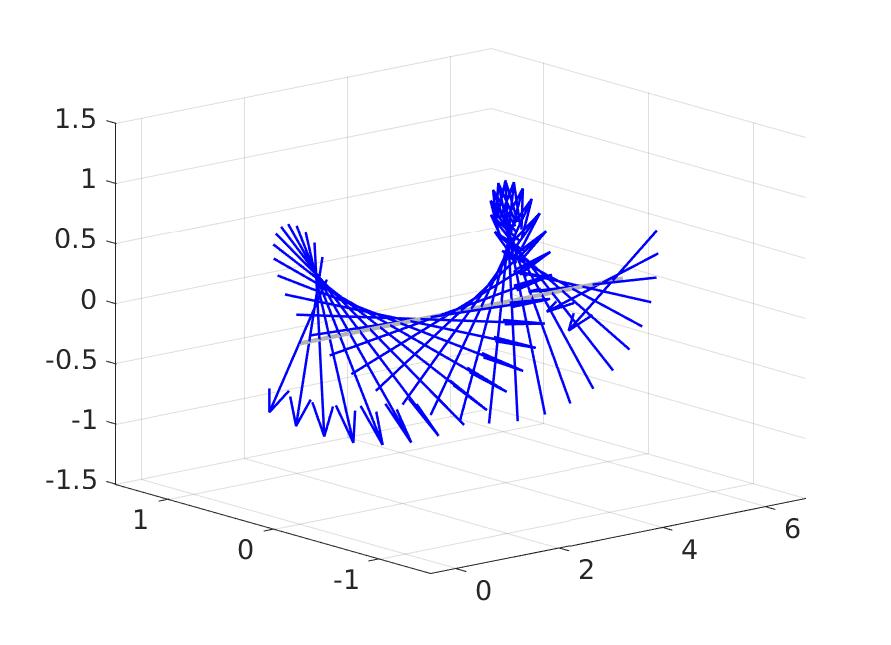}
\includegraphics[width=4.3cm]{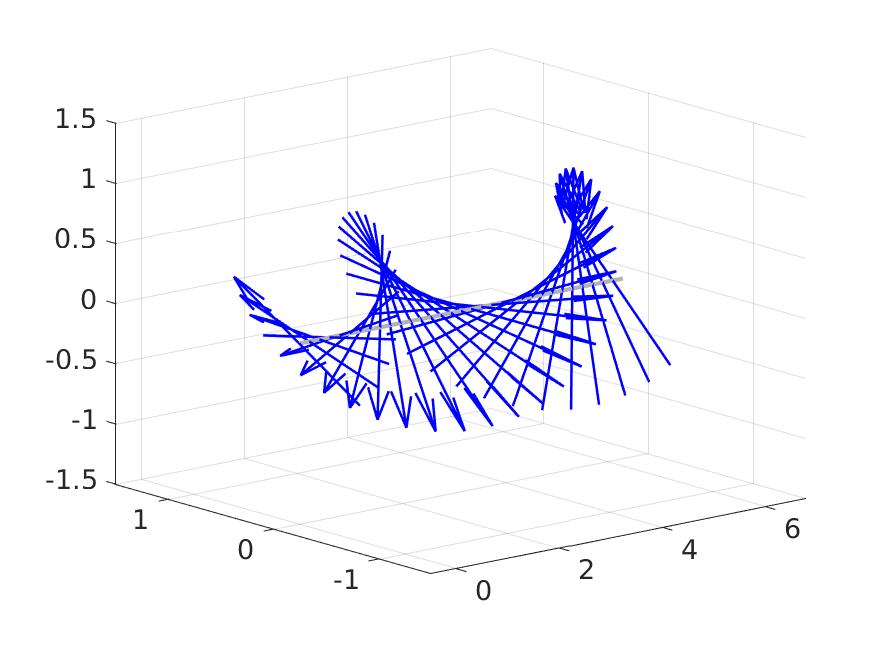}
\includegraphics[width=4.3cm]{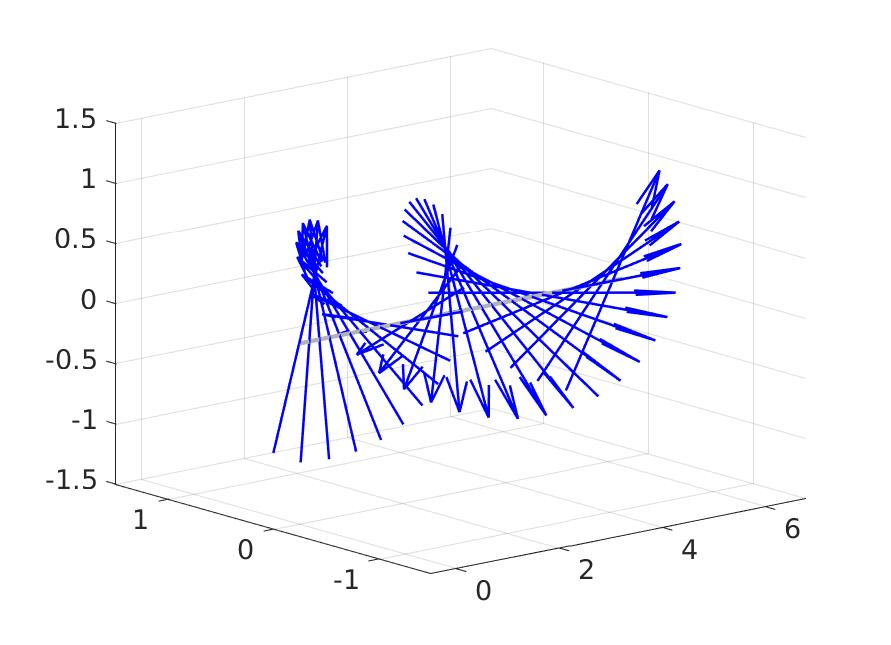} 
\caption{\label{fig:ex_blaschke_snaps} Snapshots of approximations of a 
solitary wave in a periodic spin chain for $t_\ell =(\ell/5) T$, 
$\ell = 0,1,\dots,5$ (left to right, top to bottom),
obtained with Algorithm~\ref{alg:precess_discr_2}
for $M=32$, $h=2\pi/M$, and $\tau = h/10$.}
\end{figure}

The initial data in the second experiment are a perturbation of a 
harmonic map. We let $\Pi_{S^2}:\R^3\setminus \{0\} \to S^2$ denote the 
orthogonal projection onto the unit sphere. 

\begin{example}\label{ex:perturbed_hm}
Let $s=1/2$, $T= 4$, and for $\xi:\T \to \R^3$ with 
$\|\xi\|_{L^\infty(\T)} \le 1/2$  define 
\[
u^0(x) = \Pi_{S^2} \big[\xi_1(x),\cos(x) + \xi_2(x), \sin(x) + \xi_3(x)\big]^\transp.
\]
\end{example}

We used a perturbation~$\xi$ satisfying $\|\xi\|_{L^\infty(\T)} \le 0.05$. 
Some iterates of the discrete evolution defined by the time-stepping
scheme of Algorithm~\ref{alg:precess_discr_2} are displayed in 
Figure~\ref{fig:sol_hm_pert}. Due to the presence of the perturbation
the solution oscillates between perturbations of the stationary states
$u_\pm(x) = \pm \big[0,\cos(x),\sin(x)\big]^\transp$. Because of the
less regular solution compared with the example considered above,
slightly more iterations are needed to solve the nonlinear systems
of equations and a corresponding moderately increased violation of 
the energy conservation property is observed. For the tested
discretizations with $M=64, 128, 256$, $h=2\pi/M$, and $\tau = h/10$,
these violations were smaller than $\d = 10^{-3}$ and decayed super 
linearly as $h\to 0$. 

\begin{figure}[p]
\includegraphics[width=4.3cm]{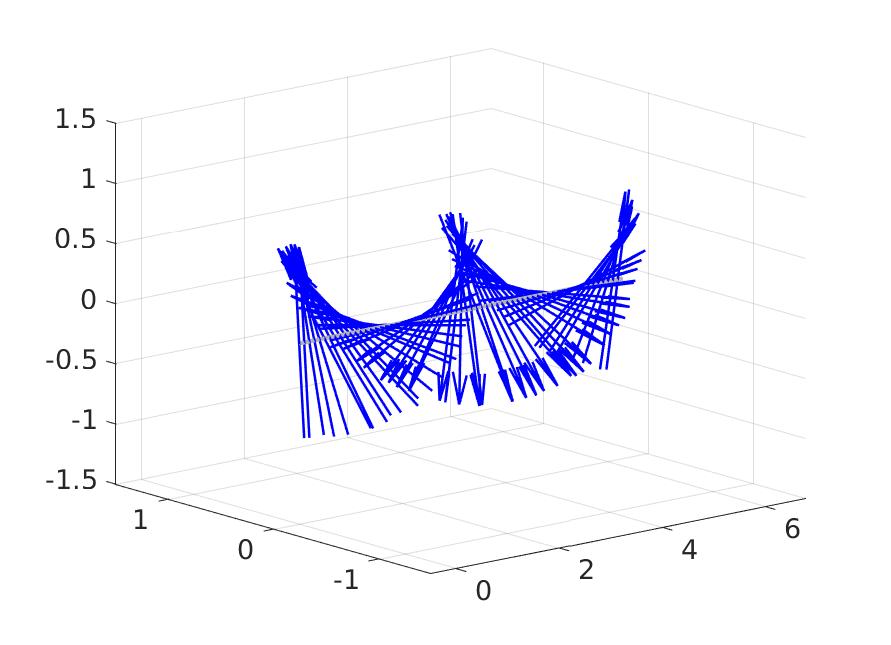}
\includegraphics[width=4.3cm]{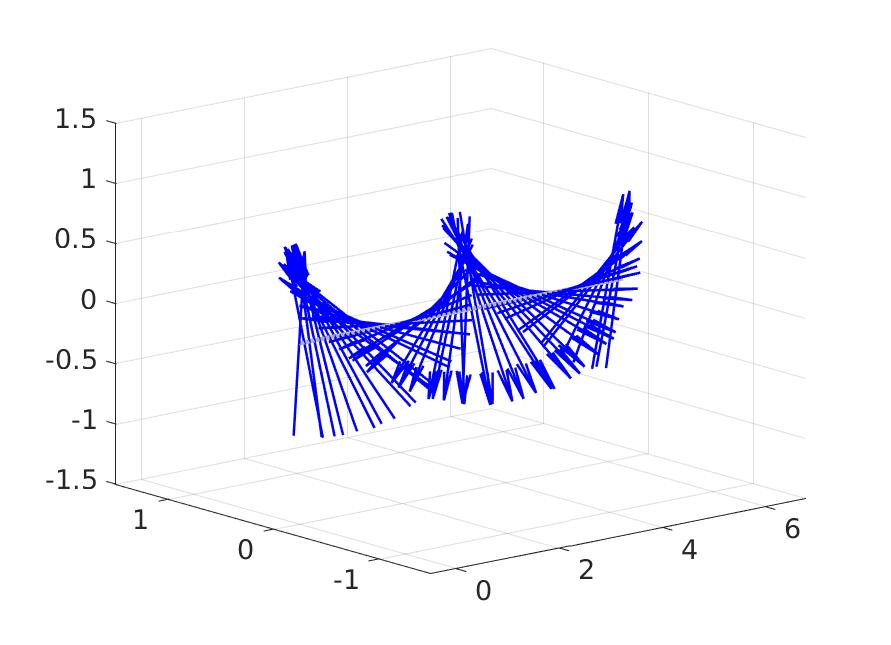} 
\includegraphics[width=4.3cm]{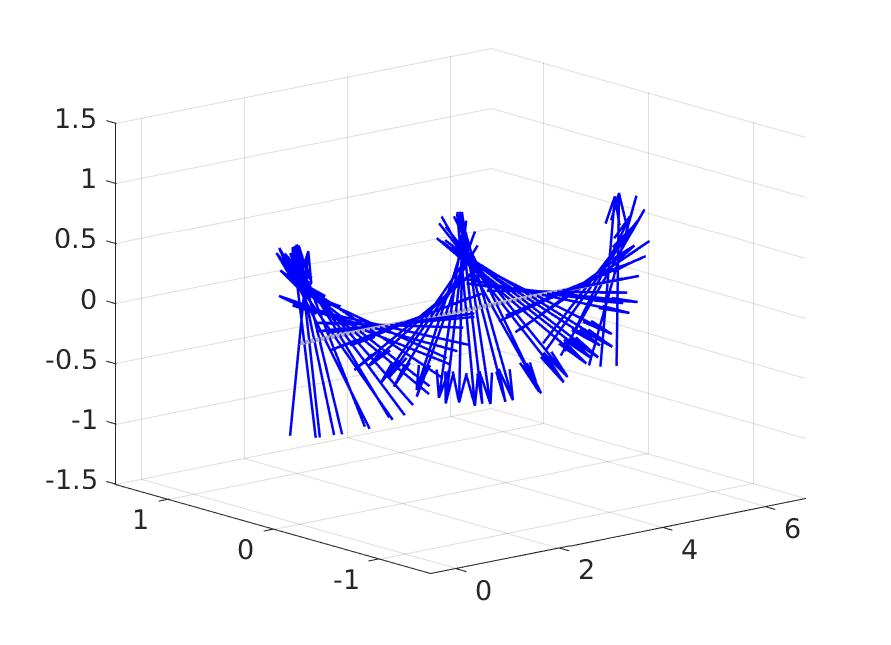} \\
\includegraphics[width=4.3cm]{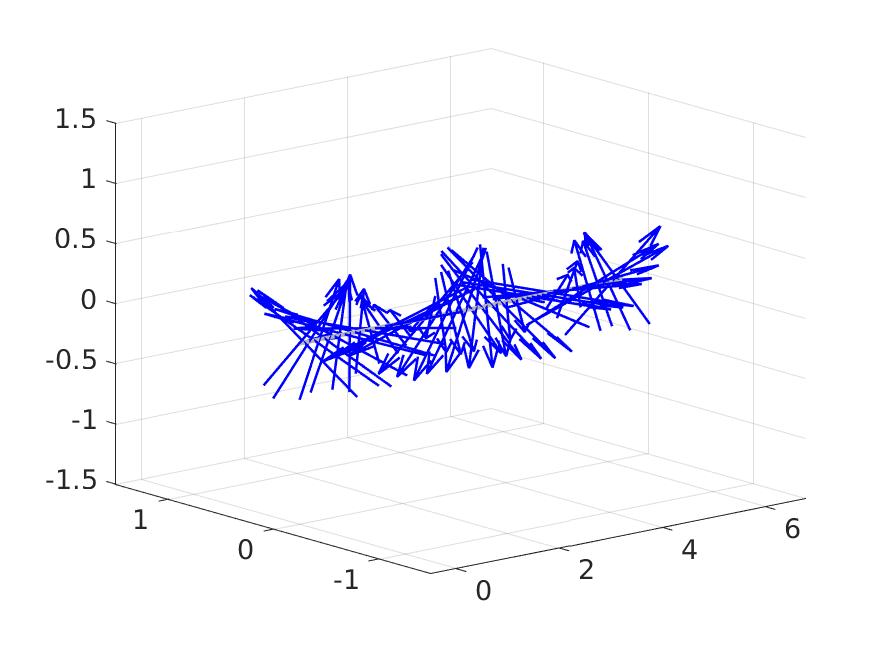} 
\includegraphics[width=4.3cm]{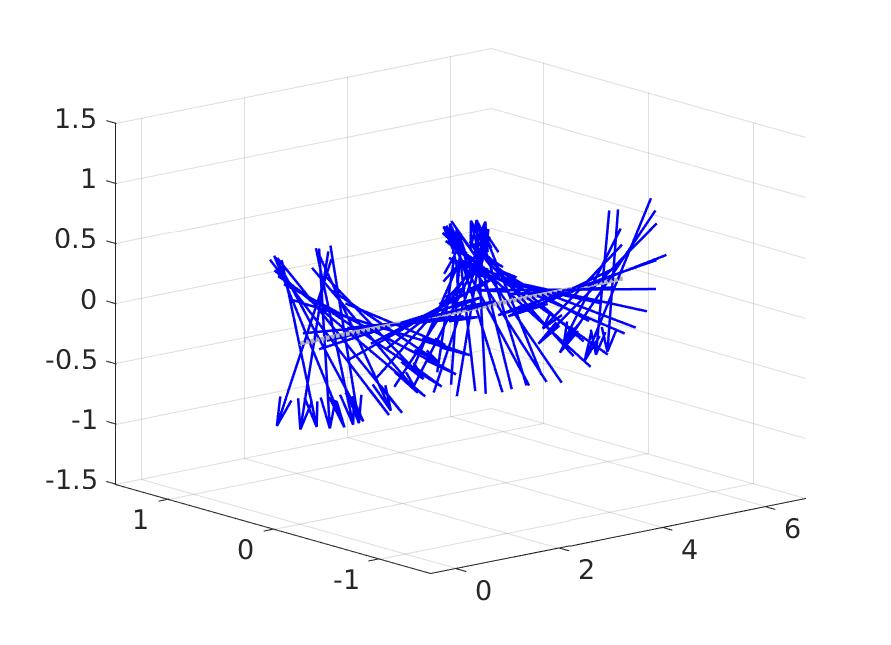} 
\includegraphics[width=4.3cm]{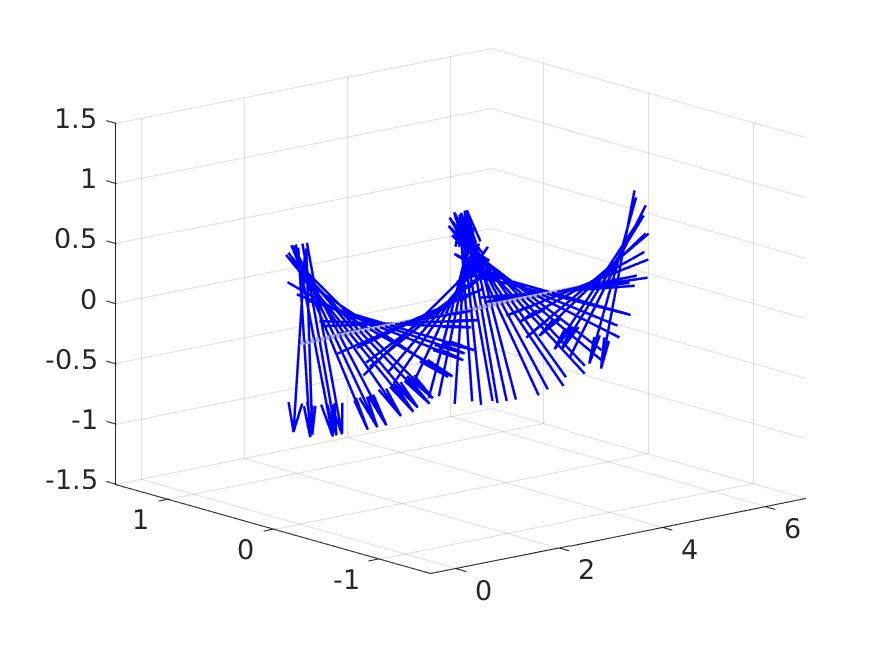} \\
\includegraphics[width=4.3cm]{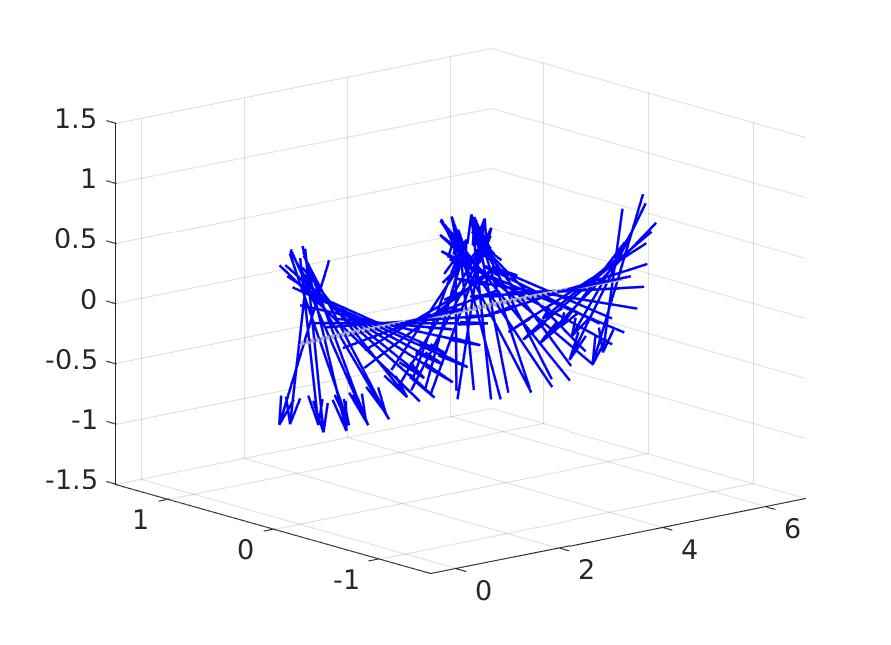} 
\includegraphics[width=4.3cm]{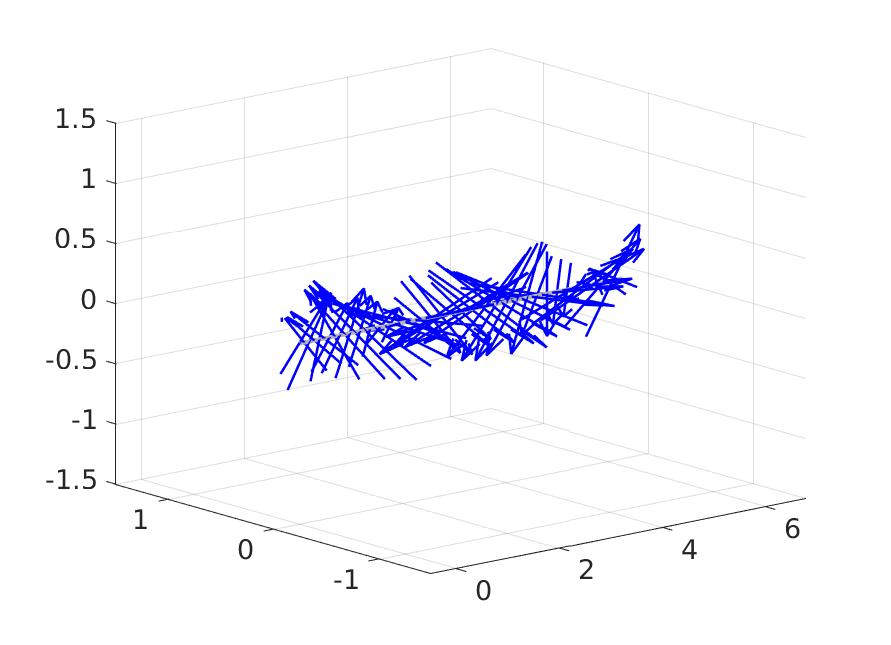} 
\includegraphics[width=4.3cm]{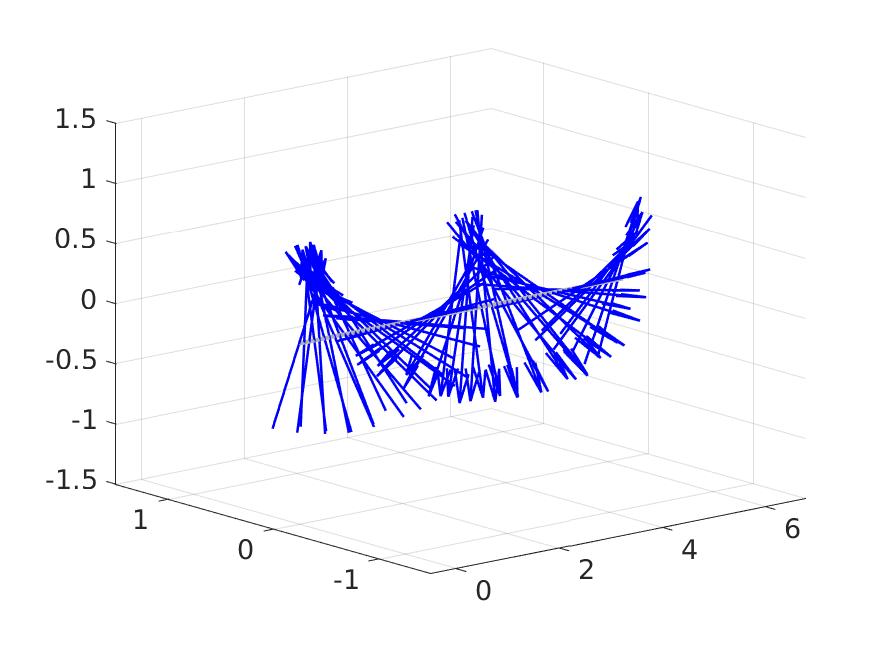} 
\caption{\label{fig:sol_hm_pert} Snapshots of approximations 
at $t_\ell = (\ell/10) T$, $\ell =0,1,\dots,8$ (left to right,
top to bottom), of an evolution resulting from a perturbed harmonic map as
initial data. The approximations oscillate between nearly
stationary states. The approximations were 
obtained with Algorithm~\ref{alg:precess_discr_2}
for $M=64$, $h=2\pi/M$, and $\tau = h/10$.}
\end{figure}


\subsection{Fractional harmonic map heat flow}
We next experimentally investigate the fractional harmonic map heat flow in
two-dimensional domains. We consider the integral fractional Laplacian as defined 
in \eqref{eq:integ}. For its discretization we follow 
\cite{GAcosta_FMBersetche_JPBorthagaray_2017a} and replace the unbounded domain
$\R^d$ by a bounded set $\tO$ with $\overline{\O}\subset \tO$, this defines a
discrete fractional Dirichlet energy $E_{s,h}$ and a corresponding bilinear form.
Our example enforces a singularity via smooth but topologically nontrivial boundary conditions which
is implemented via an additive decomposition of the unknown. 
Alternative approaches for imposing the exterior boundary condition 
are discussed in~\cite{HAntil_RKhatri_MWarma_2019a,HAntil_DVerma_MWarma_2020a,GAcosta_JPBorthagaray_NHeuer_2019a}.
The treatment of the linearized constraints follows~\cite[section~7.2.5]{Bart15-book}. 
In both examples below, as the exterior data we use a function $\vn$ with 
\[
\vn(x) = \frac{x}{|x|}
\]
for $x\in \p\O$. Our initial vector fields where obtained via normalizations of certain
random vectors at the inner nodes of the triangulations. We always used the step size
$\tau = 2h$ and as stopping criterion the condition
$\|d_t u_h^k\|_{s,h} < 10^{-6}$. 

\begin{example}\label{ex:defect}
We let $d = N= 2$, and consider the square $\Omega = (-0.5,0.5)^2$ or the disk $\O = B_{0.5}$.
As extended domain $\widetilde\Omega$ we choose a ball of radius $r=1.5$ centered at the origin.
We use an unstructured 
triangulation for $\widetilde\O$ generated using the package~Gmsh \cite{geuzaine2009gmsh}
which extends an unstructrued triangulation of $\O$. 
\end{example}

Figure~\ref{f:energy} displays the discrete energies $E_{s,h}[u_h^k]$, $k=0,1,\dots,K$, 
of the iterates $u_h^k\in \cS^1(\cT_h)^N$ for different fractional parameters $0<s<1$ 
and fixed mesh size $h=0.025$. The results confirm the theoretically established
energy decay property of Algorithm~\ref{alg:l2_flow_discr}. In particular,
a rapid initial energy decay is followed by a slower further reduction of
the energy before the process becomes nearly stationary.   
Figure~\ref{f:u_02_1566} illustrates a corresponding discrete evolution for the case $s=0.2$
and the mesh size $h=0.025$ via snapshots of the iterates provided by Algorithm~\ref{alg:l2_flow_discr}.
We observe that the initial discontinuity of the initial function along parts of the 
boundary is quickly removed and a slightly diffused point defect develops which 
moves towards the center of the domain during the evolution. 

Figures~\ref{f:u} and~\ref{f:u_circ} show nearly stationary configurations $u_h$, i.e.,  
nearly discrete fractional harmonic maps on the square and on the disk, for the values $s=0.2$, $s=0.4$, and 
$s = 0.6$, as well as decreasing mesh sizes $h = 0.048$, $h= 0.033$, and $h=0.025$ in the
case of the square. For larger values 
of $s$, we clearly observe well localized point defects. For the choice $s = 0.2$ we find that the defect
is smeared out over a neighborhood of the origin in which the numerical solution is irregular 
and whose diameter appears to decay to zero 
as $h \rightarrow 0$. Owing to limitations in the spatial resolution and the occurring topological singularity 
we are unable to identify an experimental convergence behavior to the canonical solution
candidate $u(x) = x/|x|$, cf.~\cite{Lin87}. 
However, for both cases of domains we obtain numerical solutions that appear to be
very close to this vector field.

\begin{figure}[p]		
	\includegraphics[width=0.38\textwidth]{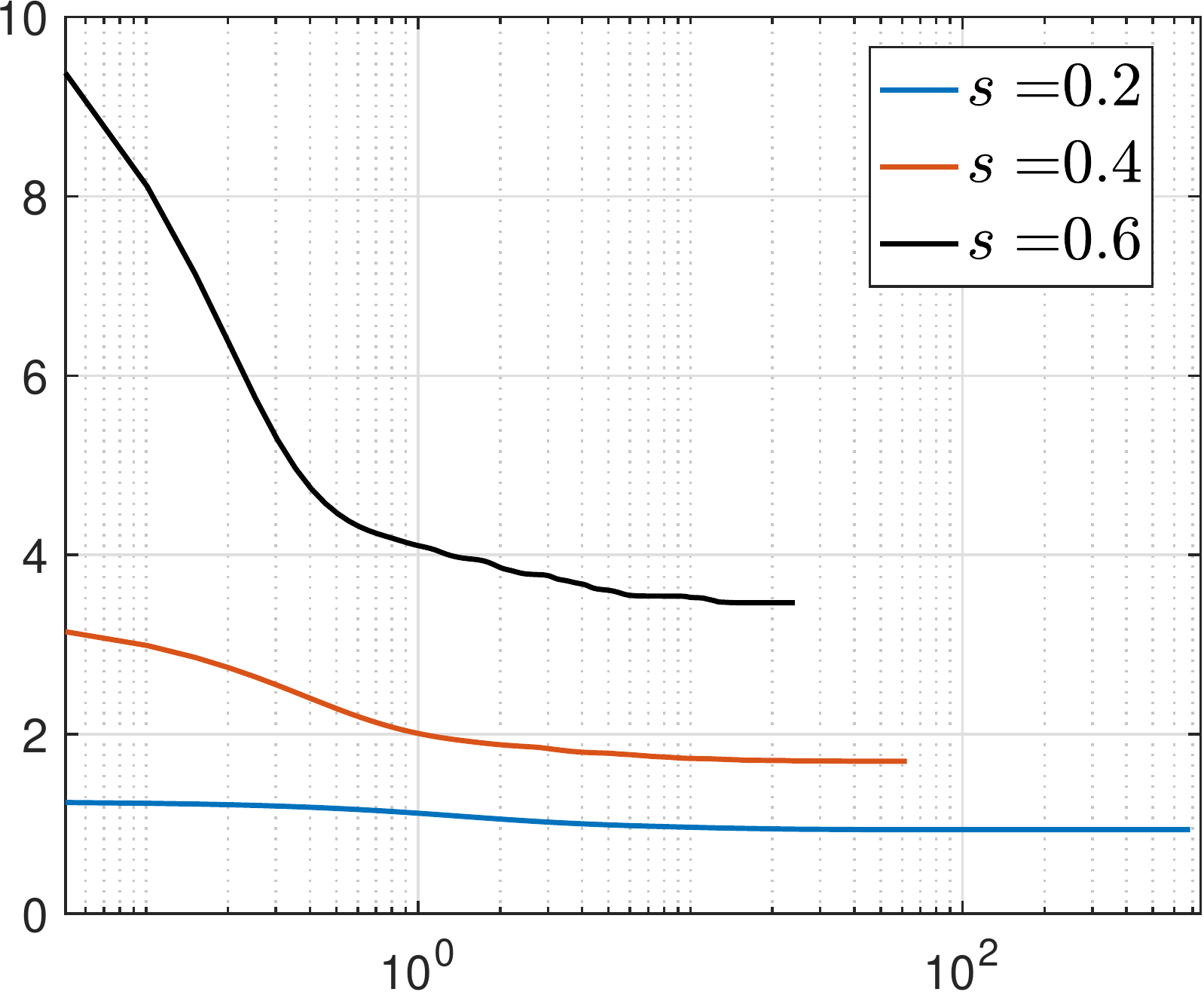}	
	\caption{\label{f:energy} Discrete fractional energies $E_{s,h}[u_h^k]$, $k=0,1,2, \dots$, with
	respect to the time~$t^k$, for a fixed mesh with meshsize $h=0.025$,
        and different values of $s$ in Example~\ref{ex:defect} with $\O= (0.5,0.5)^2$.
        For all values of $s$ an energy decay property is confirmed.}
\end{figure}

\begin{figure}[p]	
	\begin{tabular}{ccc} \medskip 
	\includegraphics[width=0.21\textwidth]{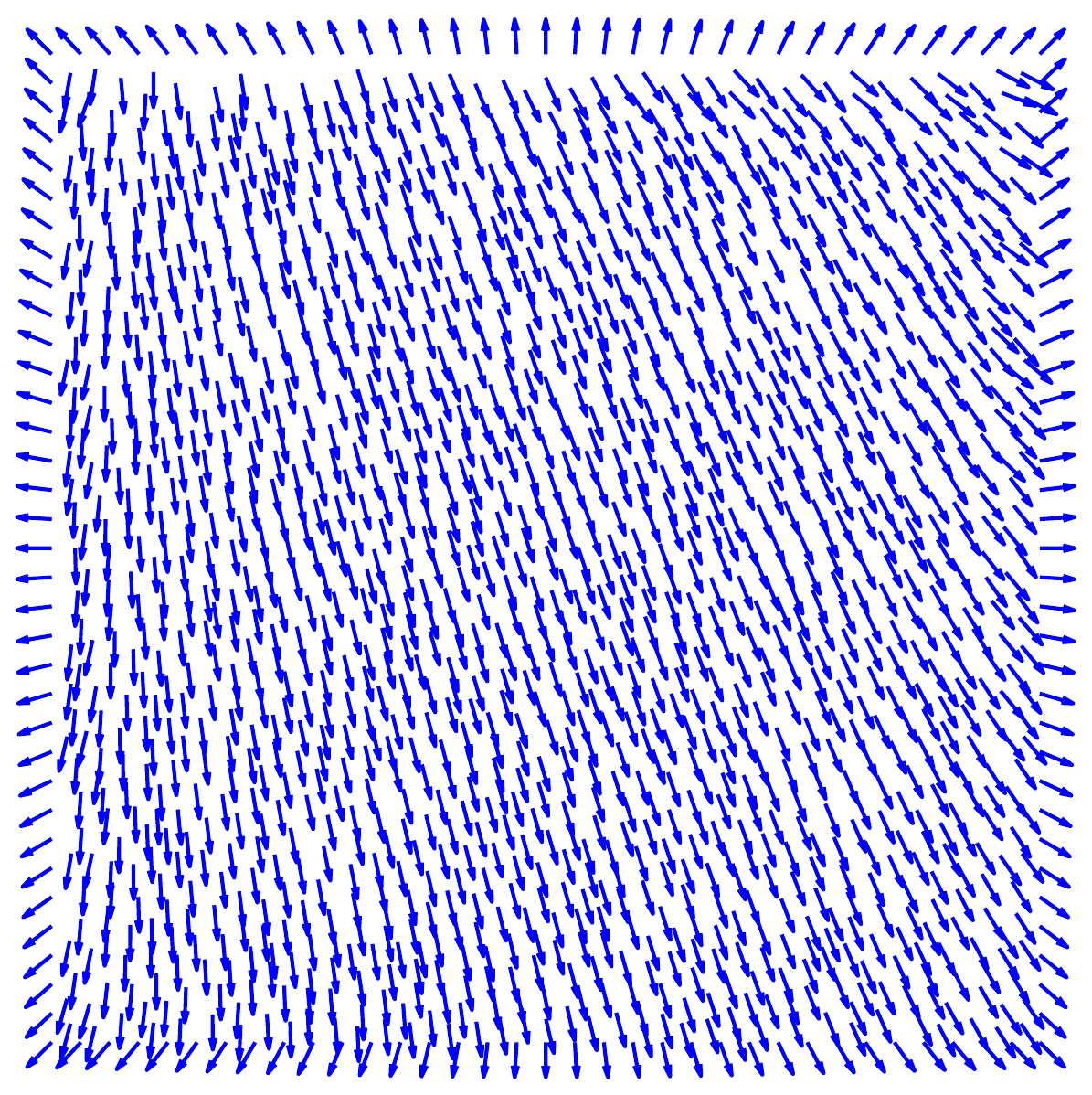} &
	\includegraphics[width=0.21\textwidth]{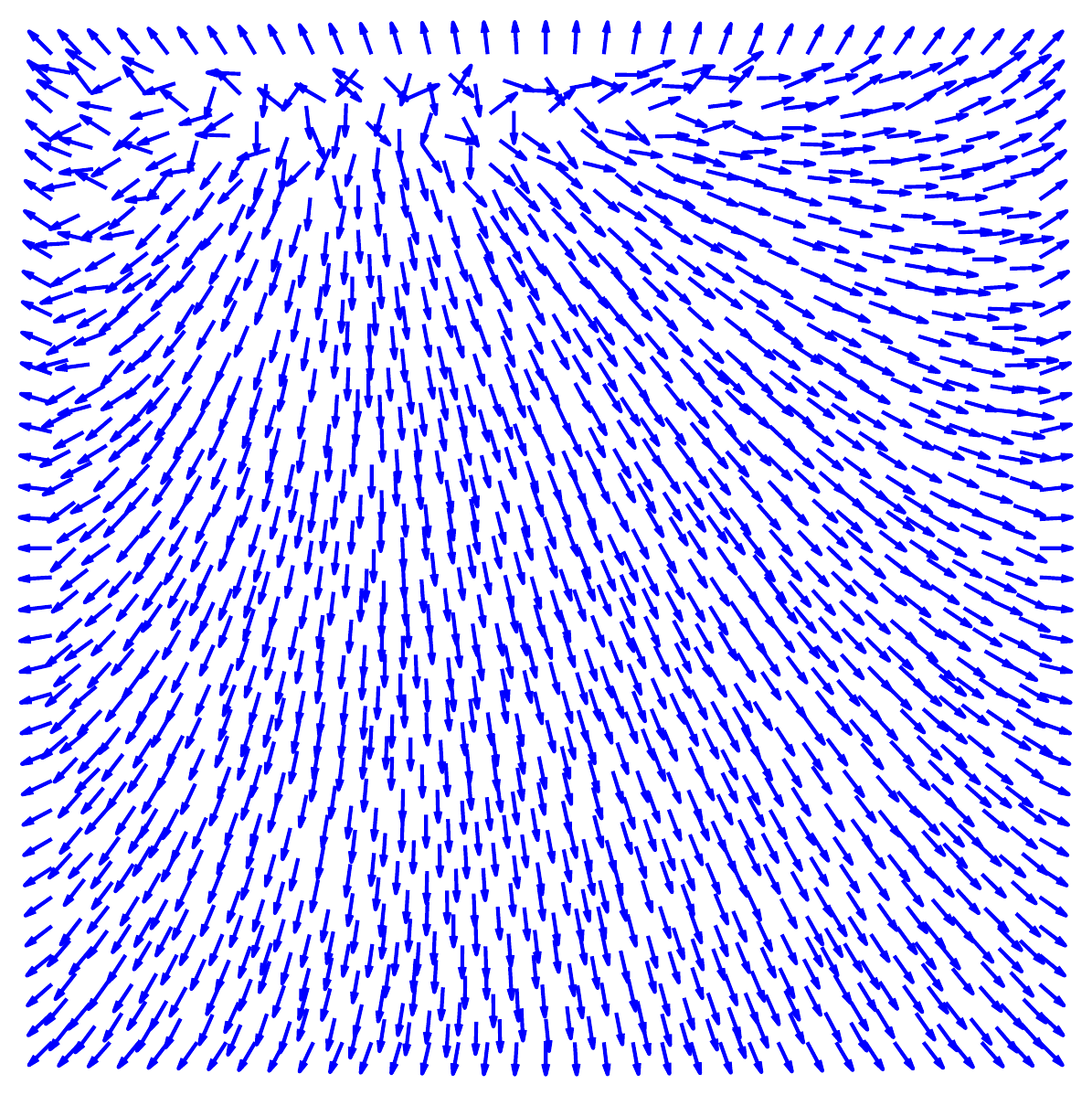} &
	\includegraphics[width=0.21\textwidth]{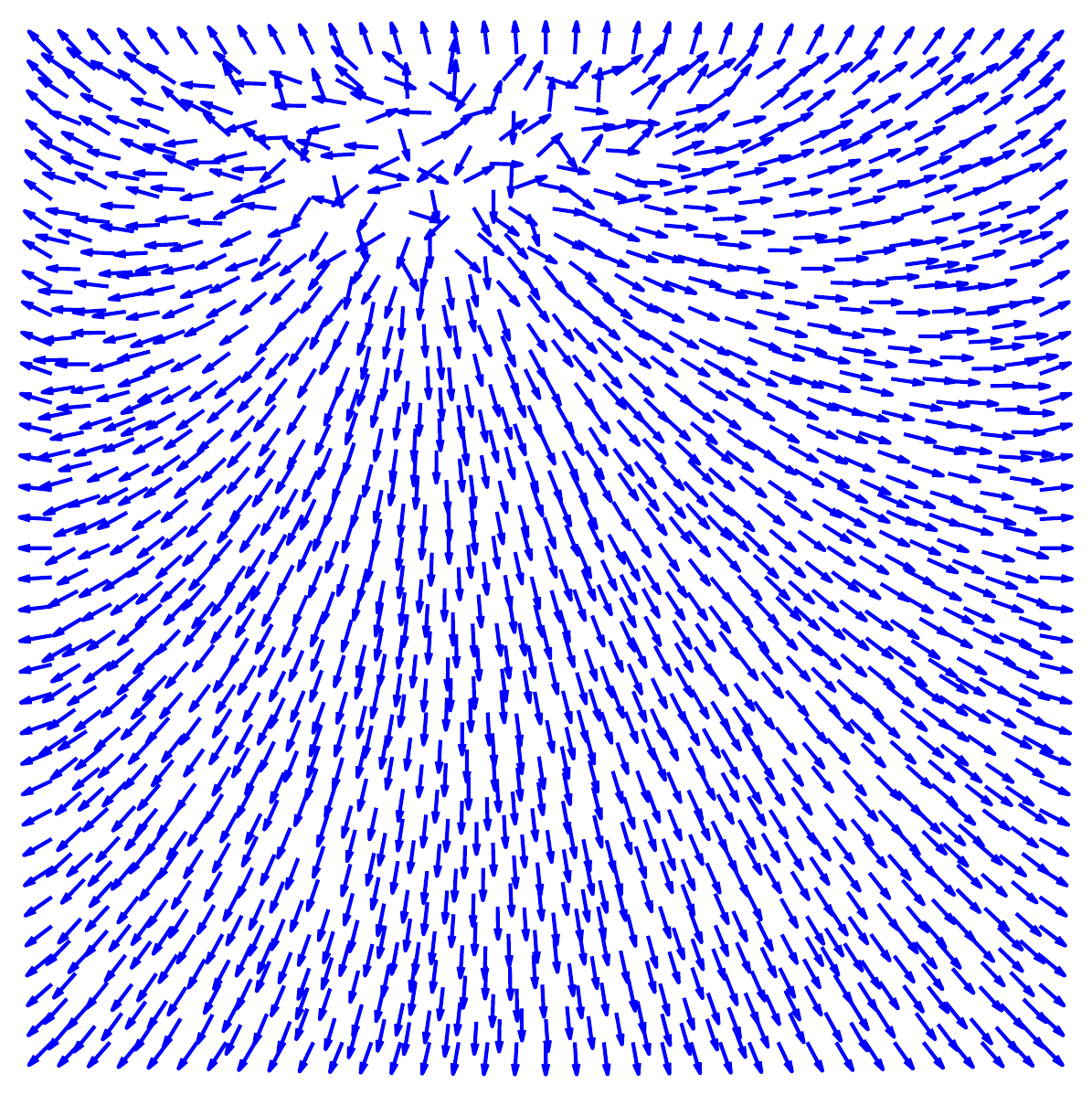} \\ \medskip 
	\includegraphics[width=0.21\textwidth]{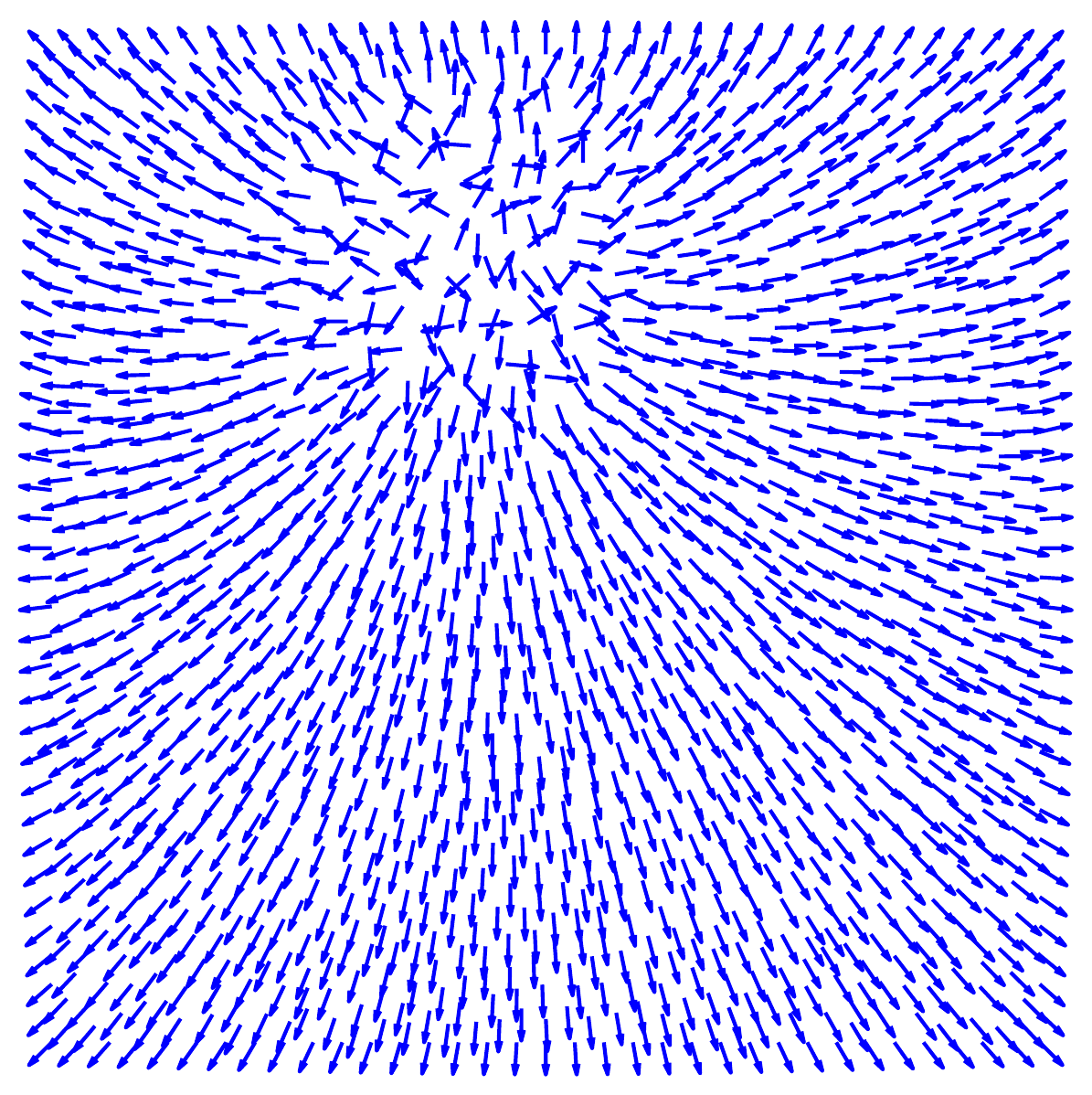}	&
	\includegraphics[width=0.21\textwidth]{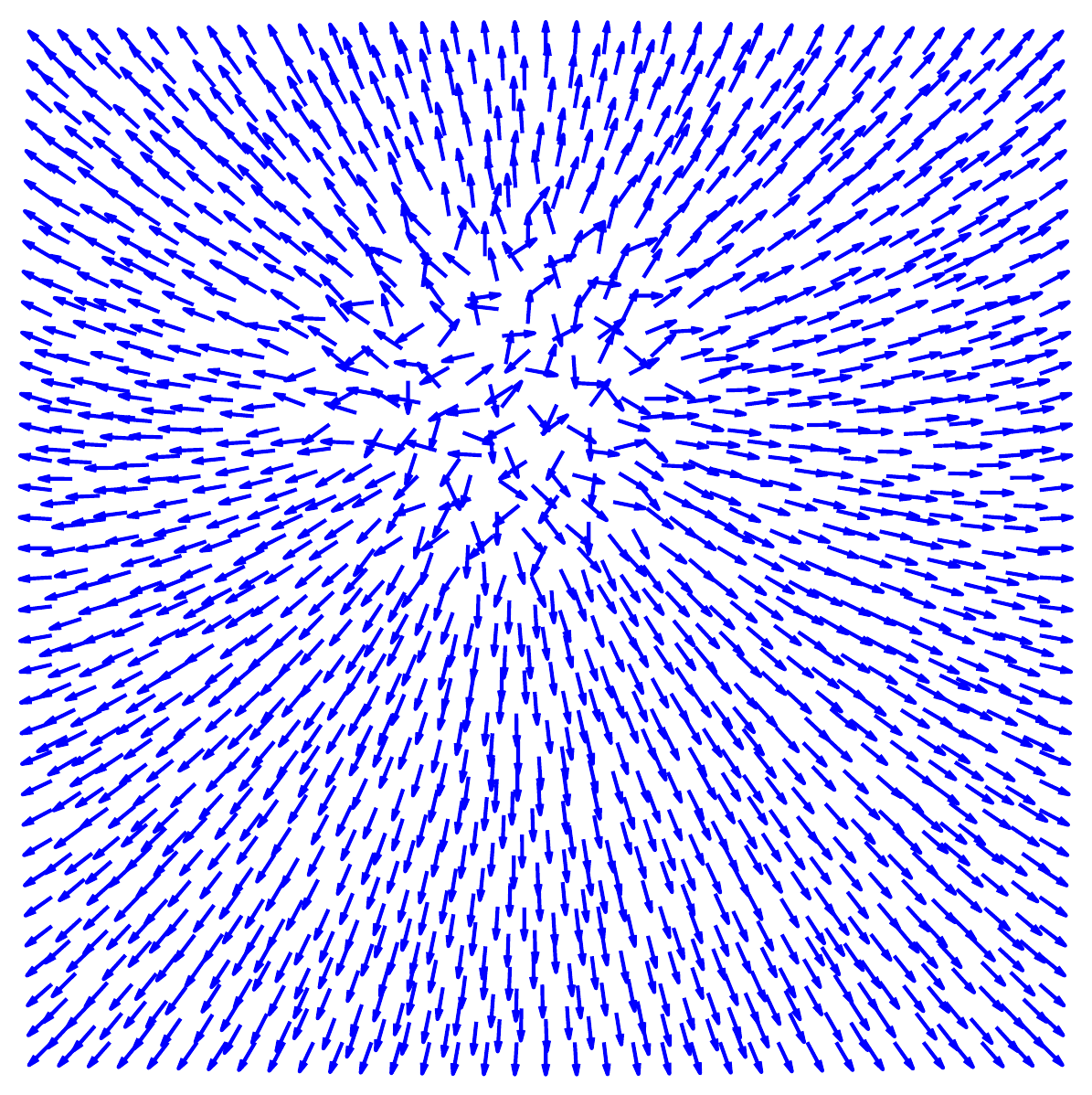} &
	\includegraphics[width=0.21\textwidth]{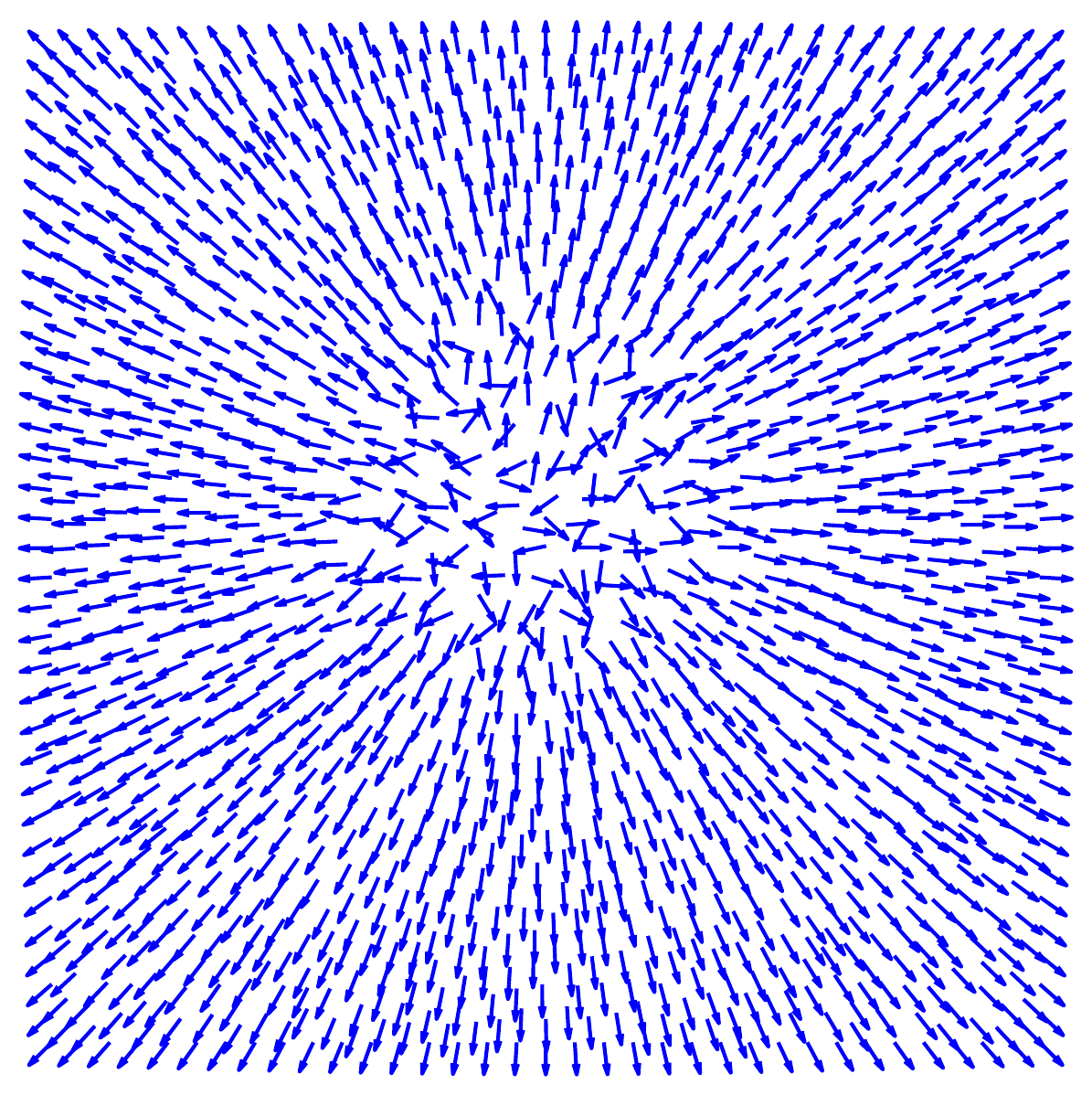} \\ 
	\includegraphics[width=0.21\textwidth]{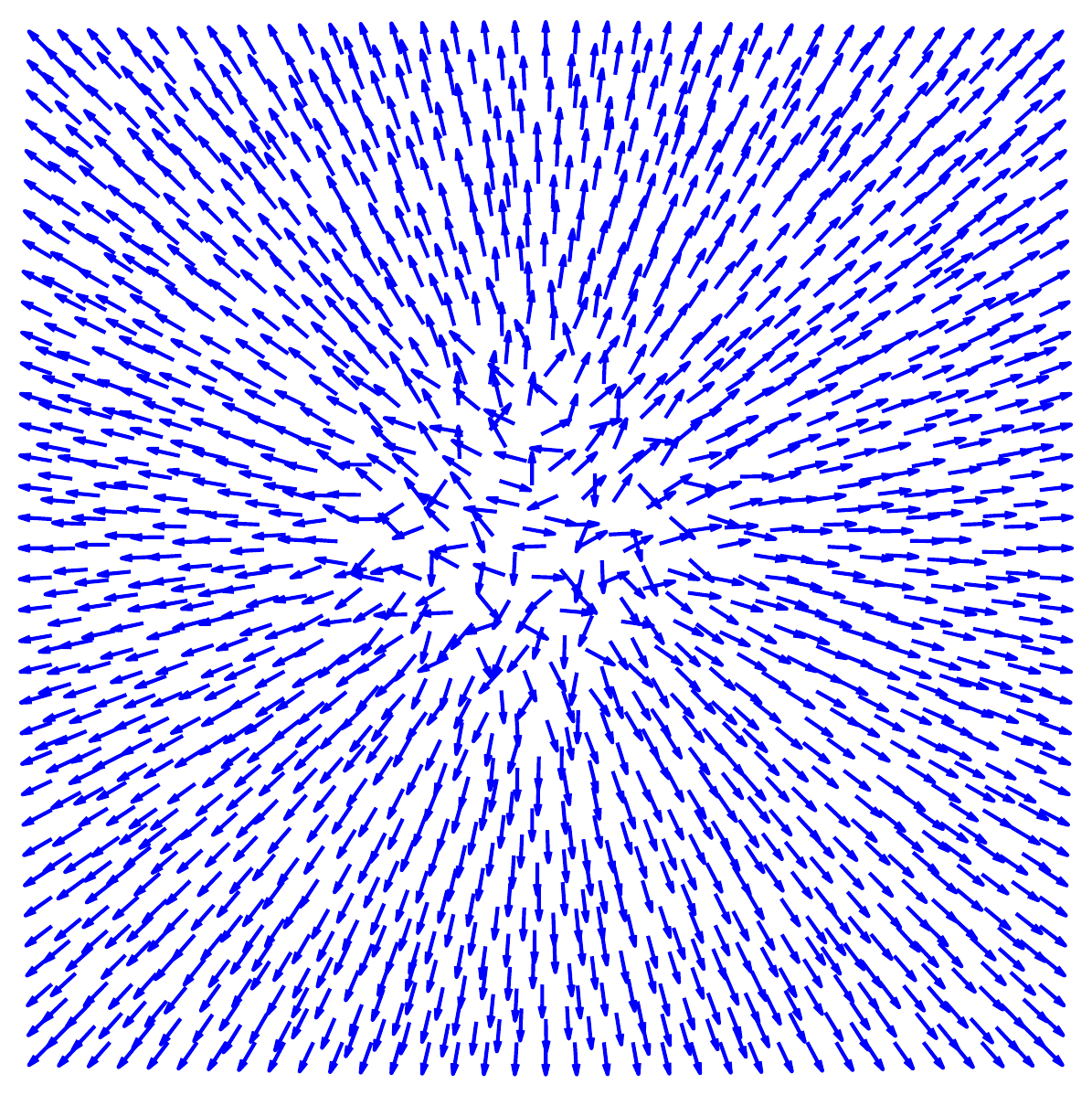}	&
	\includegraphics[width=0.21\textwidth]{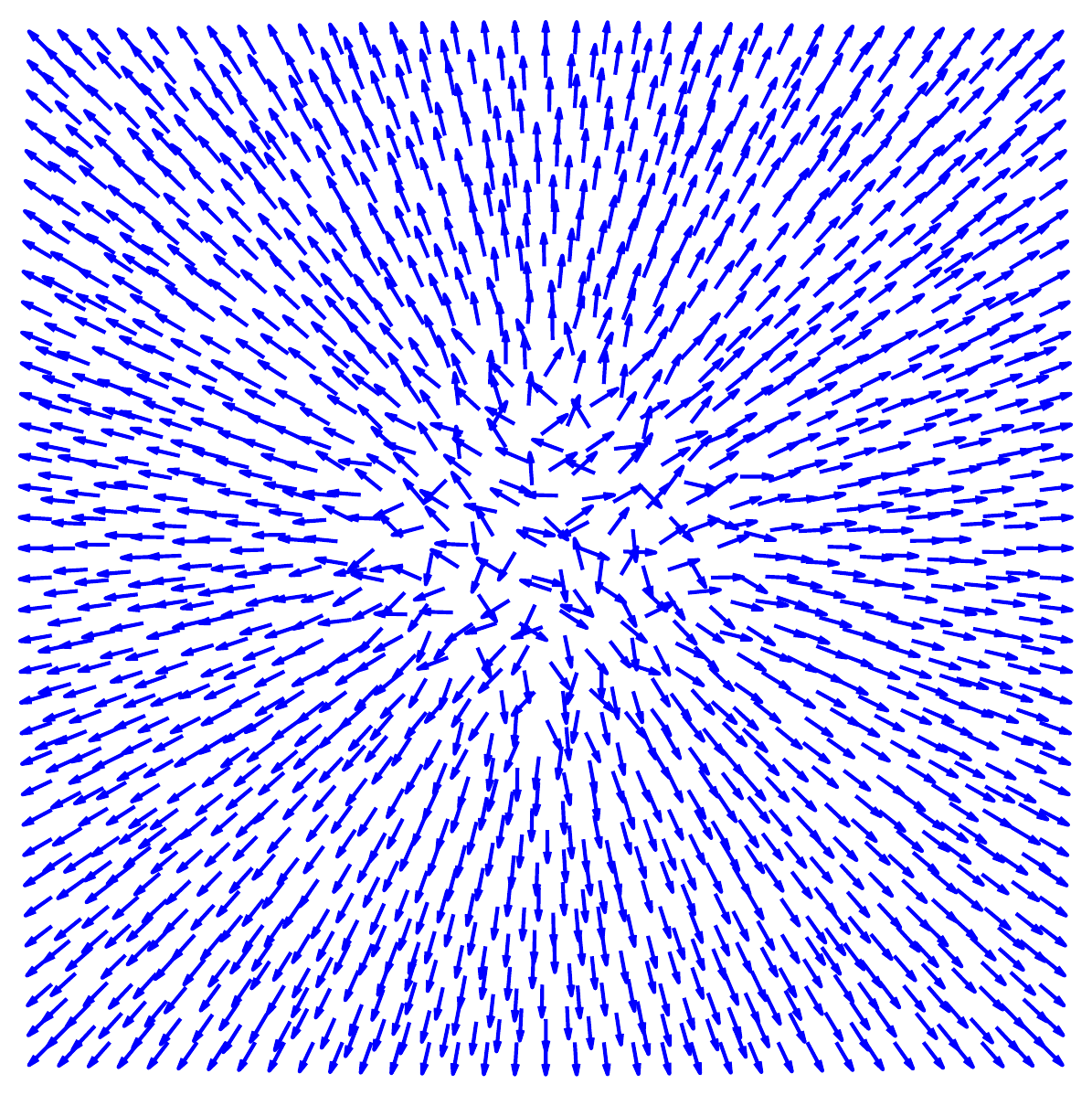}  &
	\includegraphics[width=0.21\textwidth]{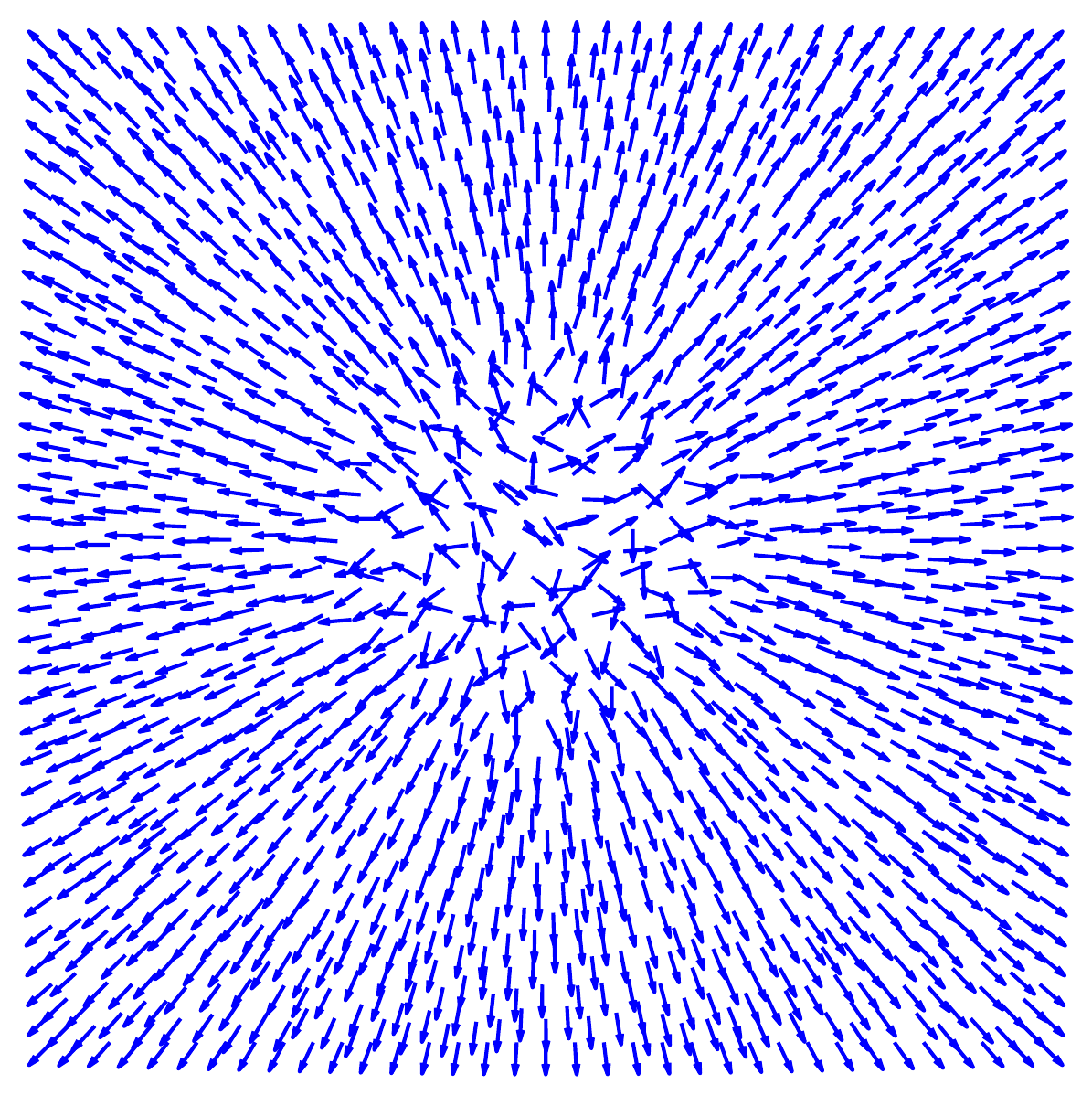} 	\\ 
	\end{tabular}
	\caption{\label{f:u_02_1566}  
        Snapshots of a discrete fractional harmonic map heat flow with $s=0.2$ and $h=0.025$ via approximations
        $u_h^k$ for $t^k=0.51, 3.03, 6.07, 12.13, 22.74, 45.49, 253.21, 510.96, 685.32$ (left to right, top to bottom) in Example~\ref{ex:defect}.
        An initial
        discontinuity at the boundary is regularized and the formation of a point defect that moves 
        to the origin is observed. 
	}
\end{figure}

\begin{figure}[p]	
	\begin{tabular}{ccc} \medskip 
	\includegraphics[width=0.21\textwidth]{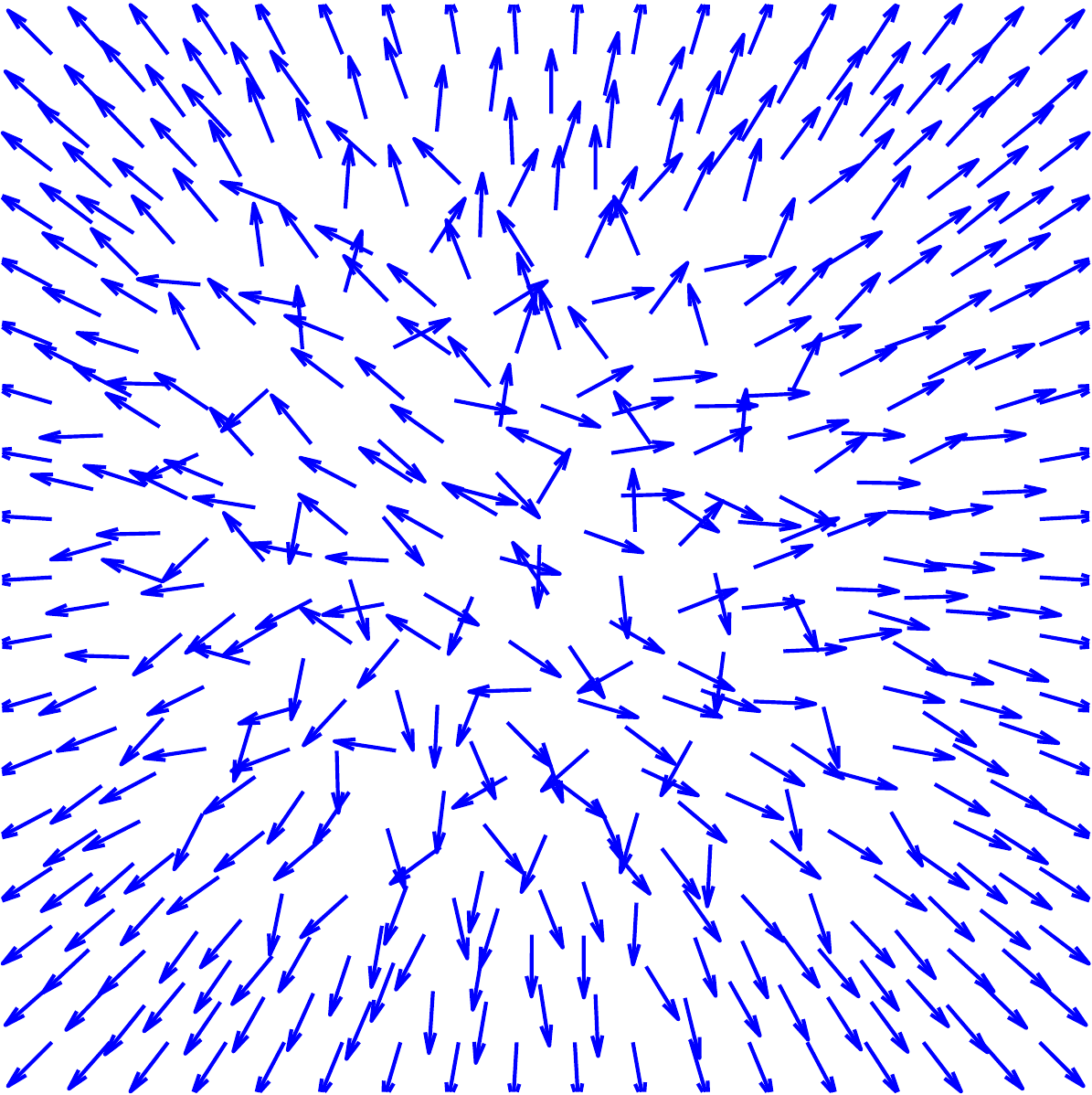} &
	\includegraphics[width=0.21\textwidth]{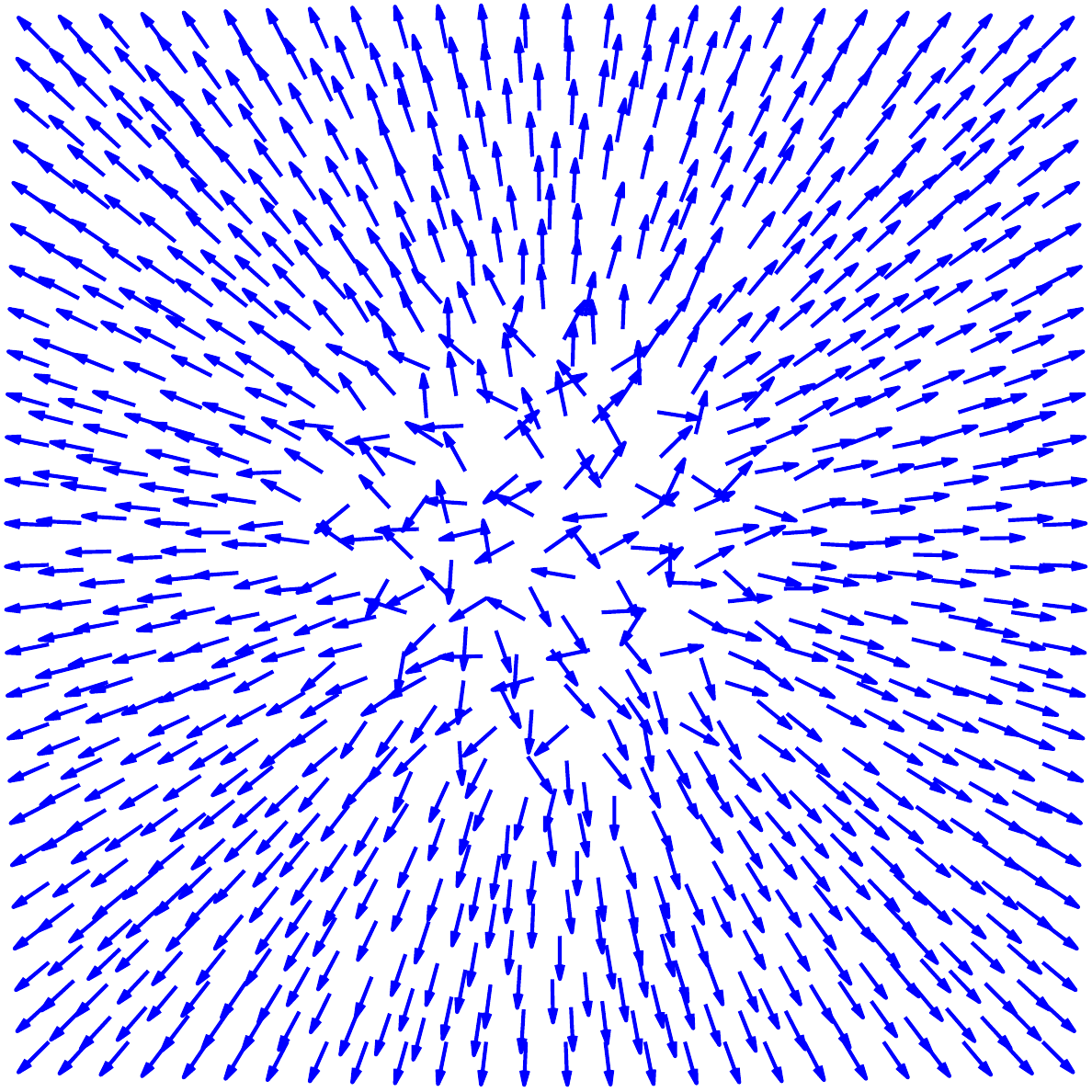} &
	\includegraphics[width=0.21\textwidth]{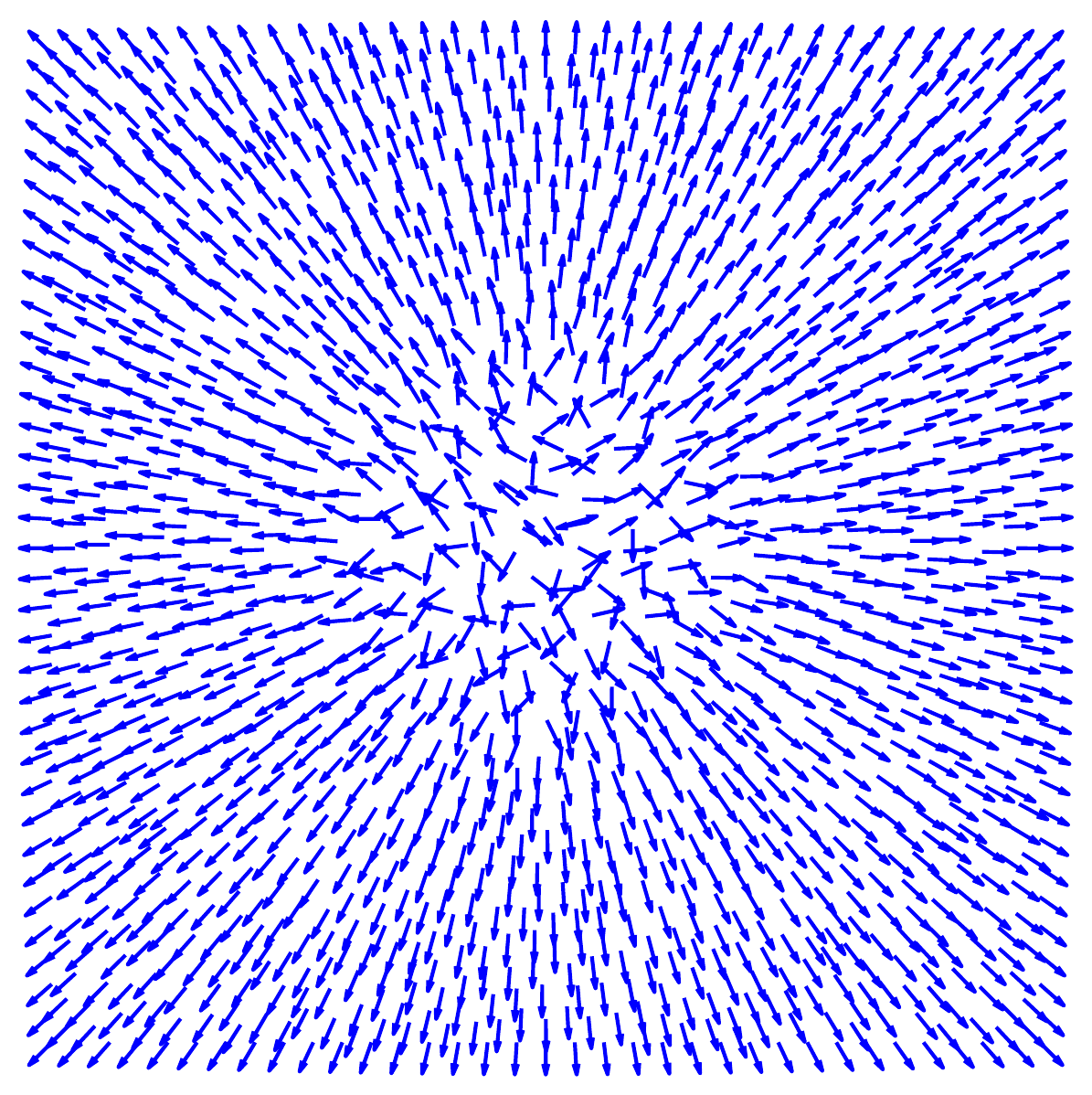} \\ \medskip 
	\includegraphics[width=0.21\textwidth]{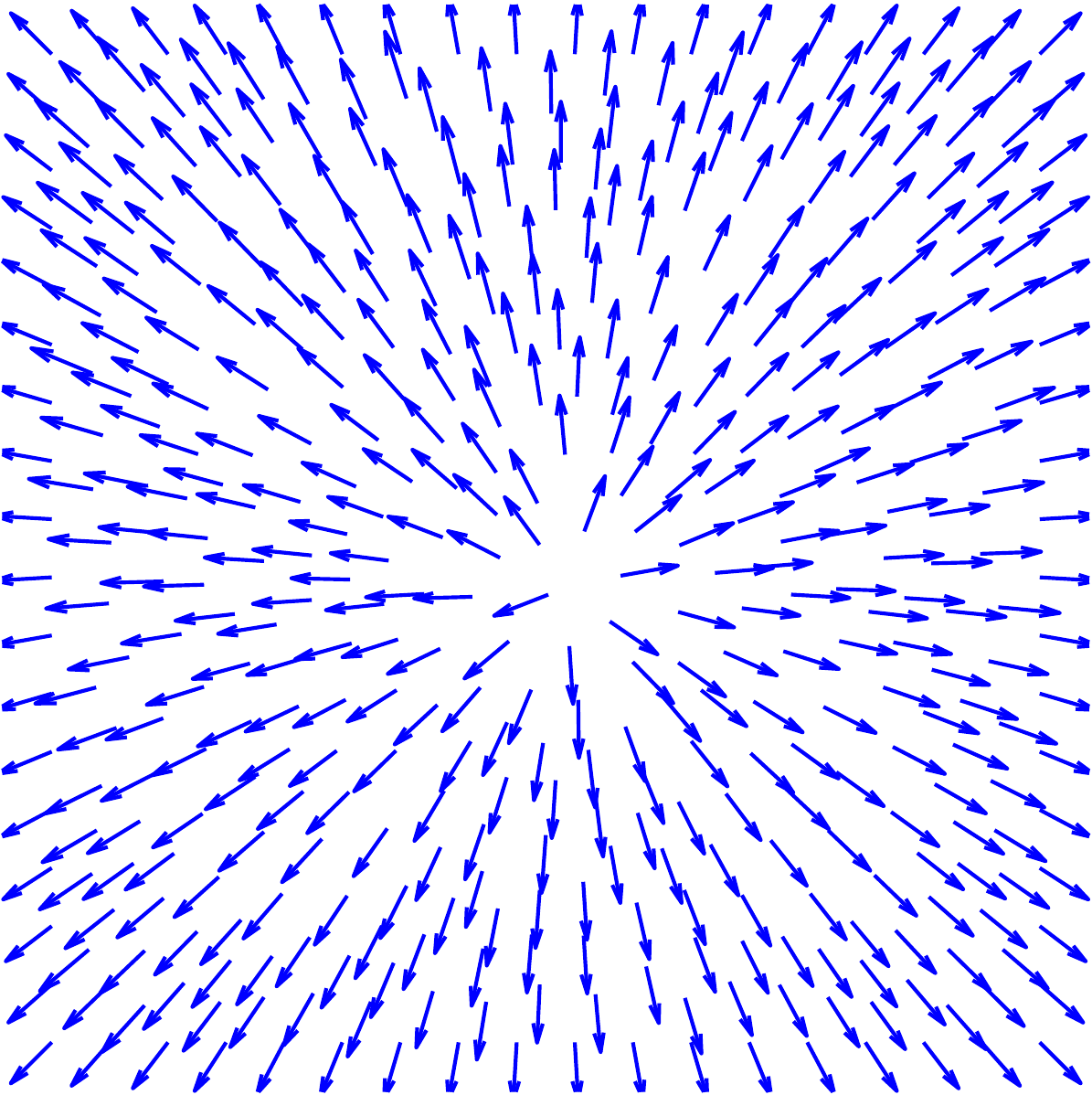}	&
	\includegraphics[width=0.21\textwidth]{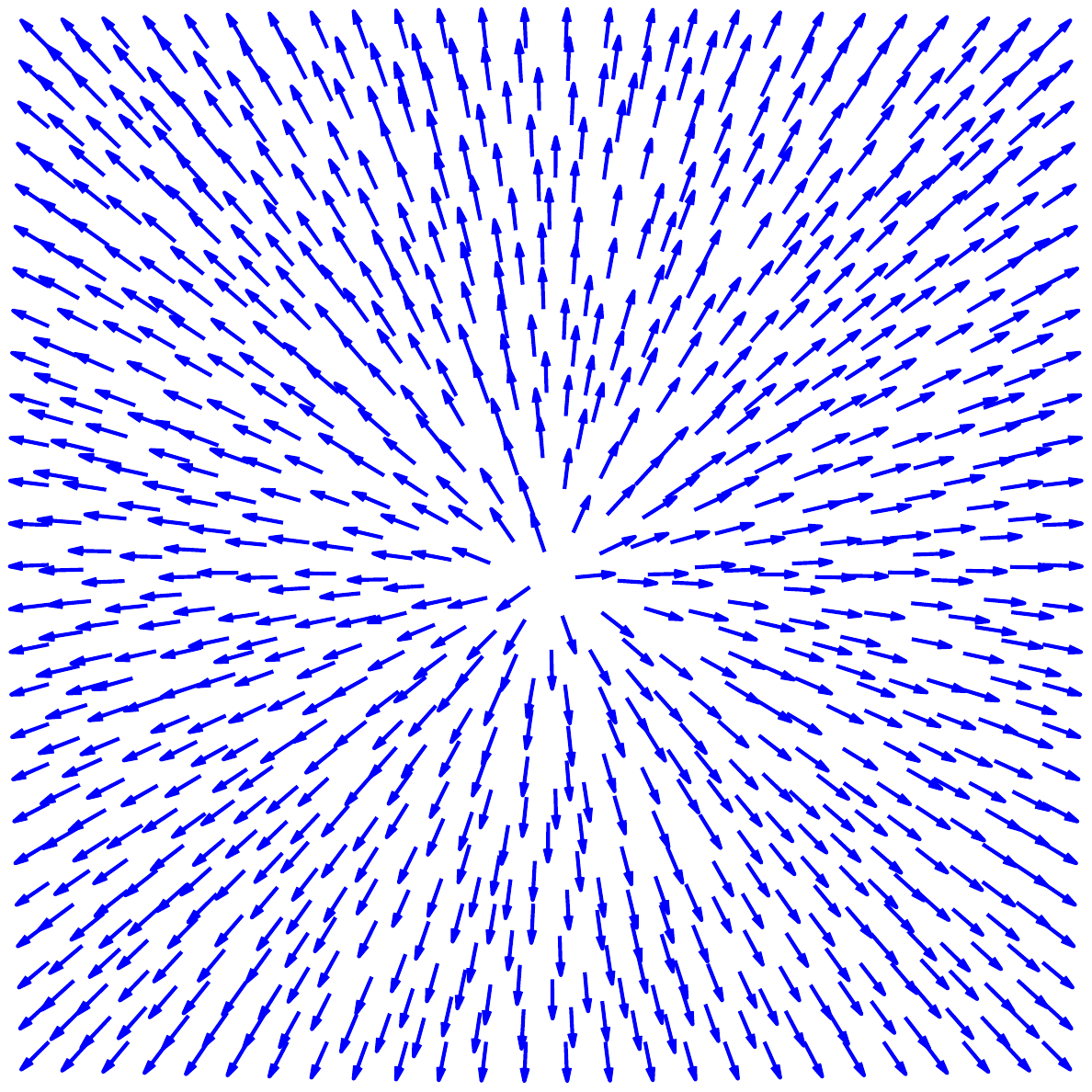} &
	\includegraphics[width=0.21\textwidth]{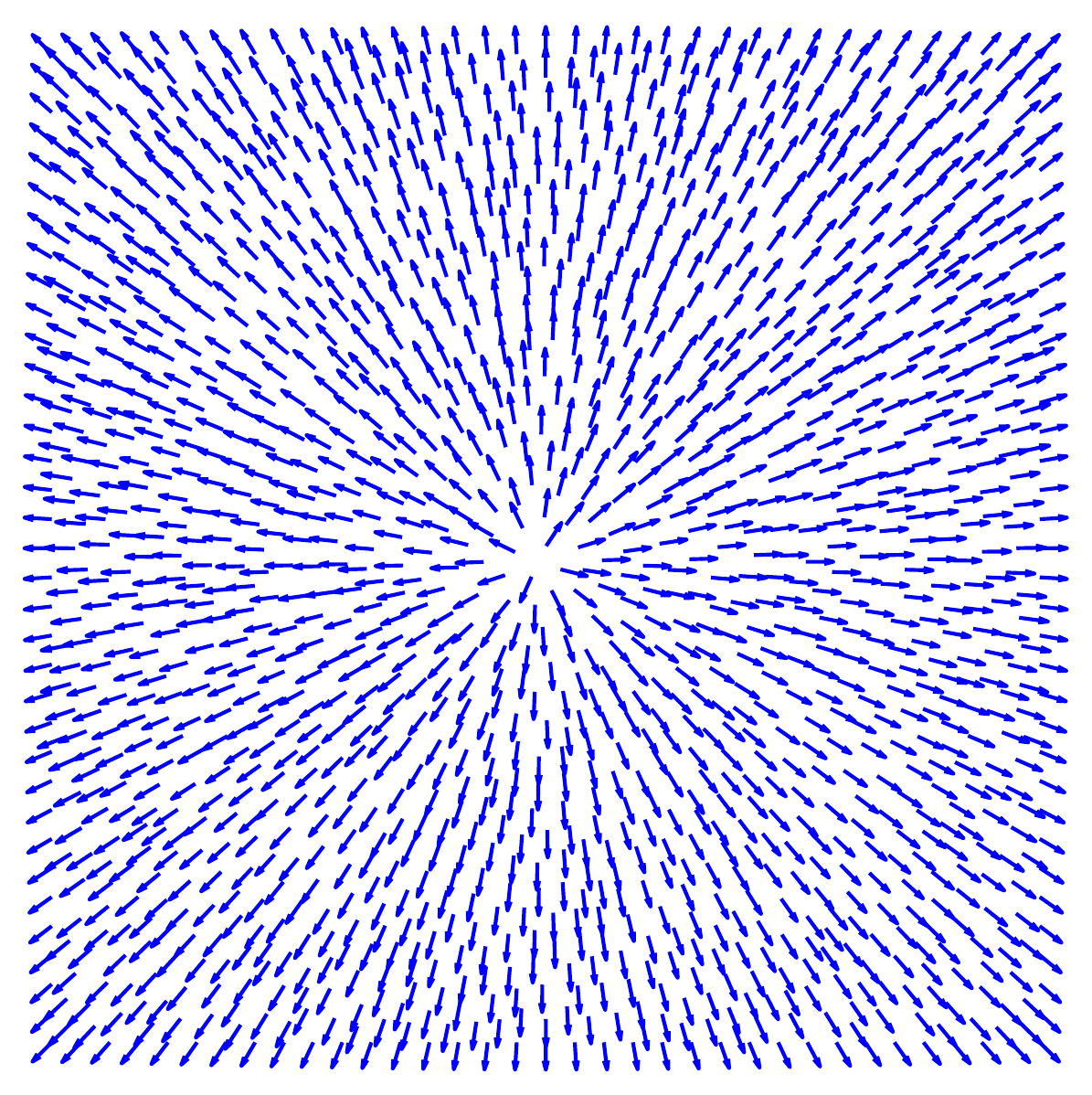} \\ 
	\includegraphics[width=0.21\textwidth]{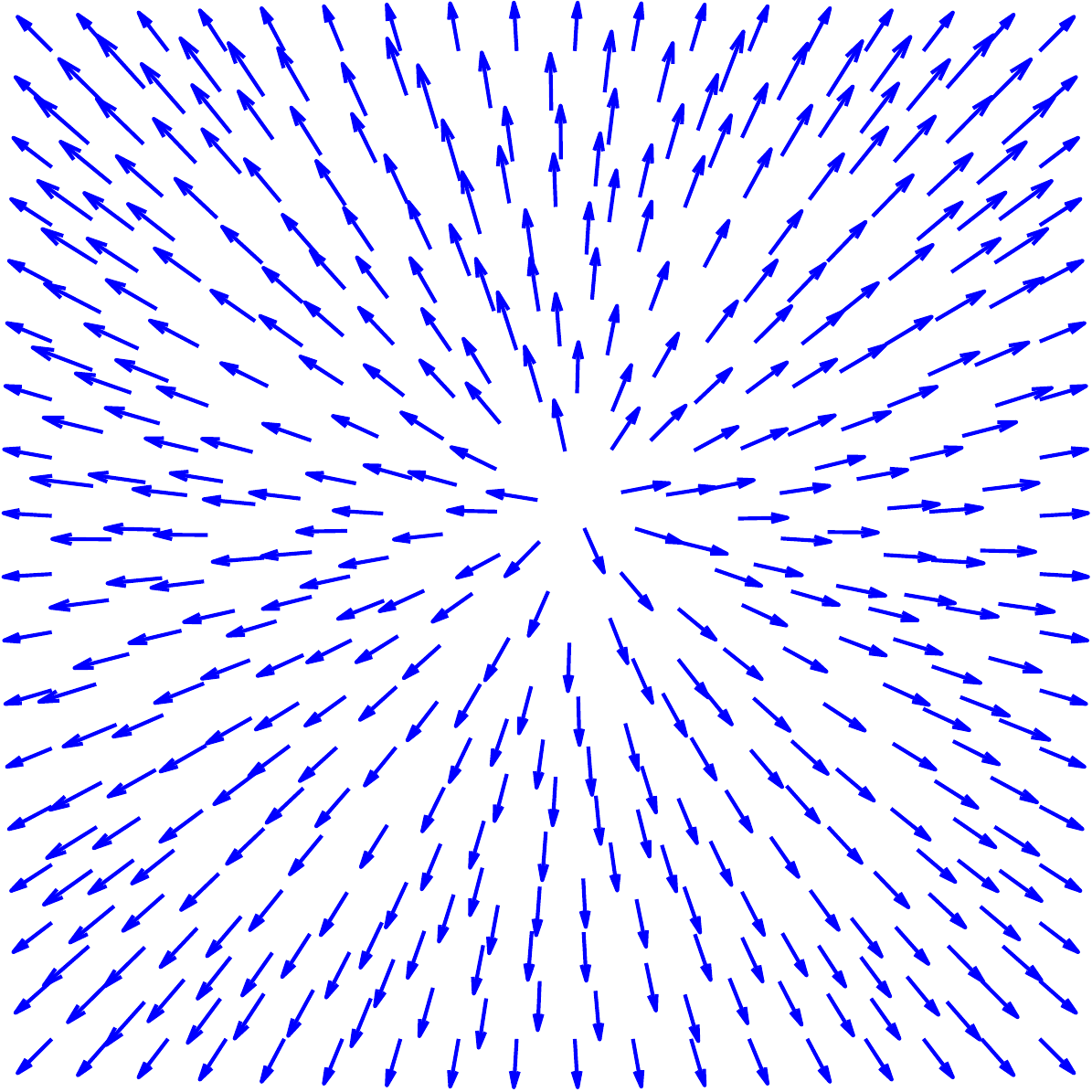}	&
	\includegraphics[width=0.21\textwidth]{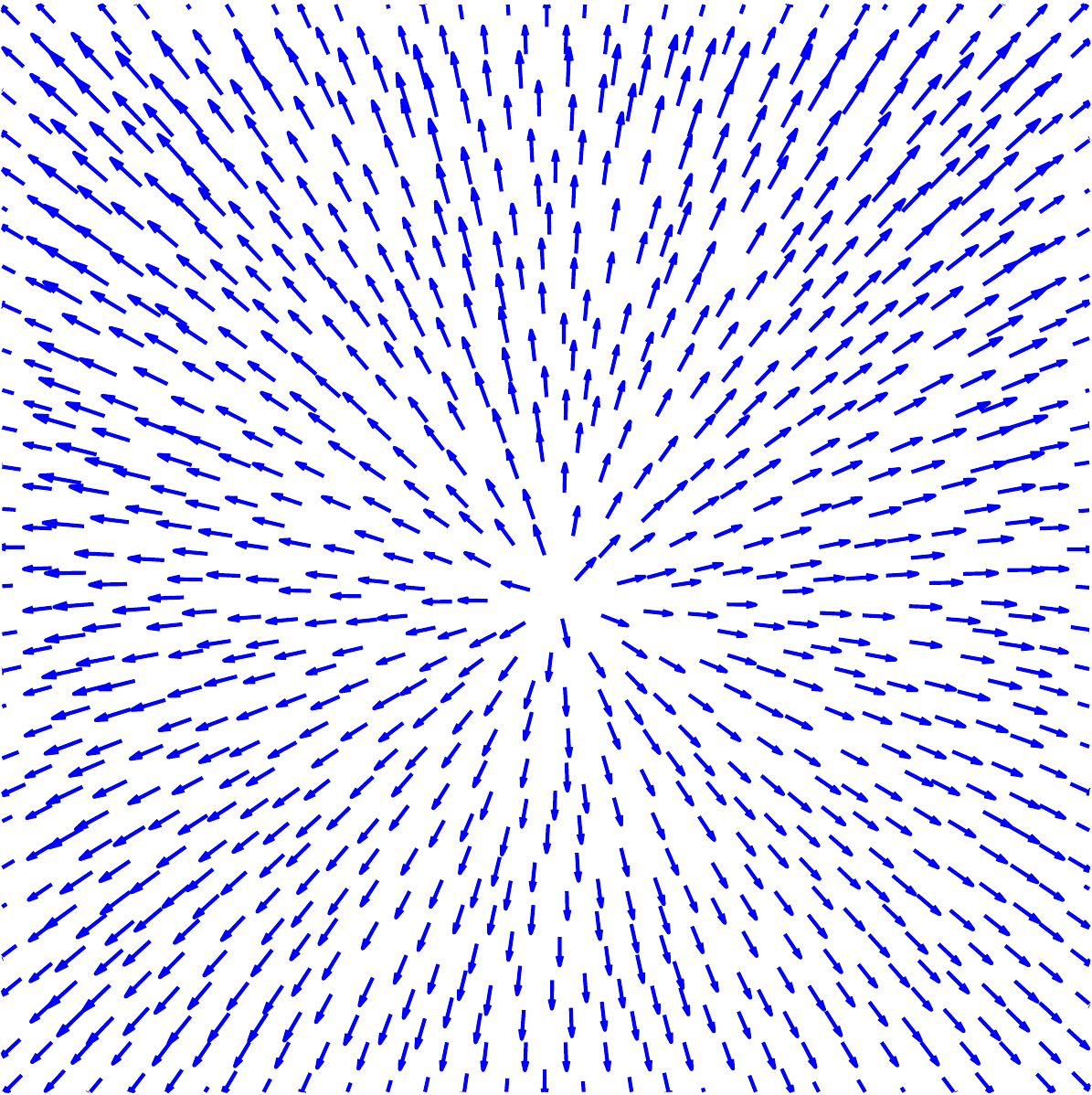}  &
	\includegraphics[width=0.21\textwidth]{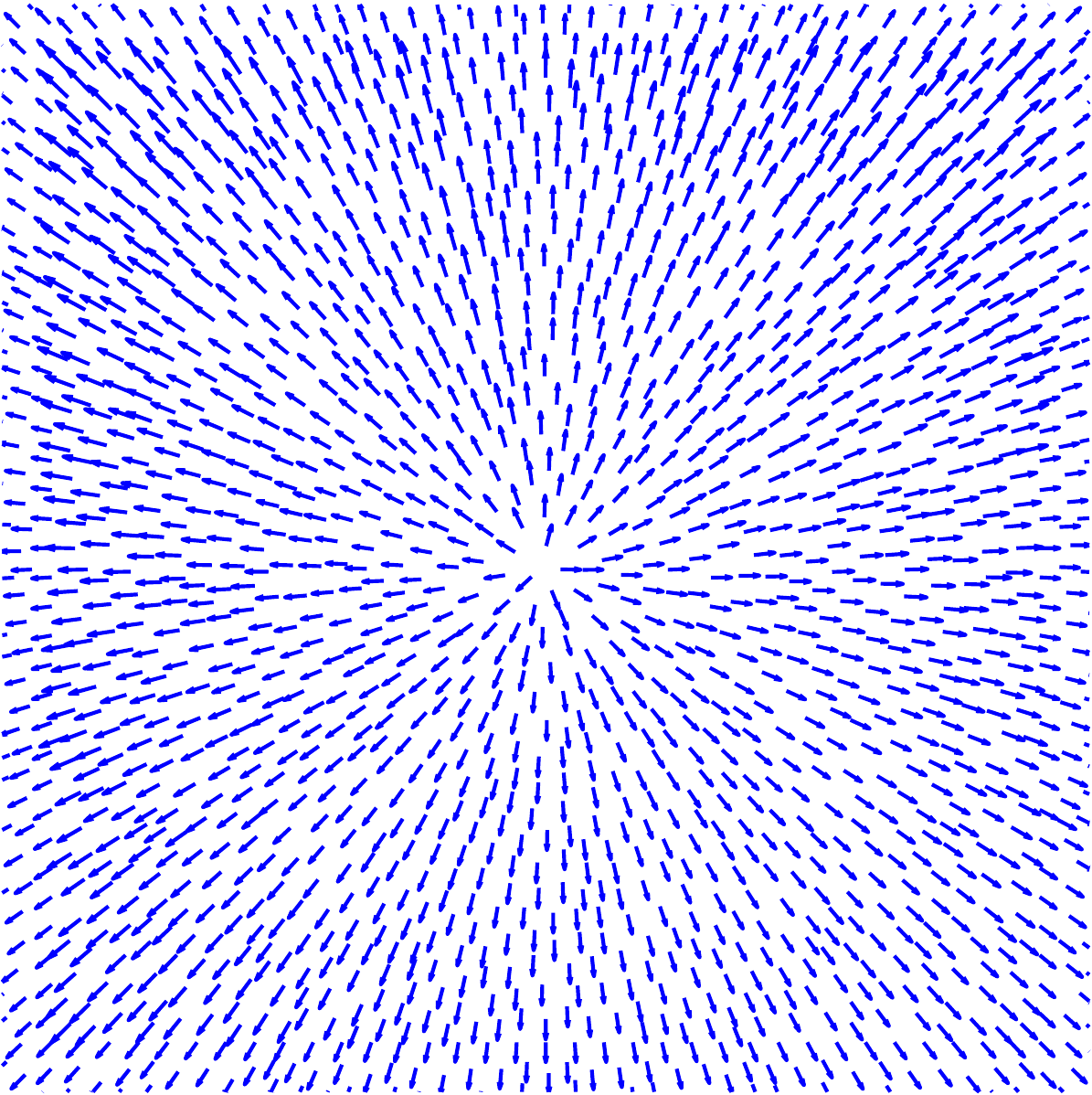} 	\\ 
	\end{tabular}
	\caption{\label{f:u}  ($\O$ square) Approximate discrete fractional harmonic maps $u_h$ in Example~\ref{ex:defect}
          for fractional parameters $s = 0.2$ (top), $s= 0.4$ (middle), $s= 0.6$ (bottom),
          and mesh sizes meshsizes $h = 0.048$ (left), $h= 0.033$ (middle), and $h=0.025$ (right). 
          In all cases a point defect is approximated which is more localized for larger values of 
          $s$ and smaller values of $h$.
	}
\end{figure}

\begin{figure}[p]	
	\begin{tabular}{ccc} \medskip 	
	\includegraphics[width=0.23\textwidth]{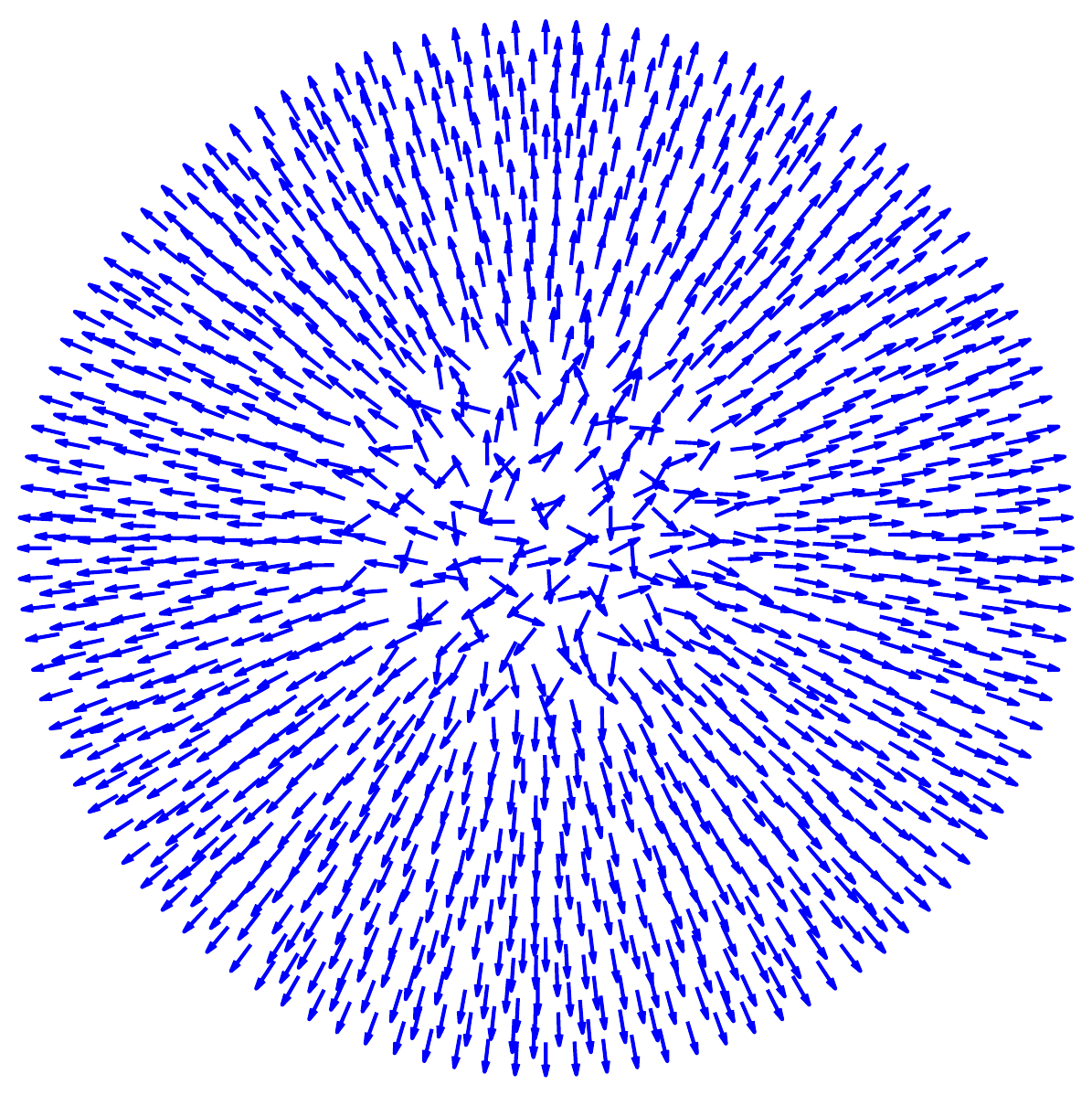} 
	\includegraphics[width=0.23\textwidth]{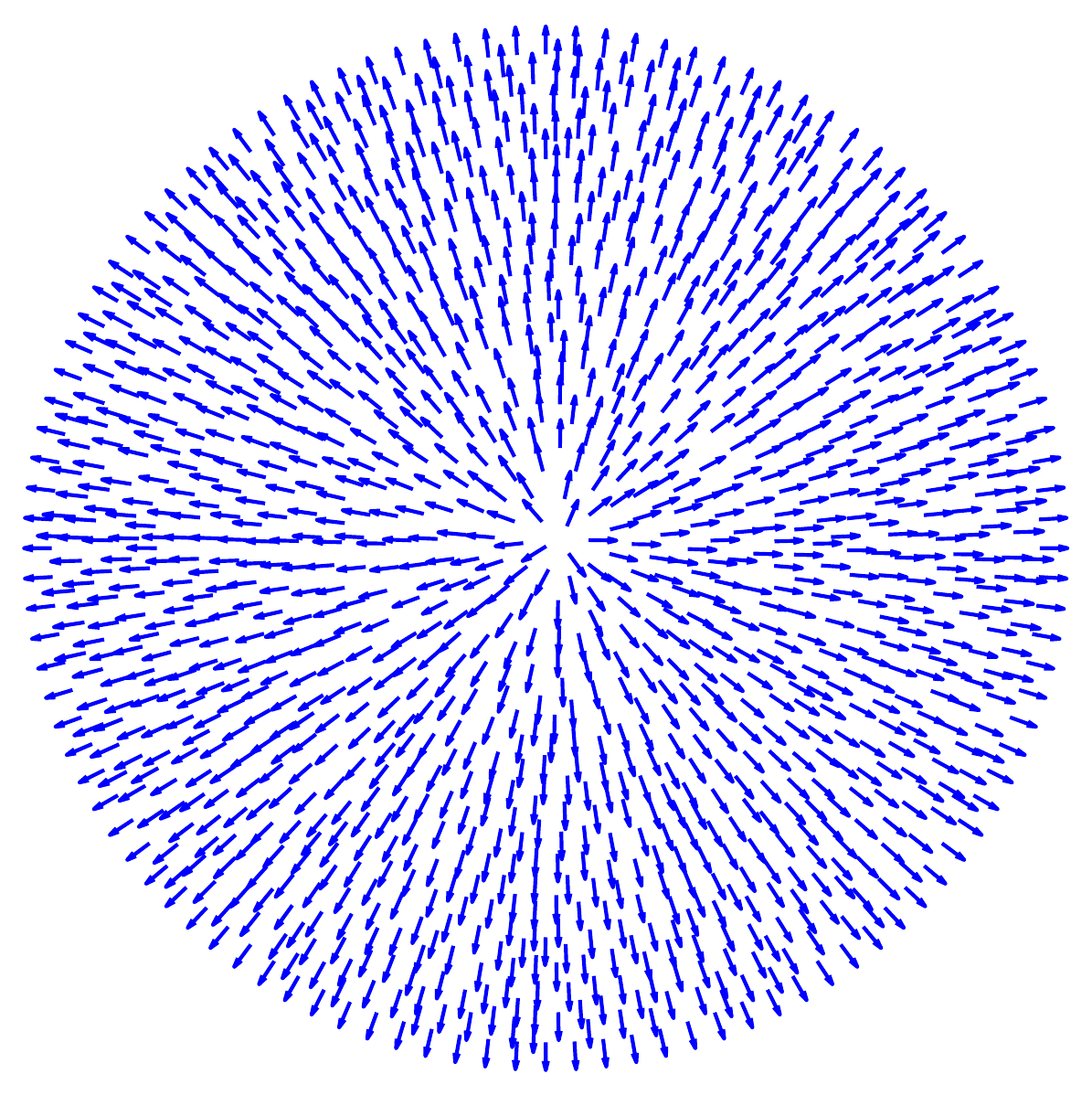} 
	\includegraphics[width=0.23\textwidth]{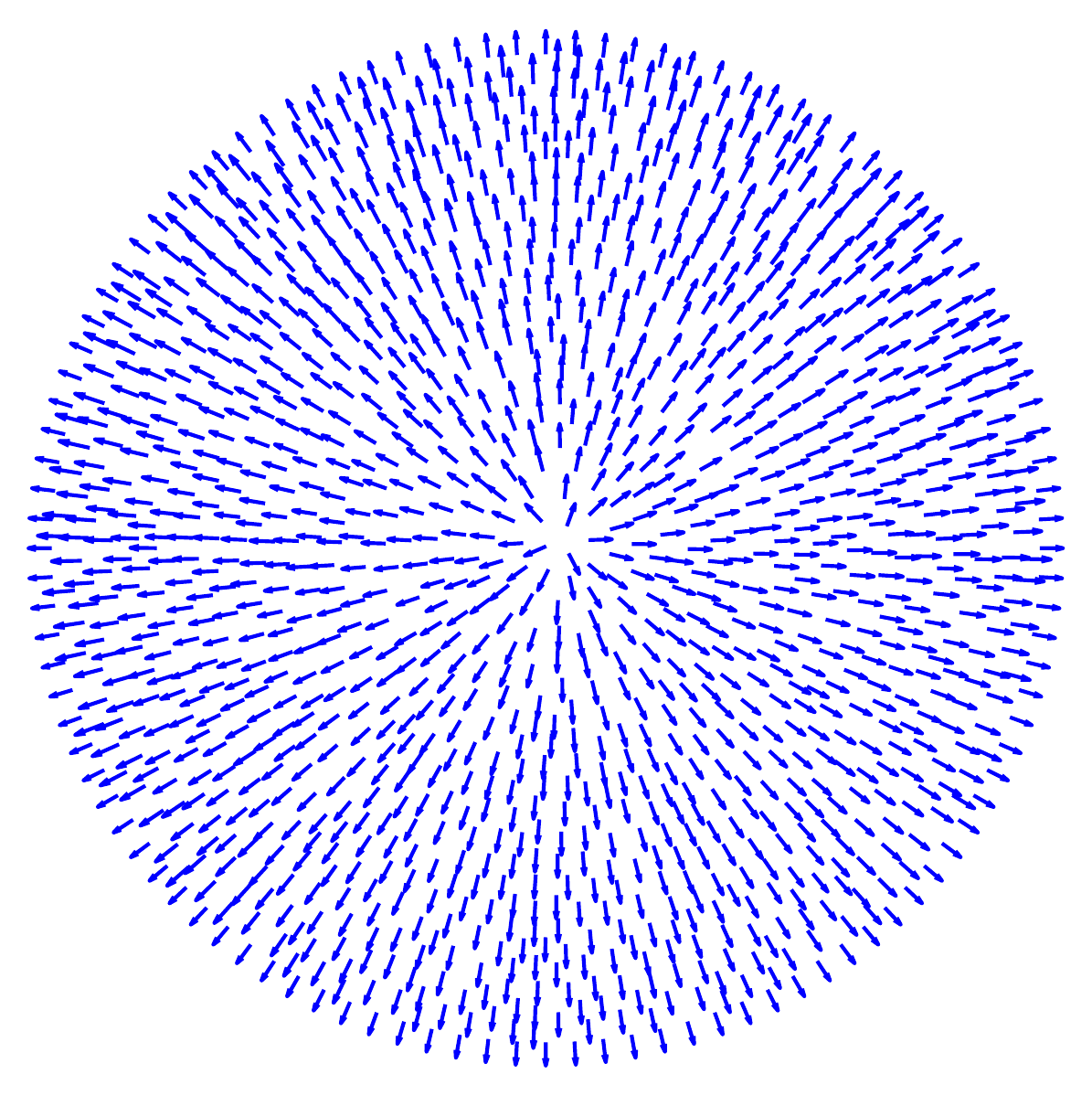} \\ 
	\end{tabular}	
	\caption{\label{f:u_circ}  ($\O$ disc) Approximate discrete fractional harmonic maps $u_h$ in Example~\ref{ex:defect}
          for fractional parameters $s = 0.2$ (left), $s= 0.4$ (middle), $s= 0.6$ (right),
          for a fixed mesh size $h=0.028$. 
          In all cases a point defect is approximated which is more localized for larger values of 
          $s$ and smaller values of $h$.
	}
\end{figure}

\subsection*{Acknowledgments.} The authors are grateful to Enno Lenzmann
for stimulating discussions and valuable hints. 

\appendix


\section{Spectral fractional Laplacian}\label{app:spec} 

In this section, we will prove results analogous to section~\ref{s:compact} but for a different definition of
fractional Laplacian. For any $s\geq0$, consider the fractional order Sobolev space 
\begin{align*}
\mathbb H^s(\Omega)=\big\{u=\sum_{k=1}^\infty u_k\varphi_k\in L^2(\Omega):
\;\;\|u\|_{\mathbb H^s(\Omega)}^2=\sum_{k=1}^\infty \lambda_k^su_k^2<\infty\big\}, 
\quad u_k=\int_{\Omega}u\varphi_k\dv{x}.
\end{align*} 
where $\{\lambda_k\}_{k\in \N}$ and $\{\varphi_k\}_{k\in \N}$ are the eigenvalues and
corresponding normalized eigenfunctions of the standard 
Laplacian for homogeneous Dirichlet boundary conditions.

The spectral fractional Dirichlet Laplacian is defined on the space $\mathbb H^s(\Omega)$ by
\begin{align*}
(-\Delta_{\Omega})^su=\sum_{k=1}^\infty\lambda_k^su_k\varphi_k \, ,\qquad u_k=\int_{\Omega}u\varphi_k\dv{x}.
\end{align*}
We have that $\|u\|_{\mathbb H^s(\Omega)}=\|\rlaps{s}u\|_{L^2(\Omega)}$. An integral representation 
of the operator $(-\Delta_{\Omega})^s$ from \cite[Eq.~(1.3)]{LACaffarelli_PRStinga_2016a} states that
for almost every $x\in\Omega$ we have 
\begin{align}\label{int-rep}
(-\Delta_{\Omega})^su(x)=\mbox{P.V.}\int_{\Omega}\left[u(x)-u(y)\right]J(x,y)\dv{y} +\kappa(x)u(x) .
\end{align}
Letting $K_\Omega(t,x,y)$ denote the heat kernel of the semigroup generated by standard Laplace
operator on $L^2(\Omega)$ and $\G$ be the usual Gamma function we have  
\begin{align*}
J(x,y)& =\frac{s}{\Gamma(1-s)}\int_0^\infty\frac{K_\Omega(t,x,y)}{t^{1+s}} \dv{t}, \\
\kappa(x)& =\frac{s}{\Gamma(1-s)}\int_0^\infty\Big(1-\int_{\Omega}K_\Omega(t,x,y)\dv{y}\Big)\frac{\dv{t}}{t^{1+s}} .
\end{align*}
From the properties of $K_\Omega$ we have that $J$ is symmetric and nonnegative and
that $\kappa$ is nonnegative. Moreover, we have the 
estimate \cite[Theorem 2.3]{LACaffarelli_PRStinga_2016a} 
\begin{equation}\label{eq:Jest} 
 0 \leq J(x,y) \aleq |x-y|^{-d-s}.
\end{equation}
As in section~\ref{s:compact} we define for $f,g \in \mathbb H^s(\Omega)$
\[
H_{s,\Omega}(f,g) = \rlaps{s} (fg) -f \rlaps{s} g - (\rlaps{s} f) g  
\]
We note that the operator extends to general bilinear operations on vector
fields $f$ and $g$ in a canonical way. Analogously to~\Cref{pr:strongconv} we 
have the following weak continuity property of the operator $H_{s,\O}$.

\begin{proposition}\label{pr:regstrongconv}
Let $\Omega \subset \R^d$ be bounded Lipschitz, $s \in (0,2)$ and $s < d$.
Assume that 
\[
 \sup_{j} \|f_j\|_{L^\infty(\Omega)} + \|f_j\|_{\mathbb H^s(\Omega)} < \infty
\]
and 
\[
\lim_{j \to \infty} \|f_j -f\|_{L^2(\Omega)} = 0.
\] 
Given any $\varphi \in C_c^\infty(\Omega)$ we have
\[
 \|H_{s,\Omega}(f_j,\varphi) - H_{s,\Omega}(f,\varphi)\|_{L^2(\Omega)} \to 0
\]
as $j\to \infty$.
\end{proposition}

\begin{proof}
We abbreviate $g_j = f_j-f$ and note that using the integral representation of the operator
$(-\Delta_\O)^{\frac{s}{2}}$ we have  
\[
 H_{s,\Omega}(g_j,\varphi) = \int_{\Omega} J(x,y) \brac{g_j(x)-g_j(y)}\brac{\varphi(x)-\varphi(y)} \dv{y} \dv{x}.
\]
Since $\varphi \in C_c^\infty(\Omega)$ is fixed, 
we can choose $\eps > 0$ and $\eta \in C_c^\infty(\Omega)$ such that $\eta(x) \equiv 1$ 
whenever $\dist(x,\supp\varphi) \leq \eps$ and $\eta \equiv 0$ whenever $\dist(x,\partial \Omega) \leq \eps$. 
Let $K = \{x \in \O: \eta(x) =1\}$.

Set $\tilde{g}_j = \eta g_j$. We then have by \Cref{la:extensionest} below that 
\begin{equation}\label{eq:tgjest}
 \|\laps{s} \tilde{g}_j \|_{L^2(\R^d)} + \|\tilde{g}_j \|_{L^2(\R^d)} \leq C(\eps) \|g_j\|_{\mathbb{H}^s(\Omega)}.
\end{equation}
We split integrals using the partition of $\O$ into $\O\setminus K$ and $K$ to obtain the
estimate 
\[
\|H_{s,\Omega}(g_j,\varphi)\|_{L^2(\Omega)}^2 \aleq I + II + III + IV,
\]
where 
\[\begin{split}
 I &= \int_{\Omega \backslash K} \abs{\int_{\Omega \backslash K} J(x,y) \brac{g_j(x)-g_j(y)}\brac{\varphi(x)-\varphi(y)} dy}^2 \dv{x}, \\
 II &= \int_{K} \abs{\int_{\Omega \backslash K} J(x,y) \brac{\tilde{g}_j(x)-\tilde{g}_j(y)}\varphi(x) dy}^2 \dv{x}, \\
 III &= \int_{\Omega \backslash K} \abs{\int_{K} J(x,y) \brac{\tilde{g}_j(x)-\tilde{g}_j(y)}\varphi(y) dy}^2 \dv{x}, \\
 IV &= \int_{K} \abs{\int_{K} J(x,y) \brac{\tilde{g}_j(x)-\tilde{g}_j(y)}\brac{\varphi(x)-\varphi(y)} dy}^2 \dv{x} .
\end{split}\]
We show that the terms $I,II,III,IV$ converge to zero as $j \to \infty$ to deduce the asserted
result. \\
\emph{Estimate for I.} Recalling that $\supp \varphi \subset K$ we find that $I=0$. \\
\emph{Estimate of II and III.} We observe that if $\varphi(x) = 0$ then 
$\dist(x,\Omega \backslash K) \geq \eps$. Thus if $x \in \Omega \backslash K$ and $y \in \supp \varphi$ (or $y \in \Omega \backslash K$ and $x \in \supp \varphi$) then $|x-y| \ageq \eps$ and thus by \eqref{eq:Jest} and thus
\[
 J(x,y) \leq C(\Omega,\eps,s)\, \min\{1+|x|,1+|y|\}^{-d-s}.
\]
We then argue exactly as in the proof of \Cref{pr:strongconv} to obtain
\[
 II + III \aleq \|\varphi\|_{L^\infty(\R^d)}^2 \|g_j\|_{L^2(\Omega)}^2.
\]
\emph{Estimate of IV.} As in the proof of \Cref{pr:strongconv} we obtain for some $t < s$ that 
\[
 IV \aleq \|\nabla \varphi\|_{L^\infty(\R^d)}^2\, \|\laps{t} \tilde{g}_j\|_{L^2(\R^d)}^2.
\]
Combining the estimates for $I,II,III,IV$ we obtain
\[
  \|H_{s,\Omega}(g_j,\varphi)\|_{L^2(\Omega)} \aleq \brac{\|\varphi\|_{L^\infty(\R^d)} + \|\nabla \varphi\|_{L^\infty(\R^d)}}\, \brac{\|g_j\|_{L^2(\Omega)} + \|\laps{t} \tilde{g}_j\|_{L^2(\R^d)}}
\]
By assumption we have $\|g_j\|_{L^2(\Omega)} \to 0$ as $j\to \infty$ 
and in view of \eqref{eq:tgjest} and \Cref{pr:RellichweakconvOmega} that
\[
 \|\laps{t} \tilde{g}_j\|_{L^2(\R^d)} \to 0.
\]
This concludes the proof.
\end{proof}

The following auxiliary estimate is needed in the proof of~\Cref{pr:regstrongconv}.

\begin{lemma}\label{la:extensionest}
Let $\Omega \subset \R^d$ be a bounded set and $\eta \in C_c^\infty(\Omega)$. 
Then for any $g \in \mathbb{H}^s(\Omega)$ we have $\eta g \in H^s(\R^d)$ with the estimate
\[
 \|\laps{s} (\eta g) \|_{L^2(\R^d)} + \|\eta g \|_{L^2(\R^d)} \leq C(\eta) \|g\|_{\mathbb{H}^s(\Omega)}.
\]
\end{lemma}

\begin{proof}
The estimate follows by an interpolation argument of the 
mapping $T_\eta g = \eta g$ as a bounded linear operator $L^2(\O)\to L(\R^d)$ and 
$W^{1,2}(\Omega) \to W^{1,2}(\R^d)$ in combination with the fact that the spaces
$\mathbb{H}^s(\Omega)$ and $W^{s,2}(\R^d)$ are equivalently obtained via interpolation. 
\end{proof}

\bibliographystyle{abbrv}
\bibliography{bib_frac_hm_num}

\end{document}

%% file: sb_macros.tex
\def\R{\mathbb{R}}
\def\C{\mathbb{C}}
\def\N{\mathbb{N}}

\def\T{\mathbb{T}}

\def\cA{\mathcal{A}}
\def\cB{\mathcal{B}}

\def\cF{\mathcal{F}}

\def\cI{\mathcal{I}}

\def\cN{\mathcal{N}}
\def\cO{\mathcal{O}}

\def\cS{\mathcal{S}}
\def\cT{\mathcal{T}}

\def\b{\beta}

\def\d{\delta}

\def\l{\lambda}

\def\p{\partial}

\def\veps{\varepsilon}

\def\vphi{\varphi}
\def\O{\Omega}

\def\G{\Gamma}

\def\wto{\rightharpoonup}

\def\transp{{\sf T}}

\def\ou{\overline{u}}
\def\hu{\widehat{u}}
\def\tu{\widetilde{u}}

\newcommand{\dv}[1]{\,{\mathrm d}#1}

\newcommand{\wcheck}[1]{#1\hspace{-.8ex}\mbox{\huge {\lower.45ex \hbox{$\textstyle \check{}$}}} \hspace{.5ex}}

\let\oldmarginpar\marginpar
\renewcommand\marginpar[1]{
  \oldmarginpar[\raggedleft\footnotesize #1]
  {\raggedright\footnotesize #1}}
\newtheorem{definition}{Definition}
\newtheorem{lemma}[definition]{Lemma}
\newtheorem{proposition}[definition]{Proposition}
\newtheorem{theorem}[definition]{Theorem}
\newtheorem{corollary}[definition]{Corollary}
\newtheorem{remark}[definition]{Remark}

\newtheorem{example}[definition]{Example}

\newtheorem{algorithm}[definition]{Algorithm}

\numberwithin{definition}{section}

%% file: frac_hm_arxv.bbl
\begin{thebibliography}{10}

\bibitem{MR3967804}
N.~Abatangelo and E.~Valdinoci.
\newblock Getting acquainted with the fractional {L}aplacian.
\newblock In {\em Contemporary research in elliptic {PDE}s and related topics},
  volume~33 of {\em Springer INdAM Ser.}, pages 1--105. Springer, Cham, 2019.

\bibitem{GAcosta_FMBersetche_JPBorthagaray_2017a}
G.~Acosta, F.~M. Bersetche, and J.~P. Borthagaray.
\newblock A short {FE} implementation for a 2d homogeneous {D}irichlet problem
  of a fractional {L}aplacian.
\newblock {\em Comput. Math. Appl.}, 74(4):784--816, 2017.

\bibitem{GAcosta_JPBorthagaray_NHeuer_2019a}
G.~Acosta, J.~P. Borthagaray, and N.~Heuer.
\newblock Finite element approximations of the nonhomogeneous fractional
  {D}irichlet problem.
\newblock {\em IMA J. Numer. Anal.}, 39(3):1471--1501, 2019.

\bibitem{AH96}
D.~R. Adams and L.~I. Hedberg.
\newblock {\em Function spaces and potential theory}, volume 314 of {\em
  Grundlehren der Mathematischen Wissenschaften [Fundamental Principles of
  Mathematical Sciences]}.
\newblock Springer-Verlag, Berlin, 1996.

\bibitem{Alou97}
F.~Alouges.
\newblock A new algorithm for computing liquid crystal stable configurations:
  the harmonic mapping case.
\newblock {\em SIAM J. Numer. Anal.}, 34(5):1708--1726, 1997.

\bibitem{Alou08}
F.~Alouges.
\newblock A new finite element scheme for {L}andau-{L}ifchitz equations.
\newblock {\em Discrete Contin. Dyn. Syst. Ser. S}, 1(2):187--196, 2008.

\bibitem{AntBar17}
H.~Antil and S.~Bartels.
\newblock Spectral approximation of fractional {PDE}s in image processing and
  phase field modeling.
\newblock {\em Comput. Methods Appl. Math.}, 17(4):661--678, 2017.

\bibitem{HAntil_PDondl_LStriet_2020a}
H.~Antil, P.~Dondl, and L.~Striet.
\newblock Approximation of integral fractional laplacian and fractional pdes
  via sinc-basis.
\newblock {\em arXiv preprint arXiv:2010.06509}, 2020.

\bibitem{HAntil_RKhatri_MWarma_2019a}
H.~Antil, R.~Khatri, and M.~Warma.
\newblock External optimal control of nonlocal {PDE}s.
\newblock {\em Inverse Problems}, 35(8):084003, 35, 2019.

\bibitem{ARS20}
H.~Antil, C.~N. Rautenberg, and A.~Schikorra.
\newblock On a fractional version of a murat compactness result and
  applications.
\newblock {\em To appear in SIAM J. of Math. Anal.}, 2021.

\bibitem{HAntil_DVerma_MWarma_2020a}
H.~Antil, D.~Verma, and M.~Warma.
\newblock External optimal control of fractional parabolic {PDE}s.
\newblock {\em ESAIM Control Optim. Calc. Var.}, 26, 2020.

\bibitem{HAntil_MWarma_2020a}
H.~Antil and M.~Warma.
\newblock Optimal control of fractional semilinear {PDE}s.
\newblock {\em ESAIM Control Optim. Calc. Var.}, 26:Paper No. 5, 30, 2020.

\bibitem{BBFP07}
J.~W. Barrett, S.~Bartels, X.~Feng, and A.~Prohl.
\newblock A convergent and constraint-preserving finite element method for the
  {$p$}-harmonic flow into spheres.
\newblock {\em SIAM J. Numer. Anal.}, 45(3):905--927, 2007.

\bibitem{Bart15}
S.~Bartels.
\newblock Fast and accurate finite element approximation of wave maps into
  spheres.
\newblock {\em ESAIM Math. Model. Numer. Anal.}, 49(2):551--558, 2015.

\bibitem{Bart15-book}
S.~Bartels.
\newblock {\em Numerical methods for nonlinear partial differential equations},
  volume~47 of {\em Springer Series in Computational Mathematics}.
\newblock Springer, Cham, 2015.

\bibitem{Bart16}
S.~Bartels.
\newblock Projection-free approximation of geometrically constrained partial
  differential equations.
\newblock {\em Math. Comp.}, 85(299):1033--1049, 2016.

\bibitem{BaFePr07}
S.~Bartels, X.~Feng, and A.~Prohl.
\newblock Finite element approximations of wave maps into spheres.
\newblock {\em SIAM J. Numer. Anal.}, 46(1):61--87, 2007/08.

\bibitem{BarPro06}
S.~Bartels and A.~Prohl.
\newblock Convergence of an implicit finite element method for the
  {L}andau-{L}ifshitz-{G}ilbert equation.
\newblock {\em SIAM J. Numer. Anal.}, 44(4):1405--1419, 2006.

\bibitem{BH93}
B.~Bojarski and P.~Hajlasz.
\newblock Pointwise inequalities for {S}obolev functions and some applications.
\newblock {\em Studia Math.}, 106(1):77--92, 1993.

\bibitem{bonito2019numerical}
A.~Bonito, W.~Lei, and J.~E. Pasciak.
\newblock Numerical approximation of the integral fractional laplacian.
\newblock {\em Numerische Mathematik}, 142(2):235--278, 2019.

\bibitem{JPBorthagaray_PCJr_2019a}
J.~P. Borthagaray and P.~Ciarlet, Jr.
\newblock On the convergence in {$H^1$}-norm for the fractional {L}aplacian.
\newblock {\em SIAM J. Numer. Anal.}, 57(4):1723--1743, 2019.

\bibitem{borthagaray2020local}
J.~P. Borthagaray, D.~Leykekhman, and R.~H. Nochetto.
\newblock Local energy estimates for the fractional laplacian.
\newblock {\em arXiv preprint arXiv:2005.03786}, 2020.

\bibitem{BoNoWa20}
J.~P. Borthagaray, R.~H. Nochetto, and S.~W. Walker.
\newblock A structure-preserving {FEM} for the uniaxially constrained
  {Q}-tensor model of nematic liquid crystals.
\newblock {\em Numer. Math.}, 145(4):837--881, 2020.

\bibitem{BLSS20}
L.~Bugiera, E.~Lenzmann, A.~Schikorra, and J.~Sok.
\newblock On symmetry of traveling solitary waves for dispersion generalized
  {NLS}.
\newblock {\em Nonlinearity}, 33(6):2797--2819, 2020.

\bibitem{MR2354493}
L.~Caffarelli and L.~Silvestre.
\newblock An extension problem related to the fractional {L}aplacian.
\newblock {\em Comm. Partial Differential Equations}, 32(7-9):1245--1260, 2007.

\bibitem{LACaffarelli_PRStinga_2016a}
L.~Caffarelli and P.~Stinga.
\newblock Fractional elliptic equations, {C}accioppoli estimates and
  regularity.
\newblock {\em Ann. Inst. H. Poincar\'e Anal. Non Lin\'eaire}, 33(3):767--807,
  2016.

\bibitem{DLR2011}
F.~Da~Lio and T.~Rivi\`ere.
\newblock Three-term commutator estimates and the regularity of
  {$\frac12$}-harmonic maps into spheres.
\newblock {\em Anal. PDE}, 4(1):149--190, 2011.

\bibitem{DAP19}
P.~D'Ancona.
\newblock A short proof of commutator estimates.
\newblock {\em J. Fourier Anal. Appl.}, 25(3):1134--1146, 2019.

\bibitem{MR2944369}
E.~Di~Nezza, G.~Palatucci, and E.~Valdinoci.
\newblock Hitchhiker's guide to the fractional {S}obolev spaces.
\newblock {\em Bull. Sci. Math.}, 136(5):521--573, 2012.

\bibitem{SDipierro_XRosOton_EValdinoci_2017a}
S.~Dipierro, X.~Ros-Oton, and E.~Valdinoci.
\newblock Nonlocal problems with {N}eumann boundary conditions.
\newblock {\em Rev. Mat. Iberoam.}, 33(2):377--416, 2017.

\bibitem{MR3310082}
A.~Fiscella, R.~Servadei, and E.~Valdinoci.
\newblock Density properties for fractional {S}obolev spaces.
\newblock {\em Ann. Acad. Sci. Fenn. Math.}, 40(1):235--253, 2015.

\bibitem{GerLen18}
P.~G\'{e}rard and E.~Lenzmann.
\newblock A {L}ax pair structure for the half-wave maps equation.
\newblock {\em Lett. Math. Phys.}, 108(7):1635--1648, 2018.

\bibitem{geuzaine2009gmsh}
C.~Geuzaine and J.-F. Remacle.
\newblock Gmsh: A 3-d finite element mesh generator with built-in pre-and
  post-processing facilities.
\newblock {\em International journal for numerical methods in engineering},
  79(11):1309--1331, 2009.

\bibitem{H96}
P.~Hajlasz.
\newblock Sobolev spaces on an arbitrary metric space.
\newblock {\em Potential Anal.}, 5(4):403--415, 1996.

\bibitem{HPPRSS19}
G.~Hrkac, C.-M. Pfeiler, D.~Praetorius, M.~Ruggeri, A.~Segatti, and
  B.~Stiftner.
\newblock Convergent tangent plane integrators for the simulation of chiral
  magnetic skyrmion dynamics.
\newblock {\em Adv. Comput. Math.}, 45(3):1329--1368, 2019.

\bibitem{KarWeb14}
T.~K. Karper and F.~Weber.
\newblock A new angular momentum method for computing wave maps into spheres.
\newblock {\em SIAM J. Numer. Anal.}, 52(4):2073--2091, 2014.

\bibitem{KPPRS19}
J.~Kraus, C.-M. Pfeiler, D.~Praetorius, M.~Ruggeri, and B.~Stiftner.
\newblock Iterative solution and preconditioning for the tangent plane scheme
  in computational micromagnetics.
\newblock {\em J. Comput. Phys.}, 398:108866, 27, 2019.

\bibitem{LenSch18}
E.~Lenzmann and A.~Schikorra.
\newblock On energy-critical half-wave maps into {$\Bbb {S}^2$}.
\newblock {\em Invent. Math.}, 213(1):1--82, 2018.

\bibitem{Lin87}
F.-H. Lin.
\newblock A remark on the map {$x/|x|$}.
\newblock {\em C. R. Acad. Sci. Paris S\'{e}r. I Math.}, 305(12):529--531,
  1987.

\bibitem{MS18}
K.~Mazowiecka and A.~Schikorra.
\newblock Fractional div-curl quantities and applications to nonlocal geometric
  equations.
\newblock {\em J. Funct. Anal.}, 275(1):1--44, 2018.

\bibitem{MP20}
V.~Millot and M.~Pegon.
\newblock Minimizing 1/2-harmonic maps into spheres.
\newblock {\em Calc. Var. Partial Differential Equations}, 59(2):Paper No. 55,
  37, 2020.

\bibitem{MillotPegonSchikorra20}
V.~Millot, M.~Pegon, and A.~Schikorra.
\newblock Partial regularity for fractional harmonic maps into spheres, 2020.

\bibitem{MillotSire15}
V.~Millot and Y.~Sire.
\newblock On a fractional {G}inzburg-{L}andau equation and 1/2-harmonic maps
  into spheres.
\newblock {\em Arch. Ration. Mech. Anal.}, 215(1):125--210, 2015.

\bibitem{M11}
R.~Moser.
\newblock Intrinsic semiharmonic maps.
\newblock {\em J. Geom. Anal.}, 21(3):588--598, 2011.

\bibitem{nochetto2017finite}
R.~H. Nochetto, S.~W. Walker, and W.~Zhang.
\newblock A finite element method for nematic liquid crystals with variable
  degree of orientation.
\newblock {\em SIAM Journal on Numerical Analysis}, 55(3):1357--1386, 2017.

\bibitem{PG11}
X.~Pu and B.~Guo.
\newblock The fractional {L}andau-{L}ifshitz-{G}ilbert equation and the heat
  flow of harmonic maps.
\newblock {\em Calc. Var. Partial Differential Equations}, 42(1-2):1--19, 2011.

\bibitem{R18}
J.~Roberts.
\newblock A regularity theory for intrinsic minimising fractional harmonic
  maps.
\newblock {\em Calc. Var. Partial Differential Equations}, 57(4):Paper No. 109,
  68, 2018.

\bibitem{S11}
A.~{Schikorra}.
\newblock {Interior and Boundary-Regularity for Fractional Harmonic Maps on
  Domains}.
\newblock {\em arxiv, unpublished}, page arXiv:1103.5203, Mar 2011.

\bibitem{S15}
A.~Schikorra.
\newblock {$\varepsilon$}-regularity for systems involving non-local,
  antisymmetric operators.
\newblock {\em Calc. Var. Partial Differential Equations}, 54(4):3531--3570,
  2015.

\bibitem{S18}
A.~Schikorra.
\newblock Boundary equations and regularity theory for geometric variational
  systems with {N}eumann data.
\newblock {\em Arch. Ration. Mech. Anal.}, 229(2):709--788, 2018.

\bibitem{SSW17}
A.~Schikorra, Y.~Sire, and C.~Wang.
\newblock Weak solutions of geometric flows associated to integro-differential
  harmonic maps.
\newblock {\em Manuscripta Math.}, 153(3-4):389--402, 2017.

\bibitem{Stein93}
E.~M. Stein.
\newblock {\em Harmonic analysis: real-variable methods, orthogonality, and
  oscillatory integrals}, volume~43 of {\em Princeton Mathematical Series}.
\newblock Princeton University Press, Princeton, NJ, 1993.
\newblock With the assistance of Timothy S. Murphy, Monographs in Harmonic
  Analysis, III.

\bibitem{S88}
M.~Struwe.
\newblock On the evolution of harmonic maps in higher dimensions.
\newblock {\em J. Differential Geom.}, 28(3):485--502, 1988.

\bibitem{CJWeiss_BGVBWaanders_HAntil_2020a}
C.~J. Weiss, B.~G. van Bloemen~Waanders, and H.~Antil.
\newblock Fractional operators applied to geophysical electromagnetics.
\newblock {\em Geophysical Journal International}, 220(2):1242--1259, 2020.

\bibitem{ZhoSto15}
T.~Zhou and M.~Stone.
\newblock Solitons in a continuous classical {H}aldane-{S}hastry spin chain.
\newblock {\em Phys. Lett. A}, 379(43-44):2817--2825, 2015.

\end{thebibliography}
